\documentclass[12pt]{amsart}
\usepackage{amsfonts}
\usepackage{mathrsfs}
\usepackage{amscd,amsmath,amssymb,amsfonts}
\usepackage{color}
\usepackage[all]{xy}
\theoremstyle{plain}
\newtheorem{thm}{Theorem}
\newtheorem{lem}[thm]{Lemma}
\newtheorem{cor}[thm]{Corollary}
\newtheorem{prop}[thm]{Proposition}
\newtheorem{conj}[thm]{Conjecture}
\theoremstyle{definition}
\newtheorem{defn}[thm]{Definition}

\newtheorem{question}[thm]{Question}
\newtheorem{rmk}[thm]{Remark}
\newtheorem{rmks}[thm]{Remarks}

\newtheorem{nota}[thm]{Notation}

\numberwithin{thm}{section} \numberwithin{equation}{section}

\newcommand{\ga}[2]{\begin{gather}\label{#1}#2 \end{gather}}

\newcommand{\sD}{{\mathcal D}}
\newcommand{\sE}{{\mathcal E}}
\newcommand{\sF}{{\mathcal F}}
\newcommand{\sG}{{\mathcal G}}
\newcommand{\sH}{{\mathcal H}}

\newcommand{\sK}{{\mathcal K}}
\newcommand{\sL}{{\mathcal L}}
\newcommand{\sM}{{\mathcal M}}
\newcommand{\sN}{{\mathcal N}}
\newcommand{\sO}{{\mathcal O}}
\newcommand{\sP}{{\mathcal P}}
\newcommand{\sQ}{{\mathcal Q}}
\newcommand{\sR}{{\mathcal R}}

\newcommand{\sU}{{\mathcal U}}
\newcommand{\sV}{{\mathcal V}}
\newcommand{\sW}{{\mathcal W}}
\newcommand{\sX}{{\mathcal X}}
\newcommand{\sY}{{\mathcal Y}}

\def\wt#1{\widetilde {#1}}

\begin{document}

\title{factorization of generalized Theta functions revisited}
\author{Xiaotao Sun}
\address{Center of Applied Mathematics, Tianjin University, No.92 Weijin Road, Tianjin 300072, P. R. China}
\email{xiaotaosun@tju.edu.cn}
\date{October 2, 2016}
\thanks{Supported in part by the National Natural Science Foundation of China No.11321101}
\begin{abstract}
This survey is based on my lectures given in last a few years. As a reference, constructions of
moduli spaces of parabolic sheaves and generalized parabolic sheaves are provided.
By a refinement of the proof of vanishing theorem, we show, without using vanishing theorem, a new observation that ${\rm dim}\,H^0(\sU_C,\Theta_{\sU_C})$ is independent of all of the choices for any smooth curves. The estimate of various codimension and computation of canonical line bundle of moduli space of generalized parabolic sheaves on a reducible curve are provided in Section 6, which is completely new.
\end{abstract}
\maketitle
\begin{quote}

\end{quote}
\section{Introduction}

Let $C$ be a smooth projective curve of genus $g$, $\mathbf{Q}$ be the quotient scheme of quotients $V\otimes\sO_C(-N)\to E\to 0$ with $$\chi(E)=\chi=d+r(1-g)$$ and let $V\otimes \sO_{C\times\mathbf{Q}}(-N)\to\sF\to 0$ ($\text{where
$V=\mathbb{C}^{P(N)}$}$) be the universal quotient on
$C\times\mathbf{Q}$. There is an ${\rm SL}(V)$-equivariant embeding
$$\mathbf{Q}\hookrightarrow \mathbf{G}={\rm Grass}_{P(m)}(V\otimes{\rm
H}^0(\sO_C(m-N))),$$
and the GIT quotient $\sU_C=\mathbf{Q}^{ss}//{\rm SL}(V)$ respecting to the polarization
$$\Theta_{\mathbf{Q}^{ss}}:={\rm
det}R\pi_{\mathbf{Q}^{ss}}(\sF)^{-k}\otimes{\rm det}(\sF_y)^{\frac{k\chi}{r}}$$
(where $\sF_y=\sF|_{\{y\}\times\mathbf{Q}}$) is the so called moduli space of semi-stable rank $r$ vector bundles of degree $d$ on $C$. When $r|k\chi$, $\Theta_{\mathbf{Q}^{ss}}$
descends to an ample line bundle $\Theta_{\sU_C}$ on $\sU_C$.
When $r=1$, the sections $s\in H^0(\sU_C,\Theta_{\sU_C})$ are nothing but the classical
\textbf{theta functions of order $k$} and ${\rm dim}\,H^0(\sU_C,\Theta_{\sU_C})=k^g.$

When $r>1$, the sections $s\in H^0(\sU_C,\Theta_{\sU_C})$
are so called \textbf{generalized theta functions of order $k$} on $\sU_C$. It is clearly a very interesting
question for mathematicians to find a formula of ${\rm dim}\,H^0(\sU_C,\Theta_{\sU_C})$, which however
was only predicted by \textbf{Conformal Field Theory}, the so called Verlinde formula. For example, when $r=2$,
$${\rm dim}\,H^0(\sU_C,\Theta_{\sU_C})=\left(\frac{k}{2}\right)^g\left(\frac{k+2}{2}\right)^{g-1}\sum^k_{i=0}\frac{(-1)^{id}}{(\sin\frac{(i+1)\pi}{k+2})^{2g-2}}.$$

According to \cite{Be}, there are two kinds of approaches for the proof of Verlinde formula: Infinite-dimensional approaches and finite-dimensional approaches (see \cite{Be} for an account). Infinite-dimensional appeoach is close to physics, which works for any group
$G$, but the geometry behind it is unclear (at least to me). Finite-dimensional approach depends on well understand of geometry of moduli spaces, but it works
only for $r=2$ (as far as I know).

One of the finite-dimensional approaches is to consider a flat family of projective curves $\sX\to T$ such that a fiber $\sX_{t_0}:=X$ ($t_0\in T$) is a connected curve with only one node $x_0\in X$ and $\sX_t$ ($t\in T\setminus\{t_0\}$) are smooth curves with a fiber $\sX_{t_1}=C$ ($t_1\neq t_0$). Then one can associate a family of moduli spaces $\sM\to T$ and a line bundle
$\Theta$ on $\sM$ such that each fiber $\sM_t=\sU_{\sX_t}$ is the moduli space of semi-stable torsion free sheaves on $\sX_t$ and $\Theta|_{\sM_t}=\Theta_{\sU_{\sX_t}}$. By degenerating $C$ to an irreducible $X$, there are two steps to
establish a recurrence relation of $D_g(r,d, k)={\rm dim}\,H^0(\sU_C,\Theta_{\sU_C})$ in term of $g$ (the genus of $C$):\begin{itemize} \item [(1)] (Invariance)  $\,\,{\rm dim}\,H^0(\sU_{\sX_t},\Theta_{\sU_{\sX_t}})$ are independent of $t\in T$; \item [(2)] (Factorization) Let $\pi:\wt X\to X$ be the normalization of $X$, then $$H^0(\sU_X,\Theta_{\sU_X})\cong\bigoplus_{\mu}H^0(\sU^{\mu}_{\widetilde{X}},\Theta_{\sU^{\mu}_{\widetilde{X}}}),$$
where $\mu=(\mu_1,\cdots,\mu_r)$ runs through $0\le\mu_r\le\cdots\le
\mu_1\le k-1$, $\sU^{\mu}_{\widetilde{X}}$ are moduli spaces of semi-stable parabolic bundles on $\wt X$ with parabolic structure at
$x_i\in \pi^{-1}(x_0)=\{x_1,\,x_2\}$ determined by $\mu$ and $\wt X$ has genus $g(\wt X)=g-1$. \end{itemize}

In order to carry throught the induction on $g$, one has to start with moduli spaces $\sU_{\sX_t}=\sU_{\sX_t}(r, d,\omega)$ of semistable parabolic torsion free sheaves $E$ on $\sX_t$
of rank $r$ and ${\rm deg}(E)=d$ with parabolic structures of type
$\{\vec n(x)\}_{x\in I}$ and weights $\{\vec a(x)\}_{x\in I}$ at
smooth points $\{x\}_{x\in I}\subset \sX_t$, where
$\omega=(k, \{\vec n(x),\,\,\vec a(x)\}_{x\in I})$ denote the parabolic data. In \cite{S1}, the factorization theorem as above (2) was proved for $\sU_X=\sU_X(r,d, \omega)$.

Let $\sU_C=\sU_C(r,d,\omega)$ be the moduli space of semi-stable parabolic bundles of rank $r$ and degree $d$ on $C$ with parabolic structures of type
$\{\vec n(x)\}_{x\in I}$ and weights $\{\vec a(x)\}_{x\in I}$ at a finite set $I\subset C$ of points, and
$$D_g(r,d,\omega)={\rm dim}\,H^0(\sU_C,\Theta_{\sU_C}).$$
If the invariance property that ${\rm dim}\,H^0(\sU_C,\Theta_{\sU_C})$ is independent of $C$ and choices of points $x\in I$ holds (for example, if
$H^1(\sU_{\sX_t},\Theta_{\sU_{\sX_t}})=0$), we will have the following recurrence relation
\ga{}{D_g(r,d,\omega)=\sum_{\mu}D_{g-1}(r,d,\omega^{\mu})} where $\mu=(\mu_1,\cdots,\mu_r)$ runs through $0\le\mu_r\le\cdots\le
\mu_1< k$ and $$\omega^{\mu}=(k, \{\vec n(x),\,\,\vec a(x)\}_{x\in I\cup\{x_1,x_2\}})$$
with $\vec n(x_i)$, $\vec a(x_i)$ ($i=1,\,2$) determined by $\mu$. A vanishing theorem
$$H^1(\sU_{\sX_t},\Theta_{\sU_{\sX_t}})=0$$ was proved in \cite{S1} when $(r-1)(g-1)+\frac{|I|}{k}\ge 2$, which implies the
invariance property for $g\ge 3$.

The recurrence relation (1.1) decreases the genus $g$, but it increases the number $|I|$ of parabolic points. By degenerating $C$ to an reducible $X=X_1\cup X_2$, we can establish a recurrence relation for the number of parabolic points if we can prove the invariance property (1) and a factorization (2). In \cite{S2}, we proved the factorization theorem
$$H^0(\sU_{X_1\cup X_2},\Theta_{\sU_{X_1\cup X_2}})\cong\bigoplus_{\mu}
H^0(\sU^{\mu}_{X_1},\Theta_{\sU^{\mu}_{X_1}})\otimes
H^0(\sU^{\mu}_{X_2},\Theta_{\sU^{\mu}_{X_2}})$$ where
$\mu=(\mu_1,\cdots,\mu_r)$ runs through $0\le\mu_r\le\cdots\le
\mu_1<k.$ If $$H^1(\sU_X,\Theta_{\sU_X})=0$$
holds for $X=X_1\cup X_2$, fix a partition $I=I_1\cup I_2$, we have
\ga{} {D_g(r,d.\omega)=\sum_{\mu}D_{g_1}(r,d_1^{\mu},\omega_1^{\mu})\cdot D_{g_2}(r,d_2^{\mu},\omega_2^{\mu}),\quad g_1+g_2=g}
where $d_1^{\mu}+d_2^{\mu}=d$, $\omega_j^{\mu}=(k, \{\vec n(x),\,\,\vec a(x)\}_{x\in I_j\cup\{x_j\}})$ ($j=1,\,2$).

For a projective variety $\hat{M}$ with an ample line bundle $\hat{\sL}$, if a reductive group $G$ acts on $\hat{M}$ with respect to the polarization $\hat{\sL}$ and assume that $\hat{\sL}$ descends to a line bundle $\sL$ on GIT quotient $M=\hat{M}^{ss}(\hat{\sL})//G$, then
$$H^i(M,\sL)=H^i(\hat{M}^{ss}(\hat{\sL}),\hat{\sL})^{inv.}.$$
If there is another $G$-variety $\hat\sY$ with an $G$-morphism $p:\hat\sY\to \hat M$ such that
$H^i(\hat M,\hat\sL)^{inv.}=H^i(\hat\sY, p^*\hat\sL)^{inv.}$, we would be able to show the vanishing theorem $H^i(M,\sL)=0$ by assuming the following statements:\begin{itemize} \item [(i)] There are line bundles $\hat\sL_1$, $\hat\sL_2$ on $\hat\sY$ such that
$p^*\hat\sL=\omega_{\hat\sY}\otimes\hat\sL_1\otimes \hat\sL_2$ (where $\omega_{\hat\sY}$ is the canonical line bundle of $\hat\sY$) and $\hat\sL_1$, $\hat\sL_2$
descend to ample line bundles $\sL_1$, $\sL_2$ on GIT quotient $\sY=\hat\sY^{ss}(\hat\sL_1)//G$;
\item [(ii)] If $\psi:\hat\sY^{ss}(\hat\sL_1)\to \sY$ is quotient map, $\omega_{\sY}=(\psi_*\omega_{\hat\sY^{ss}(\hat\sL_1)})^G$;
\item [(iii)]$H^i(\hat M,\hat\sL)^{inv.}=H^i(\hat{M}^{ss}(\hat{\sL}),\hat{\sL})^{inv.}$ and
$$H^i(\hat\sY, p^*\hat\sL)^{inv.}=H^i(\hat\sY^{ss}(\hat\sL_1), p^*\hat\sL)^{inv.}.$$ \end{itemize}
The above statements imply $H^i(M,\sL)=H^i(\sY, \omega_{\sY}\otimes\sL_1\otimes\sL_2)$, then Kodaira-type vanishing theorem for $\sY$ do the job.
To establish (i), (ii) and (iii), one has to compute canonical bundle and singularities of the moduli spaces, to estimate codimensions of
$$\hat\sY^{ss}(\hat\sL_1)\setminus \hat\sY^{s}(\hat\sL_1),\quad\hat\sM\setminus \hat\sM^{ss}(\hat\sL), \quad\hat\sY\setminus \hat\sY^{ss}(\hat\sL_1),$$
which were done in \cite{S1} for moduli spaces of parabolic bundles and generalized parabolic sheaves on an irreducible smooth curve, so that $H^1(\sU_X,\Theta_{\sU_X})=0$ was only proved for the irreducible nodal curve $X$ of genus $g\ge 3$ in \cite{S1}. If $H^1(\sU_X,\Theta_{\sU_X})=0$ holds
for both irreducible $X$ and reducible $X$ of arbitrary genus, the numbers $D_g(r,d,\omega)$ will satisfy the recurrence relation (1.1) and (1.2) which
will imply a formula of $D_g(r,d,\omega)$. However, the vanishing theorem for reducible curve $X$ remains open.

In this survey article, we provide a detail construction of various moduli spaces in Section 2. The theta line bundles $\Theta_{\sU_X}$ and
the two factorization theorems are reviewed in Section 3. We review firstly the proof of vanishing theorem for smooth curves of $g\ge 2$, then we show, without
using the vanishing of $H^1(\sU_C,\Theta_{\sU_C})$,
that the invariance property of ${\rm dim}\,H^0(\sU_C,\Theta_{\sU_C})$ holds for any smooth curve of genus $g\ge 0$ in Section 4 (see Corollary 4.8). Section 5 contains the review of vanishing theorem for irreducible node curves. Section 6 is an attempt to prove, using the same method of Section 5, the vanishing theorem $H^1(\sU_X, \Theta_{\sU_X})=0$ for reducible curve $X=X_1\cup X_2$.

\section{Construction of moduli spaces}

Let $X$ be an irreducible projective curve of genus $g$ over an
algebraically closed field of characteristic zero, which has at most
one node $x_0$. Let $I$ be a finite set of smooth points of $X$, and
$E$ be a coherent sheaf of rank $r$ and degree $d$ on $X$ (the rank
$r(E)$ is defined to be dimension of $E_{\xi}$ at generic point
$\xi\in X$, and $d=\chi(E)-r(1-g)$).

\begin{defn}\label{defn2.1} By a quasi-parabolic structure on E at a
smooth point $x\in X$, we mean a choice of flag of quotients
$$E_x=Q_{l_x+1}(E)_x\twoheadrightarrow
Q_{l_x}(E)_x\twoheadrightarrow\cdots\cdots\twoheadrightarrow
Q_1(E)_x\twoheadrightarrow Q_0(E)_x=0$$ of the fibre $E_x$ of $E$ at
$x$ (each quotient $Q_i(E)_x\twoheadrightarrow Q_{i-1}(E)_x$ in the
flag is not an isomorphism). If, in addition, a sequence of integers
called the parabolic weights $0\leq a_1(x)<a_2(x)<\cdots
<a_{l_x+1}(x)\le k$ are given, we call that $E$ has a parabolic
structure at $x$.
\end{defn}
Notice that, let $F_i(E)_x:=ker\{E_x\twoheadrightarrow Q_i(E)_x\}$,
it is equivalent to give a flag of subspaces of $E_x$:
$$E_x=F_0(E)_x\supset F_1(E)_x\supset\cdots\supset
F_{l_x}(E)_x\supset F_{l_x+1}(E)_x=0.$$

Let $r_i(x)={\rm dim}(Q_i(E)_x)$, $n_i(x)={\rm
dim}(ker\{Q_i(E)_x\twoheadrightarrow Q_{i-1}(E)_x\})$ (or simply
defining $n_i(x)=r_i(x)-r_{i-1}(x)$) and
$$\aligned
\vec a(x):&=(a_1(x),a_2(x),\cdots,a_{l_x+1}(x))\\
\vec n(x):&=(n_1(x),n_2(x),\cdots,n_{l_x+1}(x)).\endaligned$$ $\vec
a$ (resp., $\vec n$) denotes the map $x\mapsto \vec a(x)$ (resp.,
$x\mapsto\vec n(x)$).
\begin{defn} The parabolic Euler characteristic of $E$ is
$${\rm par}\chi(E):=\chi(E)-\frac{1}{k}\sum_{x\in
I}\left(a_{l_x+1}(x){\rm
dim}(E^{\tau}_x)-\sum^{l_x+1}_{i=1}a_i(x)n_i(x)\right)$$ where
$E^{\tau}\subset E$ is the subsheaf of torsion and
$E^{\tau}_x=E^{\tau}|_{\{x\}}$.
\end{defn}

\begin{defn} For any subsheaf $F\subset E$, let $Q_i(E)_x^F\subset
Q_i(E)_x$ be the image of $F$, $n_i^F={\rm
dim}(ker\{Q_i(E)_x^F\twoheadrightarrow Q_{i-1}(E)_x^F\})$ and
$${\rm par}\chi(F):=\chi(F)-\frac{1}{k}\sum_{x\in
I}\left(a_{l_x+1}(x){\rm
dim}(F^{\tau}_x)-\sum^{l_x+1}_{i=1}a_i(x)n^F_i(x)\right).$$ Then $E$
is called semistable (resp., stable) for $(k,\vec a)$ if for any
nontrivial subsheaf $E'\subset E$ such that $E/E'$ is torsion free,
one has
$${\rm par}\chi(E')\leq
\frac{{\rm par}\chi(E)}{r}\cdot r(E')\,\,(\text{resp., }<).$$
\end{defn}

\begin{rmk}\label{rmk2.4} Stable parabolic sheaf must be torsion free.
If $E$ is semistable, then $E$ is torsion free outside $x\in I$, the
quotient homomorphisms in Definition (\ref{defn2.1}) injection
$E^{\tau}_x$ to $Q_i(E)_x$ ($1\le i\le l_x$) for any $x\in I$.
Moreover, if $E^{\tau}_x\neq 0$, we must have $a_1(x)=0$ and
$a_{l_x+1}(x)=k$.
\end{rmk}

Fix a line bundle $\sO(1)$ on $X$ of ${\rm deg}(\sO(1))=c$, let
$\chi=d+r(1-g)$, $P$ denote the polynomial $P(m)=crm+\chi$,
$\sW=\sO(-N)=\sO(1)^{-N}$ and $V=\Bbb C^{P(N)}$. Consider the Quot
scheme
$${\rm Quot}(V\otimes\sW,P)(T)=\left\{
\aligned &\text{$T$-flat quotients $V\otimes\sW\to E\to 0$ over}\\
&\text{$X\times T$ with $\chi(E_t(m))=P(m)$ ($\forall\,\,t\in
T$)}\endaligned\right\},$$ and let $\mathbf{Q}\subset
Quot(V\otimes\sW,P)$ be the open set
$$\mathbf{Q}(T)=\left\{\aligned&\text{$V\otimes\sW\to E\to 0$, with
$R^1p_{T*}(E(N))=0$ and}\\&\text{$V\otimes\sO_T\to p_{T*}E(N)$
induces an isomorphism}\endaligned\right\}.$$ Choose $N$ large
enough so that every semistable parabolic sheaf with Hilbert
polynomial $P$ and parabolic structures of type $\{\vec n(x)\}_{x\in
I}$ with weights $\{\vec a(x)\}_{x\in I}$ at points $\{x\}_{x\in I}$
appears as a quotient corresponding to a point of $\mathbf{Q}$. Let
$\widetilde{\mathbf{Q}}$ be the closure of $\mathbf{Q}$ in the Quot
scheme, $V\otimes\sW\to \sF\to 0$ be the universal quotient over
$X\times\widetilde{\mathbf{Q}}$ and $\sF_x$ be the restrication of
$\sF$ on
$\{x\}\times\widetilde{\mathbf{Q}}\cong\widetilde{\mathbf{Q}}$. Let
$Flag_{\vec n(x)}(\sF_x)\to\widetilde{\mathbf{Q}}$ be the relative
flag scheme of locally free quotients of type $\vec n(x)$, and
$$\sR=\underset{x\in I}{\times_{\widetilde{\mathbf{Q}}}}
Flag_{\vec n(x)}(\sF_x)\to\widetilde{\mathbf{Q}}$$ be the product
over $\widetilde{\mathbf{Q}}$. A (closed) point
$(p,\{p_{r_1(x)},...,p_{r_{l_x}(x)}\}_{x\in I})$ of $\sR$ by
definition is given by a point $V\otimes\sW\xrightarrow{p}E\to 0$ of
the Quot scheme, together with the flags of quotients
$$\{E_x\twoheadrightarrow
Q_{r_{l_x}(x)}\twoheadrightarrow Q_{r_{l_x-1}(x)}\twoheadrightarrow
\cdots\twoheadrightarrow Q_{r_2(x)}\twoheadrightarrow
Q_{r_1(x)}\twoheadrightarrow 0\}_{x\in I}$$ where $p_{r_i(x)}:
V\otimes\sW\xrightarrow{p}E\to E_x\twoheadrightarrow
Q_{r_{l_x}(x)}\twoheadrightarrow\cdots\twoheadrightarrow
Q_{r_i(x)}$.

For large enough $m$, we have a $SL(V)$-equivariant embedding
$$\sR\hookrightarrow \mathbf{G}=Grass_{P(m)}(V\otimes W_m)\times\bold {Flag},$$
where $W_m=H^0(\sW(m))$, and $\bold{Flag}$ is defined to be
$$\bold{Flag}=\prod_{x\in I}\{Grass_{r_1(x)}(V\otimes W_m)
\times\cdots\times Grass_{r_{l_x}(x)}(V\otimes W_m)\},$$ which maps
a point $(p,\{p_{r_1(x)},...,p_{r_{l_x}(x)}\}_{x\in I})=$
$$(V\otimes\sW\xrightarrow{p}E, \{V\otimes\sW\xrightarrow{p_{r_1(x)}}Q_{r_1(x)},\,\,\cdots\,,
V\otimes\sW\xrightarrow{p_{r_{l_x}(x)}}Q_{r_{l_x}(x)}\}_{x\in I})$$
of $\sR$ to the point $(g,\{g_{r_1(x)},...,g_{r_{l_x}(x)}\}_{x\in
I})=$
$$(V\otimes
W_m\xrightarrow{g}U,\{V\otimes
W_m\xrightarrow{g_{r_1(x)}}U_{r_1(x)},\,\cdots\,, V\otimes
W_m\xrightarrow{g_{r_{l_x}(x)}}U_{r_{l_x}(x)}\}_{x\in I})$$ of
$\mathbf{G}$, where $g:=H^0(p(m))$, $U:=H^0(E(m))$,
$g_{r_i(x)}:=H^0(p_{r_i(x)}(m))$, $U_{r_i(x)}:=H^0(Q_{r_i(x)})$
($i=1,...,l_x$) and $r_i(x)=dim(Q_{r_i(x)})$.

\begin{nota}\label{nota2.5}
Given the polarisation ($N$ large enough) on $\mathbf{G}$:
$$\frac{\ell+kcN}{c(m-N)}\times\prod_{x\in
I}\{d_1(x),\cdots, d_{l_x}(x)\}$$ where $d_i(x)=a_{i+1}(x)-a_i(x)$
and $\ell$ is the rational number satisfying \ga{2.1}{\sum_{x\in
I}\sum_{i=1}^{l_x}d_i(x)r_i(x)+r\ell=k\chi}
\end{nota} By the general criteria of GIT stability, we have

\begin{prop}\label{prop2.6} A point
$(g,\{g_{r_1(x)},...,g_{r_{l_x}(x)}\}_{x\in I})\in \mathbf{G}$ is
stable (respectively, semistable) for the action of $SL(V)$, with
respect to the above polarisation (we refer to this from now on as
GIT-stability), iff for all nontrivial subspaces $H\subset V$ we
have (with $h=dim H$)
$$\aligned e(H):=&\frac{\ell+kcN}{c(m-N)}(hP(m)-P(N)dim\,g(H\otimes W_m))+\\
&\sum_{x\in
I}\sum^{l_x}_{i=1}d_i(x)(r_i(x)h-P(N)dim\,g_{r_i(x)}(H\otimes W_m))
<(\le)\,0.\endaligned$$
\end{prop}

\begin{nota}\label{nota2.7} Given a point
$(p,p_{r_1(x)},...,p_{r_{l_x}(x)}\}_{x\in I})\in\sR$, and a subsheaf
$F$ of $E$ we denote the image of $F$ in $Q_{r_i(x)}$ by
$Q_{r_i(x)}^F$. Similarly, given a quotient $E\xrightarrow{T}\sG\to
0$, set $Q_{r_i(x)}^{\sG}:=Q_{r_i(x)}/Im(ker(T))$.
\end{nota}

\begin{lem}\label{lem2.8}There exists $M_1(N)$ such that for $m\ge M_1(N)$ the
following holds. Suppose
$(p,\{p_{r(x)},p_{r_1(x)},...,p_{r_{l_x}(x)}\}_{x\in I})\in\sR$ is a
point which is GIT-semistable then for all quotients $E\xrightarrow
T\sG\to 0$ we have \ga{2.2}{h^0(\sG(N))\ge
\frac{1}{k}\left(r(\sG)(\ell+kcN)+ \sum_{x\in I}\sum^{l_x}_{i=1}
d_i(x)h^0(Q^{\sG}_{r_i(x)})\right).} In particular, $V\to H^0(E(N))$
is an isomorphism and $E$ satisfies the requirements in Remark
\ref{rmk2.4}.
\end{lem}

\begin{proof} The injectivity of $V\xrightarrow {H^0(p(N))}
H^0(E(N))$ is easy to see. Let
$$H=ker\{V\xrightarrow{H^0(p(N))}H^0(E(N))\xrightarrow{H^0(T(N))}
H^0(\sG(N))\}$$ and $F\subset E$ be the subsheaf generated by $H$.
Since all these $F$ are in a bounded family, $dim\,g(H\otimes
W_m)=h^0(F(m))=\chi(F(m))$ and $g_{r_i(x)}(H\otimes
W_m)=h^0(Q^F_{r_i(x)})$ ($\forall\,\,x\in I$) for $m\ge M_1'(N)$.
Then, by Proposition \ref{prop2.6} (with $h=dim(H)$), we have
$$\aligned e(H)=&(\ell+kcN)(rh-r(F)P(N))+(\ell+kcN)P(N)\frac{h-\chi(F(N))}{c(m-N)}
\\&+\sum_{x\in
I}\sum^{l_x}_{i=1}d_i(x)\left(r_i(x)h-P(N)h^0(Q^F_{r_i(x)})\right).\endaligned$$
By using $h\ge P(N)-h^0(\sG(N))$, $r-r(F)\ge r(\sG)$ and
$r_i(x)-h^0(Q^F_{r_i(x)})\ge h^0(Q^{\sG}_{r_i(x)}),$ we get the
inequality
$$\aligned h^0(\sG(N))\ge&
(\ell+kcN)\frac{h-\chi(F(N))}{k(m-N)c}-\frac{e(H)}{kP(N)}+\\&
\frac{1}{k}\left(r(\sG)(\ell+kcN)+ \sum_{x\in I}\sum^{l_x}_{i=1}
d_i(x)h^0(Q^{\sG}_{r_i(x)})\right).\endaligned$$ For given $N$, the
set $\{h-\chi(F(N))\}$ is finite since all these $F$ are in a
bounded family. Let $\chi(N)={\rm min}\{h-\chi(F(N))\}$. If
$\chi(N)\ge 0$, then
$$h^0(\sG(N))\ge \frac{1}{k}\left(r(\sG)(\ell+kcN)+ \sum_{x\in I}\sum^{l_x}_{i=1}
d_i(x)h^0(Q^{\sG}_{r_i(x)})\right)-\frac{e(H)}{kP(N)}.$$ When
$\chi(N)<0$, let $M_1(N)>{\rm max}\{M_1'(N),
-\chi(N)(\ell+kcN)+cN\}$ and $m\ge M_1(N)$. Then, since $e(H)\le 0$,
we have
$$h^0(\sG(N))\ge \frac{1}{k}\left(r(\sG)(\ell+kcN)+ \sum_{x\in I}\sum^{l_x}_{i=1}
d_i(x)h^0(Q^{\sG}_{r_i(x)})\right).$$

Now we show that $V\to H^0(E(N))$ is an isomorphism. To see it being
surjective, it is enough to show that one can choose $N$ such that
$H^1(E(N))=0$ for all such $E$. If $H^1(E(N))$ is nontrivial, then
there is a nontrivial quotient $E(N)\to L\subset\omega_X$ by Serre
duality, and thus
$$h^0(\omega_X)\ge h^0(L)\ge N+B,$$
where $B$ is a constant independent of $E$, we choose $N$ such that
$H^1(E(N))=0$ for all GIT-semistable points.

Let $\tau=Tor(E)$, $\sG=E/\tau$, note $h^0(\sG(N))=P(N)-h^0(\tau)$
and $$h^0(Q^{\sG}_{r_i(x)})=r_i(x)-h^0(Q^{\tau}_{r_i(x)}),$$ then
the inequality \eqref{2.2} becomes
$$kh^0(\tau)\le\sum_{x\in
I}\sum_{i=1}^{l_x}d_i(x)h^0(Q^{\tau}_{r_i(x)}) \le \sum_{x\in
I}(a_{l_x+1}(x)-a_1(x))h^0(\tau_x)$$ which implies the requirements
in Remark \ref{rmk2.4}.
\end{proof}

\begin{prop}\label{prop2.9} Suppose
$(p,\{p_{r_1(x)},...,p_{r_{l_x}(x)}\}_{x\in I})\in\sR$ is a point
corresponding to a parabolic sheaf $E$. Then $E$ is semistable iff
for any nontrivial subsheaf $F\subset E$ we have
$$\aligned&s(F):=\frac{\ell+kcN}{c(m-N)}(\chi(F(N))P(m)-P(N)\chi(F(m)))+\\
&\sum_{x\in
I}\sum^{l_x}_{i=1}d_i(x)(r_i(x)\chi(F(N))-P(N)h^0(Q_{r_i(x)}^F))
\le\,0.\endaligned$$ If $s(F)<0$ for any nontrivial $F\subset E$,
then $E$ is stable. Conversely, if $E$ is stable, then $s(F)<0$ for
any nontrivial subsheaf $F\subset E$ except that $r(F)=r$,
$\tau:=E/F=0$ outside $x\in I$, $a_{l_x+1}(x)-a_1(x)=k$ if
$\tau_x\neq 0$, and $n_1^F(x)=n_1(x)-h^0(\tau_x)$, $n^F_i(x)=n_i(x)$
($2\le i\le l_x+1$) for any $x\in I$.
\end{prop}

\begin{proof} The point corresonding to a quotient
$V\otimes\sW\xrightarrow{p}E\to 0$ and
$$\{E_x\twoheadrightarrow
Q_{r_{l_x}(x)}\twoheadrightarrow Q_{r_{l_x-1}(x)}\twoheadrightarrow
\cdots\twoheadrightarrow Q_{r_2(x)}\twoheadrightarrow
Q_{r_1(x)}\twoheadrightarrow 0\}_{x\in I}$$ $p_{r_i(x)}:
V\otimes\sW\xrightarrow{p}E\to E_x\twoheadrightarrow
Q_{r_{l_x}(x)}\twoheadrightarrow\cdots\twoheadrightarrow
Q_{r_i(x)}$. For $F\subset E$ such that $E/F$ is torsion free, we
have the flags of quotient sheaves
$$\{F\twoheadrightarrow F_x\twoheadrightarrow
Q^F_{r_{l_x}(x)}\twoheadrightarrow
Q^F_{r_{l_x-1}(x)}\twoheadrightarrow \cdots\twoheadrightarrow
Q^F_{r_2(x)}\twoheadrightarrow Q^F_{r_1(x)}\twoheadrightarrow
0\}_{x\in I}$$ Let
$n_i^F(x)=h^0(Q^F_{r_i(x)})-h^0(Q^F_{r_{i-1}(x)})$, notice that
$$\aligned \sum_{x\in I}\sum^{l_x}_{i=1}d_i(x)r_i(x)=&\,\,r\sum_{x\in
I}a_{l_x+1}(x) +\sum_{x\in
I}a_{l_x+1}(x)h^0(E^{\tau}_x)\\&-\sum_{x\in
I}\sum^{l_x+1}_{i=1}a_i(x)n_i(x)\endaligned$$
$$\aligned\sum_{x\in I}\sum^{l_x}_{i=1}d_i(x)h^0(Q^F_{r_i(x)})=&
\,\,r(F)\sum_{x\in I}a_{l_x+1}(x)+\sum_{x\in
I}a_{l_x+1}(x)h^0(F^{\tau}_x)
\\&-\sum_{x\in I}\sum^{l_x+1}_{i=1}a_i(x)n^F_i(x),\endaligned$$
$\chi(F(N))P(m)-P(N)\chi(F(m))=c(m-N)(r\chi(F)-r(F)\chi(E))$, then
$$\aligned s(F)=&\left(r\ell+rkcN+\sum_{x\in
I}\sum^{l_x}_{i=1}d_i(x)r_i(x)\right)\left(\chi(F)-\frac{r(F)}{r}\chi(E)\right)
+\\& P(N)\left(\frac{r(F)}{r}\sum_{x\in
I}\sum^{l_x}_{i=1}d_i(x)r_i(x)\,\,-\sum_{x\in
I}\sum^{l_x}_{i=1}d_i(x)h^0(Q^F_{r_i(x)})\right)\\=& kP(N)\left({\rm
par}\chi((F)-\frac{r(F)}{r}{\rm par}\chi(E)\right).\endaligned$$

For any nontrivial subsheaf $F\subset E$, let $\tau$ be the torsion
of $E/F$ and $F'\subset E$ such that $\tau=F'/F$ and $E/F'$ torsion
free. If we write $\tau=\tilde\tau+\sum_{x\in I}\tau_x,$ note
$h^0(\tau_x)+h^0(Q^F_{r_i(x)})-h^0(Q^{F'}_{r_i(x)})\ge 0$, then
$$\aligned s(F)-s(F')&=-kP(N)h^0(\widetilde{\tau})-P(N)\sum_{x\in
I}(k-a_{l_x+1}(x)+a_1(x))h^0(\tau_x) \\-P(N)&\sum_{x\in
I}\sum^{l_x}_{i=1}d_i(x)(h^0(\tau_x)+h^0(Q^F_{r_i(x)})-h^0(Q^{F'}_{r_i(x)}))\,\,\le
0.\endaligned$$ If $E$ is stable and $s(F)=0$, it is easy to see
that the last requirements in the proposition are satisfied.
\end{proof}

\begin{prop}\label{prop2.10} There exists an integer $M_1(N)>0$ such
that for $m\ge M_1(N)$ the following is true. If a point
$$(p,\{p_{r_1(x)},...,p_{r_{l_x}(x)}\}_{x\in I})\in\sR$$ is GIT-stable
(respectively, GIT-semistable), then the quotient $E$ is a stable
(respectively, semistable) parabolic sheaf and $V\to H^0(E(N))$ is
an isomorphism.\end{prop}

\begin{proof} If $(p,\{p_{r_1(x)},...,p_{r_{l_x}(x)}\}_{x\in
I})\in \sR$ is GIT-stable (GIT-semistable), by Lemma \ref{lem2.8},
$V\to H^0(E(N))$ is an isomorphism. For any nontrivial subsheaf
$F\subset E$ with $E/F$ torsion free, let $H\subset V$ be the
inverse image of $H^0(F(N))$ and $h=dim(H)$, we have (for $m>N$)
$$\chi(F(N))P(m)-P(N)\chi(F(m))\le hP(m)-P(N)h^0(F(m))$$ for $m>N$
(note $h^1(F(N))\ge h^1(F(m))$). Thus $s(F)\le e(H)$ since
$$g(H\otimes W_m)\le h^0(F(m)), \quad g_{r_i(x)}(H\otimes W_m)\le
h^0(Q^F_{r_i(x)})$$ (the inequalities are strict when $h=0$). By
Proposition \ref{prop2.6} and Proposition \ref{prop2.9}, $E$ is
stable (respectively, semistable) if the point is GIT stable
(respectively, GIT semistable).
\end{proof}

For a semistable parabolic sheaf $E$ of rank $r$ on $X$, we have,
for any subsheaf $F\subset E$, $\chi(F)\le
\frac{\chi(E)}{r}r(F)+2r|I|$. The following elementary lemma should
be well-known.
\begin{lem}\label{lem2.11} Let $E$ be a coherent sheaf of rank $r$ on $X$. If
$$\chi(F)\le\frac{\chi(E)}{r}r(F)+C,\quad \forall \,\,F\subset E.$$
Then, for any $F\subset E$ with $H^1(F)\neq 0$, we have
$$h^0(F)\le \frac{\chi(E)}{r}(r(F)-1)+C+r(F)g.$$
\end{lem}

\begin{proof} $H^1(F)\neq 0$ means that we have a nontrivial
morphism $F\to \omega_X$. Let $F'$ be the kernel of $F\to \omega_X$,
then $h^0(F)\le h^0(F')+g$. If $H^1(F')=0$, we have
$h^0(F)\le\chi(F')+g\le\frac{\chi(E)}{r}(r(F)-1)+C+g$. If
$H^1(F')\neq 0$, by repeating the arguments to $F'$, we get the
required inequality.
\end{proof}

\begin{prop}\label{prop2.12} There exist integers $N>0$ and $M_2(N)>0$ such
that for $m\ge M_2(N)$ the following is true. If a point
$$(p,\{p_{r_1(x)},...,p_{r_{l_x}(x)}\}_{x\in I})\in\sR$$ corresponds
to a semistable parabolic sheaf $E$, then the point is
GIT-semistable. Moreover, if $E$ is a stable parabolic sheaf, then
the point is GIT stable except the case $a_{l_x+1}(x)-a_1(x)=k$.
\end{prop}

\begin{proof} There is $N_1>0$ such that for any $N\ge N_1$ the
following is true. For any $V\otimes \sW\xrightarrow{p}E\to 0$ with
semistable parabolic sheaf $E$, the induced map $V\to H^0(E(N))$ is
an isomorphism.

Let $H\subset V$ be a nontrivial subspace of ${\rm dim}(H)=h$ and
$F\subset E$ be the sheaf such that $F(N)\subset E(N)$ is generated
by $H$. Since all these $F$ are in a bounded family (for fixed $N$),
$dim\,g(H\otimes W_m)=h^0(F(m))=\chi(F(m))$, $g_{r_i(x)}(H\otimes
W_m)=h^0(Q^F_{r_i(x)})$ ($\forall\,\,x\in I$) for $m\ge M_1'(N)$ and
$$e(H)=s(F)+\frac{\ell+kcm}{c(m-N)}P(N)\left(h-\chi(F(N))\right).$$
If $H^1(F(N))=0$, we have $e(H)\le s(F)$ since $h\le h^0(F(N))$.
Then $e(H)\le s(F)\le 0$ by Proposition \ref{prop2.9} since $E$ is a
semistable parabolic sheaf. If $H^1(F(N))\neq 0$, by Lemma
\ref{lem2.11}, we have
$$h^0(F(N))\le\frac{rcN+\chi}{r}(r(F)-1)+r(g+2|I|).$$
Putting $h\le h^0(F(N))$ and above inequality in the equality
$$\aligned
e(H)=&P(N)\left(kh-(\ell+kcN)r(F)+(\ell+kcN)\frac{h-\chi(F(N))}{c(m-N)}\right)
\\&-P(N)\sum_{x\in
I}\sum^{l_x}_{i=1}d_i(x)h^0(Q^F_{r_i(x)}),\endaligned$$ then, let
$C=k|\chi|+r(g+2|I|)k+|\ell|r$, we have
$$e(H)\le P(N)\left(
-kcN+C+(\ell+kcN)\frac{h-\chi(F(N))}{c(m-N)}\right).$$ Choose an
integer $N_2\ge N_1$ such that $-kcN_2+C<-1$. Then, for any fixed
$N\ge N_2$, there is an integer $M_2(N)$ such that for $m\ge M_2(N)$
$$(\ell+kcN)\frac{h-\chi(F(N))}{c(m-N)}<1$$
for any $H\subset V$, which implies $e(H)<0$ and we are done.

\end{proof}

\begin{thm}\label{thm2.13} There
exists a seminormal projective variety
$$\sU_X:=\sU_X(r,d,
\{k,\vec n(x),\vec a(x)\}_{x\in I}),$$ which is the coarse moduli
space of $s$-equivalence classes of semistable parabolic sheaves $E$
of rank $r$ and $\chi(E)=\chi=d+r(1-g)$ with parabolic structures of type
$\{\vec n(x)\}_{x\in I}$ and weights $\{\vec a(x)\}_{x\in I}$ at
points $\{x\}_{x\in I}$. If $X$ is smooth, then it is normal, with
only rational singularities.
\end{thm}

\begin{proof} Let $\sR^{ss}\subset \sR$ be the open set consisting
of semistable parabolic sheaves. $\sU_X:=\sU_X(r,\chi,I,k,\vec
a,\vec n)$ is defined to be the GIT quotient $\sR^{ss}//SL(V)$. The
statements about singularities of $\sU_X$ are proved in \cite{S1}.
The case $a_{l_x+1}(x)-a_1(x)=k$ can be covered by the same
arguments in \cite{S1} where we proved that $\sH$ is normal with
only rational singularities.
\end{proof}

When $X$ is a reduced projective curve with two smooth irreducible
components $X_1$ and $X_2$ of genus $g_1$ and $g_2$ meeting at only
one point $x_0$ (which is the only node of $X$), we fix an ample
line bundle $\sO(1)$ of degree $c$ on $X$ such that
$deg(\sO(1)|_{X_i})=c_i>0$ ($i=1,2$). For any coherent sheaf $E$,
$P(E,n):=\chi(E(n))$ denotes its Hilbert polynomial, which has degree
$1$. We define the rank of $E$ to be
$$r(E):=\frac{1}{deg(\sO(1))}\cdot \lim \limits_{n\to\infty}\frac{P(E,n)}
{n}.$$ Let $r_i$ denote the rank of the restriction of $E$ to $X_i$
($i=1,2$), then
$$P(E,n)=(c_1r_1+c_2r_2)n+\chi(E),\quad r(E)=
\frac{c_1}{c_1+c_2}r_1+\frac{c_2}{c_1+c_2}r_2.$$ We say that $E$ is
of rank $r$ on $X$ if $r_1=r_2=r$, otherwise it will be said of rank
$(r_1,r_2)$.

Fix a finite set $I=I_1\cup I_2$ of smooth points on $X$, where
$I_i=\{x\in I\,|\,x\in X_i\}$ ($i=1,2$), and parabolic data $\omega=\{k,\vec
n(x),\vec a(x)\}_{x\in I}$ with
$$\ell:=\frac{k\chi-\sum_{x\in I}\sum^{l_x}_{i=1}d_i(x)r_i(x)}{r}$$
(recall $d_i(x)=a_{i+1}(x)-a_i(x)$, $r_i(x)=n_1(x)+\cdots+n_i(x)$).
Then we will indicate how the
same construction gives moduli space of semistable parabolic sheaves
on $X$ (see \cite{S2} for details). For simplicity, we only state
the case that $a_{l_x+1}(x)-a_1(x)<k$ ($\forall\,x\in I$).

\begin{defn}\label{defn2.14} For any coherent sheaf $F$ of rank $(r_1,r_2)$, let
$$m(F):= \frac{r(F)-r_1}{k}\sum_{x\in I_1}a_{l_x+1}(x)+
\frac{r(F)-r_2}{k}\sum_{x\in I_2}a_{l_x+1}(x),$$ the modified
parabolic Euler characteristic and slop of $F$ are
$${\rm par}\chi_m(F):={\rm par}\chi(F)+m(F),\quad {\rm par}\mu_m(F):=\frac{{\rm par}\chi_m(F)}{r(F)}.$$
A parabolic sheaf $E$ is called semistable (resp. stable) if, for
any subsheaf $F\subset E$ such $E/F$ is torsion free, one has, with
the induced parabolic structure,
$${\rm par}\chi_m(F)\le \frac{{\rm par}\chi_m(E)}{r(E)}r(F)\quad (resp.<).$$
\end{defn}

There is a similar $\sR$ and a $SL(V)$-equivariant embedding
$\sR\hookrightarrow \mathbf{G}$. As the same as Notation
\ref{nota2.5}, give the polarization on $\mathbf{G}$:
$$\frac{\ell+kcN}{c(m-N)}\times\prod_{x\in
I}\{d_1(x),\cdots, d_{l_x}(x)\}.$$ Then we have the same Proposition
\ref{prop2.6}, Lemma \ref{lem2.8}, Proposition \ref{prop2.9} and
Lemma \ref{lem2.11}. The modification in the proof of Proposition
\ref{prop2.9} is: for $F\subset E$ of rank $(r_1,r_2)$ such that
$E/F$ is torsion free, we have
$$\sum_{x\in I}\sum^{l_x}_{i=1}d_i(x)r_i(x)=\,\,r\sum_{x\in
I}a_{l_x+1}(x)-\sum_{x\in I}\sum^{l_x+1}_{i=1}a_i(x)n_i(x),$$
$$\sum_{x\in I}\sum^{l_x}_{i=1}d_i(x)h^0(Q^F_{r_i(x)})=
r_1\sum_{x\in I_1}a_{l_x+1}(x)+r_2\sum_{x\in
I_2}a_{l_x+1}(x)-\sum_{x\in I}\sum^{l_x+1}_{i=1}a_i(x)n^F_i(x),$$
$$s(F)=kP(N)\left({\rm
par}\chi_m((F)-\frac{r(F)}{r}{\rm par}\chi_m(E)\right).$$ In
particular, we have

\begin{prop}\label{prop2.15} There exist integers $N>0$ and $M_2(N)>0$ such
that for $m\ge M_2(N)$ the following is true. If a point
$$(p,\{p_{r_1(x)},...,p_{r_{l_x}(x)}\}_{x\in I})\in\sR$$ corresponds
to a quasi-parabolic sheaf $E$, then the point is GIT-semistable
(resp. GIT-stable) under the above polarization if and only if $E$
is a semistable (resp. stable) parabolic sheaf for the weights
$0\leq a_1(x)<a_2(x)<\cdots <a_{l_x+1}(x)<k$ ($\forall\,\,x\in I$).
\end{prop}

\begin{thm}\label{thm2.16} There
exists a reduced, seminormal projective scheme
$$\sU_X:=\sU_X(r,d,\sO(1),
\{k,\vec n(x),\vec a(x)\}_{x\in I_1\cup I_2})$$ which is the coarse
moduli space of $s$-equivalence classes of semistable parabolic
sheaves $E$ of rank $r$ and $\chi(E)=\chi=d+r(1-g)$ with parabolic structures
of type $\{\vec n(x)\}_{x\in I}$ and weights $\{\vec a(x)\}_{x\in
I}$ at points $\{x\}_{x\in I}$. The moduli space $\sU_X$ has at most
$r+1$ irreducible components.
\end{thm}

\begin{proof} Let $\sR^{ss}\subset\sR$ be the open set of
semi-stable parabolic sheaves. $\sU_X:=\sU_X(r, d,\sO(1),
\{k,\vec n(x),\vec a(x)\}_{x\in I_1\cup I_2})$ is defined to be the
GIT quotient $\sR^{ss}//{\rm SL}(V)$. Let $\sU_X^0\subset \sU_X$ be
the dense open set of locally free sheaves. For any $E\in \sU_X^0$,
let $E_1$ and $E_2$ be the restrictions of $E$ to $X_1$ and $X_2$.
By the exact sequence
$$0\to E_1(-x_0)\to E\to E_2\to 0$$
and semi-stability of $E$, we have
$$\frac{c_1}{c_1+c_2}par\chi_m(E)\le par\chi_m(E_1)\le\frac{c_1}{c_1+c_2}
par\chi_m(E)+r,$$
$$\frac{c_2}{c_1+c_2}par\chi_m(E)\le par\chi_m(E_2)\le\frac{c_2}{c_1+c_2}
par\chi_m(E)+r.$$
For $j=1,\,2$ and $\omega=\{k,\vec n(x),\vec a(x)\}_{x\in I_1\cup I_2}$, let $\chi_j=\chi(E_j)$ and \ga{2.3}
{n^{\omega}_j=\frac{1}{k}\left(r\frac{c_j}{c_1+c_2}\ell+\sum_{x\in
I_j}\sum^{l_x}_{i=1}d_i(x)r_i(x)\right).} Then the above
inequalities can be rewritten as \ga{2.4} {n^{\omega}_1\le\chi_1\le
n^{\omega}_1+r,\quad n^{\omega}_2\le\chi_2\le n^{\omega}_2+r.} There are at most $r+1$ possible
choices of $(\chi_1,\chi_2)$ satisfying \eqref{2.4} and
$\chi_1+\chi_2=\chi+r$, each of the choices corresponds an
irreducible component of $\sU_X$.
\end{proof}

\begin{rmks}\label{rmks2.17} (1) If $n^{\omega}_j$ ($j=1,\,2$) are not integers, then there are at most $r$
irreducible components $\sU_X^{\chi_1,\,\chi_2}\subset\sU_X$ of
$\sU_X$ with \ga{2.5} {n^{\omega}_1<\chi_1< n^{\omega}_1+r,\quad n^{\omega}_2<\chi_2< n^{\omega}_2+r}
such that the (dense) open set of parabolic bundles $E\in
\sU_X^{\chi_1,\,\chi_2}$ satisfy
$$\chi(E|_{X_1})=\chi_1,\quad \chi(E|_{X_2})=\chi_2.$$
For any $\chi_1,\,\chi_2$ satisfying \eqref{2.5}, let $\sU_{X_1}$
(resp. $\sU_{X_2}$) be the moduli space of semistable parabolic
bundles of rank $r$ and Euler characteristic $\chi_1$ (resp.
$\chi_2$), with parabolic structures of type $\{\vec n(x)\}_{x\in
I_1}$ (resp. $\{\vec n(x)\}_{x\in I_2}$) and weights $\{\vec
a(x)\}_{x\in I_1}$ (resp. $\{\vec a(x)\}_{x\in I_2}$) at points
$\{x\}_{x\in I_1}$ (resp. $\{x\}_{x\in I_2}$), then
$\sU_X^{\chi_1,\,\chi_2}$ is not empty if $\sU_{X_j}$ ($j=1,\,2$)
are not empty (See Proposition 1.4 of \cite{S2}). In fact,
$\sU_X^{\chi_1,\,\chi_2}$ contains a stable parabolic bundle if one
of $\sU_{X_j}$ ($j=1,\,2$) contains a stable parabolic bundle.

(2) Let $E\in \sU_X$, for any nontrivial $F\subset E$ of rank
$(r_1,r_2)$ such that $E/F$ torsion free, we have \ga{2.6}{\aligned
&k r(F)(par\mu_m(F)-par\mu_m(E))\\&=k\chi(F)-\sum_{x\in
I}\sum_{i=1}^{l_x}d_i(x)h^0(Q^F_{r_i(x)})-r(F)\ell,\endaligned}
which implies the following facts: (i) When $\ell=0$, the moduli
spaces $\sU_X$ is independent of the choices of $\sO(1)$. (ii) When
$\ell\neq0$, we can choose $\sO(1)$ such that all the numbers $n^{\omega}_1$,
$n^{\omega}_2$ and $r(F)\ell$ (for all possible $r_1\neq r_2$) are not integers
(we call such $\sO(1)$ \textbf{a generic polarization}, its
existence is an easy excise). Then, for any $E\in\sU_X\setminus
\sU_X^s$ (i.e. non-stable sheaf), the sub-sheaf $F\subset E$ of rank
$(r_1,r_2)$ with $par\mu_m(F)=par\mu_m(E)$ must have $r_1=r_2$.
\end{rmks}

When $X$ is a connected nodal curve (irreducible or reducible) of
genus $g$, with only one node $x_0$, let $\pi:\widetilde{X}\to X$ be
the normalization and $\pi^{-1}(x_0)=\{x_1,x_2\}$. Then the
normalization $\phi: \sP\to\sU_X$ of $\sU_X$ is given by moduli
space of generalized parabolic sheaves (GPS) on $\widetilde{X}$.

Recall that a GPS $(E,Q)$ of rank $r$ on $\widetilde{X}$ consists of
a sheaf $E$ on $\widetilde{X}$, torsion free of rank $r$ outside
$\{x_1,x_2\}$ with parabolic structures at the points of $I$ (we
identify $I$ with $\pi^{-1}(I)$) and an $r$-dimensional quotient
$$E_{x_1}\oplus E_{x_2}\xrightarrow{q} Q\to 0.$$

The moduli space $\sP$ consists of semistable $(E,Q)$ with
additional parabolic structures at the points of $I$ (we identify
$I$ with $\pi^{-1}(I)$) given by the data $\omega=(r,\chi,\{\vec
n(x),\,\vec a(x)\}_{x\in I}, \mathcal{O}(1),k)$ satisfying
$$\sum_{x\in I}\sum^{l_x}_{i=1}d_i(x)r_i(x)+r\wt\ell=k\wt\chi$$ where
$d_i(x)=a_{i+1}(x)-a_i(x)$, $\wt\chi=\chi+r$, $\wt\ell=k+\ell$ and
the pullback $\pi^*\mathcal{O}(1)$ is denoted by
$\wt{\mathcal{O}}(1)$ (See \cite{S1} and \cite{S2} for details).

\begin{defn}\label{defn2.18} A GPS $(E,Q)$ is called semistable (resp.,
stable), if for every nontrivial subsheaf $E'\subset E$ such that
$E/E'$ is torsion free outside $\{x_1,x_2\},$ we have, with the
induced parabolic structures at points $\{x\}_{x\in I}$,
$$par\chi_m(E')-dim(Q^{E'})\leq
r(E')\cdot\frac{par\chi_m(E)-dim(Q)}{r(E)} \,\quad (\text{resp.,
$<$}),$$ where $Q^{E'}=q(E'_{x_1}\oplus E'_{x_2})\subset Q.$
\end{defn}

When $X$ is irreducible, let $\wt P$ denote the polynomial
$\wt{P}(m)=crm+\wt\chi$, $\wt\sW=\wt{\sO}(-N)=\wt{\sO}(1)^{-N}$ and
$\wt V=\Bbb C^{\wt{P}(N)}$. Consider the Quot scheme
$${\rm Quot}(\wt V\otimes\wt\sW,P)(T)=\left\{
\aligned &\text{$T$-flat quotients $\wt V\otimes\wt \sW\to E\to 0$ over}\\
&\text{$\wt X\times T$ with $\chi(E_t(m))=\wt P(m)$
($\forall\,\,t\in T$)}\endaligned\right\},$$ and let
$\mathbf{Q}\subset Quot(\wt V\otimes\wt\sW,P)$ be the open set
$$\mathbf{Q}(T)=\left\{\aligned&\text{$\wt V\otimes\wt\sW\to E\to 0$, with
$R^1p_{T*}(E(N))=0$ and}\\&\text{$\wt V\otimes\sO_T\to p_{T*}E(N)$
induces an isomorphism}\endaligned\right\}.$$ Let
$\widetilde{\mathbf{Q}}$ be the closure of $\mathbf{Q}$ in the Quot
scheme, $\wt V\otimes\wt\sW\to \wt\sF\to 0$ be the universal
quotient over $\wt X\times\widetilde{\mathbf{Q}}$ and $\wt\sF_x$ be
the restriction of $\wt\sF$ on
$\{x\}\times\widetilde{\mathbf{Q}}\cong\widetilde{\mathbf{Q}}$. Let
$Flag_{\vec n(x)}(\wt\sF_x)\to\widetilde{\mathbf{Q}}$ be the
relative flag scheme of locally free quotients of type $\vec n(x)$,
and
$$\wt\sR=\underset{x\in I}{\times_{\widetilde{\mathbf{Q}}}}
Flag_{\vec n(x)}(\wt\sF_x)\to\widetilde{\mathbf{Q}},\quad
\wt{\sR'}=\widetilde{\sR}\times_{\widetilde{\mathbf{Q}}}
Grass_r(\wt\sF_{x_1}\oplus\wt\sF_{x_2}).$$ A (closed) point
$(p,\{p_{r_1(x)},...,p_{r_{l_x}(x)}\}_{x\in I},q_s)$ of $\wt{\sR'}$
by definition is given by a point $\wt
V\otimes\wt\sW\xrightarrow{p}E\to 0$ of the Quot scheme, together
with the flags of quotients
$$\{E_x\twoheadrightarrow
Q_{r_{l_x}(x)}\twoheadrightarrow Q_{r_{l_x-1}(x)}\twoheadrightarrow
\cdots\twoheadrightarrow Q_{r_2(x)}\twoheadrightarrow
Q_{r_1(x)}\twoheadrightarrow 0\}_{x\in I}$$ and a $r$-dimensional
quotient $E_{x_1}\oplus E_{x_2}\xrightarrow{q} Q\to 0 $, where
$p_{r_i(x)}: \wt V\otimes\wt\sW\xrightarrow{p}E\to
E_x\twoheadrightarrow
Q_{r_{l_x}(x)}\twoheadrightarrow\cdots\twoheadrightarrow Q_{r_i(x)}$
and $q_s:\wt V\otimes\wt\sW\xrightarrow{p}E\to E_{x_1}\oplus
E_{x_2}\xrightarrow{q} Q$. Choose $N$ large enough so that every
semistable GPS $(E,Q)$ with $\chi(E(m))=\wt P(m)$ and parabolic
structures of type $\{\vec n(x)\}_{x\in I}$ with weights $\{\vec
a(x)\}_{x\in I}$ at points $\{x\}_{x\in I}$ appears as a point of
$\wt{\sR'}$. For large enough $m$, we have a $SL(\wt V)$-equivariant
embedding
$$\wt{\sR'}\hookrightarrow \mathbf{G}'=Grass_{\wt P(m)}(\wt V\otimes W_m)\times\bold {Flag}\times Grass_r(\wt V\otimes W_m),$$
where $W_m=H^0(\wt\sW(m))$, and $\bold{Flag}$ is defined to be
$$\bold{Flag}=\prod_{x\in I}\{Grass_{r_1(x)}(\wt V\otimes W_m)
\times\cdots\times Grass_{r_{l_x}(x)}(\wt V\otimes W_m)\},$$ which
maps a point $(p,\{p_{r_1(x)},...,p_{r_{l_x}(x)}\}_{x\in
I},q_s)=(\wt V\otimes\wt\sW\xrightarrow{p}E,$
$$ \{\wt V\otimes\wt\sW\xrightarrow{p_{r_1(x)}}Q_{r_1(x)},\,\,\cdots\,,
\wt V\otimes\wt\sW\xrightarrow{p_{r_{l_x}(x)}}Q_{r_{l_x}(x)}\}_{x\in
I},\wt V\otimes\wt\sW\xrightarrow{q_s}Q)$$ of $\wt{\sR'}$ to the
point $(g,\{g_{r_1(x)},...,g_{r_{l_x}(x)}\}_{x\in I},g_G)=(\wt
V\otimes W_m\xrightarrow{g}U,$
$$\{\wt V\otimes
W_m\xrightarrow{g_{r_1(x)}}U_{r_1(x)},\,\cdots\,, \wt V\otimes
W_m\xrightarrow{g_{r_{l_x}(x)}}U_{r_{l_x}(x)}\}_{x\in I},\wt
V\otimes W_m\xrightarrow{g_G}U_r)$$ of $\mathbf{G}'$, where
$g:=H^0(p(m))$, $U:=H^0(E(m))$, $g_{r_i(x)}:=H^0(p_{r_i(x)}(m))$,
$U_{r_i(x)}:=H^0(Q_{r_i(x)})$ ($i=1,...,l_x$), $g_G:=H^0(q_s(m))$,
$U_r:=H^0(Q)$ and $r_i(x)=dim(Q_{r_i(x)})$. Given $\mathbf{G}'$ the
polarisation
$$\frac{(\ell+kcN)}{c(m-N)}\times\prod_{x\in
I} \{d_1(x),\cdots,d_{l_x}(x)\}\times k.$$ Then, by the general
criteria of GIT stability, we have

\begin{prop}\label{prop2.19} A point
$(g,\{g_{r_1(x)},...,g_{r_{l_x}(x)}\}_{x\in I},g_G)\in \mathbf{G}'$
is stable (respectively, semistable) for the action of $SL(\wt V)$,
with respect to the above polarisation (we refer to this from now on
as GIT-stability), iff for all nontrivial subspaces $H\subset \wt V$
we have (with $h=dim H$)
$$\aligned e(H):=&\frac{\ell+kcN}{c(m-N)}(h\wt P(m)-\wt P(N)dim\,g(H\otimes W_m))+ \\
&\sum_{x\in I}\sum^{l_x}_{i=1}d_i(x)(r_i(x)h-\wt
P(N)dim\,g_{r_i(x)}(H\otimes W_m))\\&+k(rh-\wt P(N)dim\,g_G(H\otimes
W_m))<(\le)\,0.\endaligned$$
\end{prop}

\begin{lem}\label{lem2.20}There exists $M_1(N)$ such that for $m\ge M_1(N)$ the
following holds. Suppose
$(p,\{p_{r(x)},p_{r_1(x)},...,p_{r_{l_x}(x)}\}_{x\in
I},q_s)\in\wt{\sR'}$ is GIT-semistable, then for all quotients
$E\xrightarrow T\sG\to 0$ we have $$h^0(\sG(N))\ge
\frac{1}{k}\left(r(\sG)(\ell+kcN)+ \sum_{x\in I}\sum^{l_x}_{i=1}
d_i(x)h^0(Q^{\sG}_{r_i(x)})\right)+h^0(Q^{\sG}).$$ In particular,
$\wt V\to H^0(E(N))$ is an isomorphism and $E$ satisfies the
following conditions: (1) the torsion ${\rm Tor}\,E$ of $E$ is
supported on $\{x_1,x_2\}$ and $q:({\rm Tor}\,E)_{x_1}\oplus ({\rm
Tor}\,E)_{x_2}\hookrightarrow Q$, (2) if $N$ is large enough, then
$H^1(E(N)(-x-x_1-x_2))=0$ for all $E$ and $x\in\wt X$.
\end{lem}

\begin{proof} Let
$H=ker\{\wt V\xrightarrow{H^0(p(N))}H^0(E(N))\xrightarrow{H^0(T(N))}
H^0(\sG(N))\}$ and $F\subset E$ be the subsheaf generated by $H$.
Since all these $F$ are in a bounded family, there exists an integer
$M_1'(N)$ such that $dim\,g(H\otimes W_m)=h^0(F(m))=\chi(F(m))$,
$g_{r_i(x)}(H\otimes W_m)=h^0(Q^F_{r_i(x)})$ ($\forall\,\,x\in I$)
and $dim\,g_G(H\otimes W_m)=h^0(Q^F)$ for $m\ge M_1'(N)$. Then, by
Proposition \ref{prop2.19} (with $h=dim(H)$), we have
$$\aligned &e(H)=(\ell+kcN)(rh-r(F)\wt{P}(N))+(\ell+kcN)\wt{P}(N)\frac{h-\chi(F(N))}{c(m-N)}
\\&+\sum_{x\in
I}\sum^{l_x}_{i=1}d_i(x)\left(r_i(x)h-\wt{P}(N)h^0(Q^F_{r_i(x)})\right)+k(rh-\wt
P(N)h^0(Q^F)).\endaligned$$ By using $h\ge \wt P(N)-h^0(\sG(N))$,
$r-r(F)\ge r(\sG)$, $r_i(x)-h^0(Q^F_{r_i(x)})\ge
h^0(Q^{\sG}_{r_i(x)})$ and $r-h^0(Q^F)\ge h^0(Q^{\sG})$, we get the
inequality
$$\aligned h^0(\sG(N))\ge&
(\ell+kcN)\frac{h-\chi(F(N))}{k(m-N)c}-\frac{e(H)}{k\wt
P(N)}+h^0(Q^{\sG})+\\& \frac{1}{k}\left(r(\sG)(\ell+kcN)+ \sum_{x\in
I}\sum^{l_x}_{i=1} d_i(x)h^0(Q^{\sG}_{r_i(x)})\right).\endaligned$$
For given $N$, the set $\{h-\chi(F(N))\}$ is finite since all these
$F$ are in a bounded family. Let $\chi(N)={\rm
min}\{h-\chi(F(N))\}$. If $\chi(N)\ge 0$, then
$$\aligned h^0(\sG(N))\ge &\frac{1}{k}\left(r(\sG)(\ell+kcN)+ \sum_{x\in I}\sum^{l_x}_{i=1}
d_i(x)h^0(Q^{\sG}_{r_i(x)})\right)\\&+h^0(Q^{\sG})-\frac{e(H)}{k\wt
P(N)}.\endaligned$$ When $\chi(N)<0$, let $M_1(N)>{\rm
max}\{M_1'(N), -\chi(N)(\ell+kcN)+cN\}$ and $m\ge M_1(N)$. Then,
since $e(H)\le 0$, we have
$$h^0(\sG(N))\ge \frac{1}{k}\left(r(\sG)(\ell+kcN)+ \sum_{x\in I}\sum^{l_x}_{i=1}
d_i(x)h^0(Q^{\sG}_{r_i(x)})\right)+h^0(Q^{\sG}).$$

Now we show that $\wt V\to H^0(E(N))$ is an isomorphism. The
injectivity of $\wt V\xrightarrow {H^0(p(N))} H^0(E(N))$ is easy to
see. To see it being surjective, it is enough to show that one can
choose $N$ such that $H^1(E(N))=0$ for all such $E$. We prove
$H^1(E(N)(-x_1-x_2-x))=0$ for any $x\in \wt X$. Otherwise, there is
a nontrivial quotient $E(N)\to L\subset\omega_{\wt X}(x_1+x_2+x)$ by
Serre duality, and thus
$$h^0(\omega_{\wt X}(x_1+x_2+x))\ge h^0(L)\ge N+B,$$
where $B$ is a constant independent of $E$, we choose $N$ such that
$H^1(E(N)(-x_1-x_2-x))=0$ for all GIT-semistable points.

Let $\tau=Tor(E)$, $\sG=E/\tau$, note $h^0(\sG(N))=\wt
P(N)-h^0(\tau)$ and
$$h^0(Q^{\sG}_{r_i(x)})=r_i(x)-h^0(Q^{\tau}_{r_i(x)}),\quad h^0(Q^{\sG})=r-h^0(Q^{\tau})$$ then the
inequality in Lemma \ref{lem2.20} becomes
$$\aligned kh^0(\tau)&\le k h^0(Q^{\tau})+\sum_{x\in
I}\sum_{i=1}^{l_x}d_i(x)h^0(Q^{\tau}_{r_i(x)}) \\&\le
kh^0(Q^{\tau})+\sum_{x\in
I}(a_{l_x+1}(x)-a_1(x))h^0(\tau_x).\endaligned$$ Thus $\tau=Tor(E)$
is supported on $\{x_1,x_2\}$ (since $a_{l_x+1}(x)-a_1(x)<k$) and
$E_{x_1}\oplus E_{x_2}\xrightarrow{q}Q$ induces injection
$\tau_{x_1}\oplus\tau_{x_2}\hookrightarrow Q$.
\end{proof}

\begin{nota}\label{nota2.21} Let $\sH\subset\wt{\sR'}$ be the subscheme
parametrising the generalised parabolic sheaves $E=(E,E_{x_1} \oplus
E_{x_2}\xrightarrow{q}Q)$ satisfying the conditions (1) and (2) at
the end of Lemma \ref{lem2.20}. Then, if ${\wt{\sR'}}^{ss}$ (resp.
${\wt{\sR'}}^s$) denotes the open set of $\wt{\sR'}$ consisting of
the semistable (resp. stable) GPS, then it is clear that we have
open embedding
$${\wt{\sR'}}^{ss}\hookrightarrow
\sH\hookrightarrow \wt{\sR'}.$$
\end{nota}

\begin{prop}\label{prop2.22} Suppose
$(p,\{p_{r_1(x)},...,p_{r_{l_x}(x)}\}_{x\in I},q_s)\in\sH$ is a
point corresponding to a GPS $(E,Q)$. Then $(E,Q)$ is stable (resp.
semistable) iff for any nontrivial subsheaf $F\subset E$ we have
$$\aligned&s(F):=\frac{\ell+kcN}{c(m-N)}(\chi(F(N))\wt P(m)-\wt P(N)\chi(F(m)))+\\
&\sum_{x\in I}\sum^{l_x}_{i=1}d_i(x)(r_i(x)\chi(F(N))-\wt
P(N)h^0(Q_{r_i(x)}^F))\\&+k(r\chi(F(N))-\wt P(N)h^0(Q^F))
<(resp.\,\le)\,0.\endaligned$$
\end{prop}

\begin{proof} The point corresonding to a quotient
$\wt V\otimes\wt\sW\xrightarrow{p}E\to 0$ with
$$\{E_x\twoheadrightarrow
Q_{r_{l_x}(x)}\twoheadrightarrow Q_{r_{l_x-1}(x)}\twoheadrightarrow
\cdots\twoheadrightarrow Q_{r_2(x)}\twoheadrightarrow
Q_{r_1(x)}\twoheadrightarrow 0\}_{x\in I}$$ and $E_{x_1}\oplus
E_{x_2}\xrightarrow{q}Q\to 0$, where $q_s:\wt V\otimes\wt\sW\to
E_{x_1}\oplus E_{x_2}\xrightarrow{q}Q\to 0$ and $p_{r_i(x)}:
V\otimes\sW\xrightarrow{p}E\to E_x\twoheadrightarrow
Q_{r_{l_x}(x)}\twoheadrightarrow\cdots\twoheadrightarrow
Q_{r_i(x)}$. For $F\subset E$ such that $E/F$ is torsion free
outside $\{x_1,x_2\}$, we have the flags of quotient sheaves
$$\{F\twoheadrightarrow F_x\twoheadrightarrow
Q^F_{r_{l_x}(x)}\twoheadrightarrow
Q^F_{r_{l_x-1}(x)}\twoheadrightarrow \cdots\twoheadrightarrow
Q^F_{r_2(x)}\twoheadrightarrow Q^F_{r_1(x)}\twoheadrightarrow
0\}_{x\in I}$$ Let
$n_i^F(x)=h^0(Q^F_{r_i(x)})-h^0(Q^F_{r_{i-1}(x)})$ and $F$ have rank
$(r_1,r_2)$. Then
$$\sum_{x\in I}\sum^{l_x}_{i=1}d_i(x)r_i(x)=\,\,r\sum_{x\in
I}a_{l_x+1}(x) -\sum_{x\in I}\sum^{l_x+1}_{i=1}a_i(x)n_i(x)$$
$$\aligned\sum_{x\in I}\sum^{l_x}_{i=1}d_i(x)h^0(Q^F_{r_i(x)})=&
\,r_1\sum_{x\in I_1}a_{l_x+1}(x)+r_2\sum_{x\in I_2}a_{l_x+1}(x)
\\&-\sum_{x\in I}\sum^{l_x+1}_{i=1}a_i(x)n^F_i(x).\endaligned$$ Thus we have
$$\aligned s(F)&=k\wt P(N)\left(\aligned&\chi(F)-\frac{1}{k}\sum_{x\in
I}\sum^{l_x}_{i=1}d_i(x)h^0(Q^F_{r_i(x)})-h^0(Q^F)\\&-\frac{r(F)}{r}\left(\chi(E)-r-\frac{1}{k}\sum_{x\in
I}\sum^{l_x}_{i=1}d_i(x)r_i(x)\right)\endaligned\right)\\&=k\wt
P(N)\left(par\chi_m(F)-{\rm dim}(Q^F)-r(F)\frac{par\chi_m(E)-{\rm
dim}(Q)}{r(E)}\right).\endaligned$$ $(E,Q)$ is semi-stable (resp.
stable) iff $s(F)\le 0$ (resp. $s(F)<0$) for nontrivial $F\subset E$
such that $E/F$ torsion free outside $\{x_1,x_2\}$.

For any nontrivial subsheaf $F\subset E$, let $\tau$ be the torsion
of $E/F$ and $F'\subset E$ such that $\tau=F'/F$ and $E/F'$ torsion
free. If we write $\tau=\tilde\tau+\tau_{x_1}+\tau_{x_2}+\sum_{x\in
I}\tau_x$, then
$$\aligned s(F)-s(F')=&-k\wt P(N)h^0(\widetilde{\tau})-\wt P(N)\sum_{x\in
I}(k-a_{l_x+1}(x)+a_1(x))h^0(\tau_x) \\&-\wt P(N)\sum_{x\in
I}\sum^{l_x}_{i=1}d_i(x)(h^0(\tau_x)+h^0(Q^F_{r_i(x)})-h^0(Q^{F'}_{r_i(x)}))\\&-k\wt
P(N)(h^0(\tau_{x_1})+h^0(\tau_{x_2})+h^0(Q^F)-h^0(Q^{F'})).\endaligned$$
Since $h^0(\tau_x)+h^0(Q^F_{r_i(x)})-h^0(Q^{F'}_{r_i(x)})\ge 0$ and
$h^0(\tau_{x_1}\oplus\tau_{x_2})+h^0(Q^F)-h^0(Q^{F'})\ge 0$, we have
$s(F)\le s(F')$ and $s(F)<s(F')$ if $\tilde\tau+\sum_{x\in
I}\tau_x\neq 0$. Thus stability of $(E,Q)$ implies $s(F)<0$ for any
nontrivial $F\subset E$.
\end{proof}

\begin{prop}\label{prop2.23} There exist integers $N$ and $M(N)>0$ such
that for $m\ge M(N)$ the following is true. A point
$$(E,Q)=(p,\{p_{r_1(x)},...,p_{r_{l_x}(x)},q_s\}_{x\in I})\in\wt{\sR'}$$ is GIT-stable
(respectively, GIT-semistable) if and only if $(E,Q)$ is a stable
(respectively, semistable) GPS such that $\wt V\to H^0(E(N))$ is an
isomorphism and $(p,\{p_{r_1(x)},...,p_{r_{l_x}(x)},q_s\}_{x\in
I})\in\sH.$\end{prop}

\begin{proof} If $(p,\{p_{r_1(x)},...,p_{r_{l_x}(x)}\}_{x\in
I},q_s)\in \wt{\sR'}$ is GIT-stable (GIT-semistable), by Lemma
\ref{lem2.20}, $\wt V\to H^0(E(N))$ is an isomorphism and
$$(p,\{p_{r_1(x)},...,p_{r_{l_x}(x)},q_s\}_{x\in I})\in\sH.$$ For any
nontrivial subsheaf $F\subset E$ such that $E/F$ is torsion free
outside $\{x_1,x_2\}$, let $H\subset \wt V$ be the inverse image of
$H^0(F(N))$ and $h=dim(H)$, note $h^1(F(N))\ge h^1(F(m))$ when
$m>N$, we have
$$\chi(F(N))\wt P(m)-\wt P(N)\chi(F(m))\le h\wt P(m)-\wt P(N)h^0(F(m)).$$
Thus $s(F)\le e(H)$ since ${\rm dim}\,g(H\otimes W_m)\le h^0(F(m))$
and
$${\rm dim}\,g_{r_i(x)}(H\otimes W_m)\le
h^0(Q^F_{r_i(x)}),\quad {\rm dim}\,g_G(H\otimes W_m)\le h^0(Q^F)$$
(the inequalities are strict when $h=0$). By Proposition
\ref{prop2.19} and Proposition \ref{prop2.22}, $(E,Q)$ is stable
(respectively, semistable) if the point is GIT stable (respectively,
GIT semistable).

There is $N_1>0$ such that for any $N\ge N_1$ the following is true.
For any $\wt V\otimes\wt \sW\xrightarrow{p}E\to 0$ with semistable
GPS $(E,Q)$, the induced map $\wt V\to H^0(E(N))$ is an isomorphism
and $(E,Q)\in \sH$.

Let $H\subset \wt V$ be a nontrivial subspace of ${\rm dim}(H)=h$
and $F\subset E$ be the sheaf such that $F(N)\subset E(N)$ is
generated by $H$. Since all these $F$ are in a bounded family (for
fixed $N$), there is a $M_1(N)$ such that $${\rm dim}\,g(H\otimes
W_m)=h^0(F(m))=\chi(F(m)),\quad {\rm dim}\,g_G(H\otimes
W_m)=h^0(Q^F)$$ and $g_{r_i(x)}(H\otimes W_m)=h^0(Q^F_{r_i(x)})$
($\forall\,\,x\in I$) whenever $m\ge M_1(N)$, which imply that
$$e(H)=s(F)+\frac{\ell+kcm}{c(m-N)}\wt P(N)\left(h-\chi(F(N))\right).$$
If $H^1(F(N))=0$, we have $e(H)\le s(F)$ since $h\le h^0(F(N))$.
Then $e(H)\le s(F)<(resp.\le)\,0$ by Proposition \ref{prop2.22} when
$(E,Q)$ is stable (resp. semistable). If $H^1(F(N))\neq 0$, by Lemma
\ref{lem2.11}, we have
$$h^0(F(N))\le\frac{rcN+\wt\chi}{r}(r(F)-1)+A$$
where $A$ is a constant. Putting $h\le h^0(F(N))$ and above
inequality in
$$\aligned
e(H)=&\wt
P(N)\left(kh-(\ell+kcN)r(F)+(\ell+kcN)\frac{h-\chi(F(N))}{c(m-N)}\right)
\\&-\wt P(N)\sum_{x\in
I}\sum^{l_x}_{i=1}d_i(x)h^0(Q^F_{r_i(x)})-k\wt
P(N)h^0(Q^F),\endaligned$$ then, let $C=k|\chi|+(|A|+|\ell|)r$, we
have
$$e(H)\le \wt P(N)\left(
-kcN+C+(\ell+kcN)\frac{h-\chi(F(N))}{c(m-N)}\right).$$ Choose an
integer $N_2\ge N_1$ such that $-kcN_2+C<-1$. Then, for any fixed
$N\ge N_2$, there is an integer $M_2(N)$ such that for $m\ge M_2(N)$
$$(\ell+kcN)\frac{h-\chi(F(N))}{c(m-N)}<1$$
for any $H\subset V$, which implies $e(H)<0$ and we are done.
\end{proof}

\begin{thm}\label{thm2.24} When $\wt X$ is irreducible, there exists
a (coarse) moduli space $\sP^s$ of stable GPS on $\wt X$, which is a
smooth variety. There is an open immersion
$\sP^s\hookrightarrow\sP$, where $\sP$ is the moduli space of
$s$-equivalence classes of semi-stable GPS on $\wt X$, which is
reduced, irreducible and normal projective variety with at most
rational singularities.
\end{thm}

\begin{proof} Let $\sP^s:={\wt{\sR'}}^s//SL(\wt V)$ and $\sP:={\wt{\sR'}}^{ss}//SL(\wt
V)$ be the GIT quotient. When $(E,Q)$ is a stable GPS, $E$ must be
torsion free. Thus ${\wt{\sR'}}^s$ is a smooth variety, so is
$\sP^s$. By Proposition 3.2 of \cite{S1}, $\sH$ is reduced, normal
with at most rational singularities, so are
${\wt{\sR'}}^{ss}\subset\sH$  and $\sP$.
\end{proof}

The above construction also works for the case when
$\wt X=X_1\sqcup X_2$ is a disjoint union of two irreducible smooth
curves. However, for later applications, we need to use a different quotient
space $\wt \sR$. Let $\chi_1$ and $\chi_2$ be integers such that
$\chi_1+\chi_2-r=\chi$, and fix, for $i=1,2$, the polynomials
$P_i(m)=c_irm+\chi_i$ and $\sW_i=\sO_{X_i}(-N)$ where
$\sO_{X_i}(1)=\sO (1)|_{X_i}$ has degree $c_i$. Write $V_i=\Bbb
C^{P_i(N)}$ and consider the Quot schemes $Quot(V_i\otimes\sW_i,
P_i)$, let $\wt{\textbf{Q}}_i$ be the closure of the open set
$$\mathbf{Q}_i=\left\{\aligned&\text{$V_i\otimes\sW_i\to E_i\to 0$, with
$H^1(E_i(N))=0$ and}\\&\text{$V_i\to H^0(E_i(N))$ induces an
isomorphism}\endaligned\right\},$$ we have the universal quotient
$V_i\otimes\sW_i\to \sF^i\to 0$ on $X_i\times \wt{\textbf{Q}}_i$ and
the relative flag scheme
$$\sR_i=\underset{x\in I_i}{\times_{\wt{\textbf{Q}_i}}}
Flag_{\vec n(x)}(\sF^i_x)\to \wt{\textbf{Q}}_i.$$ Let
$\sF=\sF^1\oplus\sF^2$ denote direct sum of pullbacks of $\sF^1$,
$\sF^2$ on $$\wt X\times
(\wt{\textbf{Q}}_1\times\wt{\textbf{Q}}_2)=(X_1\times\wt{\textbf{Q}}_1)\sqcup(X_2\times\wt{\textbf{Q}}_2).$$
Let $\sE$ be the pullback of $\sF$ to $\wt
X\times(\sR_1\times\sR_2)$, $\wt V=V_1\oplus V_2$ and
$$\rho:\widetilde{\sR'}:=Grass_r(\sE_{x_1}\oplus\sE_{x_2})\to\wt\sR:=\sR_1\times\sR_2\to
\wt{\textbf{Q}}:=\wt{\textbf{Q}}_1\times\wt{\textbf{Q}}_2.$$ Note
that $V_1\otimes\sW_1\oplus V_2\otimes\sW_2\to\sF\to 0$ is a
$\wt{\textbf{Q}}_1\times\wt{\textbf{Q}}_2$-flat quotient with
Hilbert polynomial $\wt P(m)=P_1(m)+P_2(m)$ on $\widetilde X\times
(\wt{\textbf{Q}}_1\times\wt{\textbf{Q}}_2)$, we have for $m$ large
enough a $G$-equivariant embedding
$$\wt{\textbf{Q}}_1\times\wt{\textbf{Q}}_2\hookrightarrow Grass_{\wt P(m)}(V_1\otimes W_1^m
\oplus V_2\otimes W_2^m),$$ where $W_i^m=H^0(\sW_i(m))$ and
$G=(GL(V_1)\times GL(V_2))\cap SL(\wt V).$ Moreover, for large
enough $m$, we have a $G$-equivariant embedding
$$\wt{\sR'}\hookrightarrow \mathbf{G}'=Grass_{\wt P(m)}(\wt V\otimes W_m)\times\bold {Flag}\times Grass_r(\wt V\otimes W_m)$$
(\textbf{Warning }: $\wt V\otimes W_m:=V_1\otimes W_1^m \oplus
V_2\otimes W_2^m$), which maps a point $$(p=p_1\oplus p_2,
\{p_{r_1(x)},...,p_{r_{l_x}(x)}\}_{x\in I},q_s)\in\wt{\sR'},$$ where
$V_i\otimes \sW_i\xrightarrow{p_i} E_i\to 0,$ $(V_1\otimes
\sW_1)\oplus (V_2\otimes \sW_2)\xrightarrow{p=p_1\oplus p_2}
E:=E_1\oplus E_2$ denotes the quotient on $\wt X=X_1\sqcup X_2$ and
$$\{\,\,(V_1\otimes\sW_1)\oplus (V_2\otimes
\sW_2)\xrightarrow{p_{r_i(x)}}Q_{r_i(x)}\to 0,\,\,1\le i\le
l_x\,\,\}_{x\in I},$$ $(V_1\otimes \sW_1)\oplus (V_2\otimes
\sW_2)\xrightarrow{q_s}Q$ denotes the surjection of sheaves
$$q_s:(V_1\otimes \sW_1)\oplus (V_2\otimes \sW_2)\to E_{x_1}\oplus
E_{x_2}\xrightarrow{q}Q\to 0,$$ to the point
$(g,\{g_{r_1(x)},...,g_{r_{l_x}(x)}\}_{x\in I},g_G)=(\wt V\otimes
W_m\xrightarrow{g}U,$
$$\{\wt V\otimes
W_m\xrightarrow{g_{r_1(x)}}U_{r_1(x)},\,\cdots\,, \wt V\otimes
W_m\xrightarrow{g_{r_{l_x}(x)}}U_{r_{l_x}(x)}\}_{x\in I},\wt
V\otimes W_m\xrightarrow{g_G}U_r)$$ of $\mathbf{G}'$, where
$g:=H^0(p(m))$, $U:=H^0(E(m))$, $g_{r_i(x)}:=H^0(p_{r_i(x)}(m))$,
$U_{r_i(x)}:=H^0(Q_{r_i(x)})$ ($i=1,...,l_x$), $g_G:=H^0(q_s(m))$,
$U_r:=H^0(Q)$ and $r_i(x)=dim(Q_{r_i(x)})$. Given $\mathbf{G}'$ the
polarisation
$$\frac{\ell+kcN}{c(m-N)}\times\prod_{x\in
I} \{d_1(x),\cdots,d_{l_x}(x)\}\times k.$$
Then we have criterion (see Proposition 1.14 and 2.4 of \cite{Bh})
\begin{prop} A point
$(g,\{g_{r_1(x)},...,g_{r_{l_x}(x)}\}_{x\in I},g_G)\in \bold G'$ is
stable (semistable) for the action of $G$, with respect to the above
polarisation, iff for all nontrivial subspaces $H\subset \wt V$,
where $H=H_1\oplus H_2$, $H_i\subset V_i$ ($i=1,2$), we have (with
$h=dim H$, $\wt H:=H_1\otimes W^m_1\oplus H_2\otimes W_2^m$)
$$\aligned e(H):=&\frac{\ell+kcN}{c(m-N)}\left(\wt{P}(m)h-\wt{P}(N)dim\,g(\wt
H)\right)\\&+\sum_{x\in
I}\sum^{l_x}_{i=1}d_i(x)\left(r_i(x)h-\wt{P}(N)dim\,g_{r_i(x)}(\wt
H)\right)\\&+k\left(rh-\wt P(N)dim\,g_G(\wt
H)\right)<(\le)\,0.\endaligned$$
\end{prop}

The Lemma \ref{lem2.20} and Proposition \ref{prop2.22} (thus Proposition \ref{prop2.23})
are also true for the case $\wt X=X_1\sqcup X_2$. Thus we have

\begin{thm}\label{thm2.26} When $\wt X=X_1\sqcup X_2$, there exists
a (coarse) moduli space $\sP^s$ of stable GPS on $\wt X$, which is a
smooth scheme. There is an open immersion
$\sP^s\hookrightarrow\sP$, where $\sP$ is the moduli space of
$s$-equivalence classes of semi-stable GPS on $\wt X$, which is a disjoint union of at most $r+1$
irreducible, normal projective varieties with at most
rational singularities.
\end{thm}

\begin{proof} For any $\chi_1$ and $\chi_2$ satisfying $\chi_1+\chi_2=\chi+r$ and
$$n^{\omega}_1\le\chi_1\le n^{\omega}_1+r,\quad n^{\omega}_2\le\chi_2\le n^{\omega}_2+r,$$
let $\,\,\sP_{\chi_1,\,\chi_2}^s:=\wt{\sR'}^s//G\,$, $\,\,\sP_{\chi_1,\,\chi_2}:=\wt{\sR'}^{ss}//G\,\,$ and
$$\sP^s:=\bigsqcup_{\chi_1+\chi_2=\chi+r}\sP_{\chi_1,\,\chi_2}^s,\quad \sP:=\bigsqcup_{\chi_1+\chi_2=\chi+r}\sP_{\chi_1,\,\chi_2}.$$
Then $\sP_{\chi_1,\,\chi_2}^s$ are smooth varieties and $\sP_{\chi_1,\,\chi_2}$ are reduced, irreducible and normal projective varieties with at most
rational singularities.
\end{proof}

\section{Factorization of generalized theta functions}

The moduli spaces $\sU_X:=\sU_X(r,d,\sO(1), \{k,\vec n(x),\vec
a(x)\}_{x\in I})$ is independent of the choice of $\sO(1)$ when
$X$ is irreducible. However, when $X=X_1\cup X_2$, the moduli spaces $\sU_X:=\sU_X(r,d,\sO(1), \{k,\vec n(x),\vec
a(x)\}_{x\in I})$
depends on the choice of $\sO(1)$ (more precisely, it
only depends on the degree $c_i$ of $\sO(1)|_{X_i}$). We will
require in this section that \ga{3.1}
{\text{$\ell:=\frac{k\chi-\sum_{x\in
I}\sum^{l_x}_{i=1}d_i(x)r_i(x)}{r}$ is an integer.}}

When $X$ is irreducible, for any divisor $L=\sum_q\ell_qz_q$ of
degree $\ell$ on $X$ (supported on smooth points), there is an ample
line bundle $$\Theta_{\sU_X,\,L}=\Theta(r,d,\{k,\vec
n(x),\vec a(x)\}_{x\in I},L)$$ on $\sU_X$, which is called a theta line bundle
on $\sU_X$. We are going to define it as follows.

By a family of parabolic sheaves of rank $r$ and Euler
characteristic $\chi$ with parabolic structures of type $\{\vec
n(x)\}_{x\in I}$ and weights $\{\vec a(x)\}_{x\in I}$ at points
$\{x\}_{x\in I}$ parametrized by $T$, we mean a sheaf $\sF$ on
$X\times T$, flat over $T$, and torsion free with rank $r$ and Euler
characteristic $\chi$ on $X\times\{t\}$ for every $t\in T$, together
with, for each $x\in I$, a flag
$$\sF_{\{x\}\times T}=\sQ_{\{x\}\times T,l_x+1}\twoheadrightarrow\sQ_{\{x\}\times T,l_x}\twoheadrightarrow \sQ_{\{x\}\times T,l_x-1}
\twoheadrightarrow\cdots\twoheadrightarrow \sQ_{\{x\}\times
T,1}\twoheadrightarrow0$$ of quotients of type $\vec n(x)$ and
weights $\vec a(x)$. We define $\Theta_{\sF,\,L}$ to be
$$({\rm det}R\pi_T\sF)^{-k}\otimes\bigotimes_{x\in I}
\lbrace\bigotimes^{l_x}_{i=1} {\rm det}(\sQ_{\{x\}\times
T,i})^{d_i(x)}\rbrace\otimes\bigotimes_q{\rm det}(\sF_{\{z_q\}\times
T})^{\ell_q}$$ where $\pi_T$ is the projection $X\times T\to T$ and
${\rm det}\,R\pi_T\sF$ is the determinant of cohomology: $\{{\rm
det}\,R\pi_T\sF\}_t:={\rm det}\,H^0(X,\sF_t)\otimes{\rm
det}\,H^1(X,\sF_t)^{-1}$. We have the following theorem (see
\cite{NR} for $r=2$ and \cite{Pa} for $r>2$):

\begin{thm}\label{thm3.1} Let $X$ be irreducible and $L=\sum_q\ell_qz_q$ a divisor
of degree $\ell$ supported on smooth points of $X$. Then there is an
unique ample line bundle $\Theta_{\sU_X,\,L}=\Theta(r,d,\{k,\vec n(x), \vec
a(x)\}_{x\in I},L)$ on $\sU_X$ such that \begin{itemize}
\item[(1)]
for any family of parabolic sheaf $\sF$ of rank $r$ and degree $d$
parametrised by $T$, with parabolic structures of type $\{\vec
n(x)\}_{x\in I}$ at points $\{x\}_{x\in I}$, semistable with respect
to the weights $\{\vec a(x)\}_{x\in I}$, we have
$\phi_T^*\Theta_{\sU_X,\,L}=\Theta_{\sF,\,L}$, where
$\phi_T:T\to\sU_X$ is the morphism induced by $\sF$.
\item[(2)] for any two choices $L$ and $L'$,
$\Theta_{\sU_X,\,L}$ and $\Theta_{\sU_X,\,L'}$ are algebraically
equivalent.
\end{itemize}
\end{thm}

\begin{proof} (1) Let $\sE$ be the universal family on $X\times
\sR^{ss}$, then the line bundle $\Theta_{\sE,\,L}$ on $\sR^{ss}$,
which was defined as $$({\rm
det}R\pi_{\sR^{ss}}\sE)^{-k}\otimes\bigotimes_{x\in I}
\lbrace\bigotimes^{l_x}_{i=1} {\rm det}(\sQ_{\{x\}\times
\sR^{ss},i})^{d_i(x)}\rbrace\otimes\bigotimes_q{\rm
det}(\sE_{\{z_q\}\times \sR^{ss}})^{\ell_q},$$ descends to the line
bundle $\Theta_{\sU_X,\,L}$ on $\sU_X$ (see \cite{Pa} for the
detail).

(2) Let $X^0\subset X$ be the open set of smooth points and
$L_0=L-z$, where $z$ is a point in the support of $L$. It is enough
to show that $\Theta_{\sU_X,\,L}$ is algebraically equivalent to
$\Theta_{\sU_X,\,L_0+y}$ for any $y\in X^0$. To prove it, note that
$X^0\times\sR^{ss}\to X^0\times\sU_X$ is a good quotient and the
line bundle
$$\pi_{\sR^{ss}}^*(\Theta_{\sE,\,L}\otimes {\rm det}(\sE_z)^{-1})\otimes {\rm
det}(\sE)$$ descends to a line bundle $\sL$ on $X^0\times\sU_X$ such
that $$\sL|_{\{z\}\times\sU_X}=\Theta_{\sU_X,\,L}\,,\quad
\sL|_{\{y\}\times\sU_X}=\Theta_{\sU_X,\,L_0+y}$$ i.e.
$\Theta_{\sU_X,\,L}$ and $\Theta_{\sU_X,\,L_0+y}$ are algebraically
equivalent.

The ampleness of $\Theta_{\sU_X,\,L}$ follows the ampleness of
$\Theta_{\sU_X,\,\ell\cdot y}$, which is the descendant of
restriction (on $\sR^{ss}$) of the polarization (Notation
\ref{nota2.5}) if we choose $\sO(1)=\sO(cy)$.
\end{proof}

When $X=X_1\cup X_2$, we choose $\sO(1)=\sO_X(c_1y_1+c_2y_2)$ such
that \ga{3.2}{\ell_i=\frac{c_i\ell}{c_1+c_2}\,\,(i=1,2)\,\, are\,\,
integers.}
Then the following theorem can be proven similarly (see
\cite{S2} for the detail).

\begin{thm}\label{thm3.2} Let $X=X_1\cup X_2$ and $L_i=\sum_{q\in
X_i}\ell_qz_q$ be a divisor of degree $\ell_i$ supported on
$X_i\setminus\{x_0\}$. Then there is an unique ample line bundle
$\Theta_{\sU_X,\,L_1+L_2}=\Theta(r,d,\{k,\vec n(x),\vec a(x)\}_{x\in
I_1\cup I_2},L_1+L_2)$ on $\sU_X$ such that
\begin{itemize}
\item[(1)]
for any family of parabolic sheaf $\sF$ of rank $r$ and degree $d$
parametrised by $T$, with parabolic structures of type $\{\vec
n(x)\}_{x\in I}$ at points $\{x\}_{x\in I}$, semistable with respect
to the weights $\{\vec a(x)\}_{x\in I}$, we have
$\phi_T^*\Theta_{\sU_X,\,L_1+L_2}=\Theta_{\sF,\,L_1+L_2}$, where
$\phi_T:T\to\sU_X$ is the morphism induced by $\sF$.
\item[(2)] for any two choices $L_1+L_2$, $L'_1+L'_2$,
$\Theta_{\sU_X,\,L_1+L_2}$ and $\Theta_{\sU_X,\,L'_1+L'_2}$ are
algebraically equivalent.
\end{itemize}
\end{thm}

\begin{rmks}\label{rmks3.3} (1) When $X$ is irreducible, the map $E\mapsto
E\otimes\sO_X(\pm y)$ induces an isomorphism ($\ell\mapsto \ell\pm
k$)
$$f:\sU_X(r,d,
\{k,\vec n(x),\vec a(x)\}_{x\in I})\to \sU_X(r,d\pm r, \{k,\vec
n(x),\vec a(x)\}_{x\in I})$$ such that $\Theta_{\sU_X,\,L\pm
ky}=f^*\Theta_{\sU_X,\,L}$ for the divisor $L=\sum_q\ell_qz_q$ of
degree $\ell$.

(2) If $\ell\neq 0$, for any $L=\sum_{q\in X^0}\ell_qz_q$ of degree
$\ell$, then $\Theta_{\sU_X,\,L}$ is the descendant of restriction
(on $\sR^{ss}$) of the polarization (Notation \ref{nota2.5}) if we
choose $\sO(1)=\sO(\sum_q\frac{|\ell|\ell_q}{\ell}z_q)$ where
$c=|\ell|$.

\end{rmks}

In the rest of this paper, we will fix a smooth point $y\in X$ (and
$y_i\in X_i$ when $X$ is reducible), and choose
$$L=\ell_yy+\sum_{x\in I}\alpha_xx,\qquad L_i=\ell_{y_i}y_i+\sum_{x\in I_i}\alpha_xx\quad(i=1,\,2).$$
This choice determines, when $X$ is irreducible, the theta line
bundle
$$\Theta_{\sU_X}=\Theta(r,d,\{k,\vec n(x),\vec
a(x),\alpha_x\}_{x\in I},\ell_y)$$ where
$\ell_y+\sum_{x\in I}\alpha_x=\ell$, and it determines, when $X$ is
reducible,
$$\Theta_{\sU_X}=\Theta(r,d,\{k,\vec n(x),\vec a(x),\alpha_x\}_{x\in I_1\cup I_2},\ell_{y_1},\ell_{y_2})$$ where
$\ell_{y_i}+\sum_{x\in I_i} \alpha_x=\ell_i$ ($i=1,\,2$).

Now we are going to state the factorizations proved in \cite{S1} and
\cite{S2}. Firstly, let $X$ be an irreducible projective curve of
genus $g$, smooth but for one node $x_0$. Let $\pi:\widetilde{X}\to
X$ be the normalization of $X$, and $\pi^{-1}(x_0)=\{x_1,x_2\}$. Let
$I$ be a finite set of smooth points on $X$ and $y\in X$ be a fixed
smooth point. Given integers $d,$ $k,$ $r$, $\{\alpha_x\}_{x\in I}$,
$\ell_y$,
$$\aligned
\vec a(x)&=(a_1(x),a_2(x),\cdots,a_{l_x+1}(x))\\
\vec n(x)&=(n_1(x),n_2(x),\cdots,n_{l_x+1}(x))\endaligned$$
satisfying $\ell_y+\sum_{x\in I}\alpha_x=\ell$ and
$$0\leq a_1(x)<a_2(x)<\cdots<a_{l_x+1}(x)<k\quad (x\in
I).$$ Recall that $\ell $ is defined by \ga{3.3}{\sum_{x\in
I}\sum^{l_x}_{i=1}d_i(x)r_i(x)+r\ell=k(d+r(1-g))=k\chi} where
$d_i(x)=a_{i+1}(x)-a_i(x)$ and $r_i(x)=n_1(x)+\cdots+n_i(x).$

Let $\sU_X$ be the moduli space of ($s$-equivalence classes of)
parabolic torsion free sheaves of rank $r$ and degree $d$ on $X$,
with parabolic structures of type $\{\vec n(x)\}_{x\in I}$ at points
$\{x\}_{x\in I},$ semistable with respect to the weights $\{\vec
a(x)\}_{x\in I}$.

For $\mu=(\mu_1,\cdots,\mu_r)$  with $0\le\mu_r\le\cdots\le\mu_1\le
k-1,$ let $$\{d_i=\mu_{r_i}-\mu_{r_i+1}\}_{1\le i\le l}$$ be the
subset of nonzero integers in
$\{\mu_i-\mu_{i+1}\}_{i=1,\cdots,r-1}.$ We define
$$ r_i(x_1)=r_i,\quad d_i(x_1)=d_i,\quad l_{x_1}=l,\quad \alpha_{x_1}=\mu_r$$
$$ r_i(x_2)=r-r_{l-i+1},\quad d_i(x_2)=d_{l-i+1},\quad l_{x_2}=l,\quad
\alpha_{x_2}=k-\mu_1$$ and for $j=1,2$, we set
$$\aligned
\vec a(x_j)&=\left(\mu_r,\mu_r+d_1(x_j),\cdots,\mu_r+
\sum^{l_{x_j}-1}_{i=1}d_i(x_j),\mu_r+\sum^{l_{x_j}}_{i=1}d_i(x_j)\right)\\
\vec n(x_j)&=(r_1(x_j),r_2(x_j)-r_1(x_j),
\cdots,r_{l_{x_j}}(x_j)-r_{l_{x_j}-1}(x_j),r-r_{l_{x_j}}(x_j)).\endaligned$$

Let $\sU^{\mu}_{\widetilde{X}}$ be the moduli space of semistable
parabolic bundles on $\widetilde{X}$ with parabolic structures of
type $\{\vec n(x)\}_{x\in I\cup\{x_1,x_2\}}$ at points $\{x\}_{x\in
I\cup\{x_1,x_2\}}$ and weights $\{\vec a(x)\}_{x\in
I\cup\{x_1,x_2\}}$, and let
$$\Theta_{\sU^{\mu}_{\widetilde{X}}}=\Theta(r,d,\{k,\vec n(x),\vec a(x),\alpha_x\}_{x\in I\cup\{x_1,x_2\}},\ell_y).$$
Then the following is the so called \textbf{Factorization Theorem I}

\begin{thm}\label{thm3.4} There exists a (noncanonical)
isomorphism
$$H^0(\sU_X,\Theta_{\sU_X})\cong\bigoplus_{\mu}H^0(\sU^{\mu}_{\widetilde{X}},\Theta_{\sU^{\mu}_{\widetilde{X}}})$$
where $\mu=(\mu_1,\cdots,\mu_r)$ runs through $0\le\mu_r\le\cdots\le
\mu_1\le k-1.$
\end{thm}

When $X=X_1\cup X_2$, $I=I_1\cup I_2$,
$\widetilde{X}=X_1\sqcup X_2$ is the disjoint union of smooth
projective curves $X_1$ and $X_2$. Recall that
$$\Theta_{\sU_X}=\Theta(r,d,\{k,\vec n(x),\vec a(x),\alpha_x\}_{x\in I_1\cup I_2},\ell_{y_1},\ell_{y_2}),$$ where
$\ell_{y_i}+\sum_{x\in I_i} \alpha_x=\ell_i$ ($i=1,\,2$), are the theta line bundles on $$\sU_X=\sU_X(r,d,\sO(1),\omega).$$
For $\mu=(\mu_1,\cdots,\mu_r)$ with $0\le\mu_r\le\cdots\le\mu_1\le
k-1,$ we define
$$\aligned&\chi_1^{\mu}=\frac{1}{k}\left(r\ell_1+\sum_{x\in I_1}\sum^{l_x}_{i=1}
d_i(x)r_i(x)\right)+\frac{1}{k}\sum^r_{i=1}
\mu_i=n^{\omega}_1+\frac{1}{k}\sum^r_{i=1}
\mu_i\\
&\chi_2^{\mu}=\frac{1}{k}\left(r\ell_2+\sum_{x\in I_2}\sum^{l_x}_{i=1}
d_i(x)r_i(x)\right)+r-\frac{1}{k}
\sum^r_{i=1}\mu_i=n^{\omega}_2+r-\frac{1}{k}
\sum^r_{i=1}\mu_i.\endaligned$$ One can check that the numbers
satisfy ($j=1,2$) \ga{3.4}{\sum_{x\in
I_j\cup\{x_j\}}\sum^{l_x}_{i=1}d_i(x)r_i(x)+r\sum_{x\in
I_j\cup\{x_j\}}\alpha_x+r\ell_{y_j}=k\chi_j^{\mu}.}

Let $\omega_j^{\mu}=\{k, \vec n(x),\vec a(x)\}_{x\in I_j\cup\{x_j\}}$ ($j=1,2$), $d_j^{\mu}=\chi_j^{\mu}+r(g_j-1)$ and
$$\sU^{\mu}_{X_j}:=\sU_{X_j}(r, d_j^{\mu},\omega_j^{\mu})$$ be
the moduli space of $s$-equivalence classes of semistable parabolic
bundles $E$ of rank $r$ on $X_j$ and $\chi(E)=\chi_j^{\mu}$,
together with parabolic structures of type $\{\vec n(x)\}_{x\in
I\cup\{x_j\}}$ and weights $\{\vec a(x)\}_{x\in I\cup\{x_j\}}$ at
points $\{x\}_{x\in I\cup\{x_j\}}$. We define $\sU^{\mu}_{X_j}$ to
be empty if $\chi_j^{\mu}$ is not an integer. Let
$$\Theta_{\sU^{\mu}_{X_j}}=
\Theta(r,d_j^{\mu},\{k,\vec n(x),\vec
a(x),\alpha_x\}_{x\in I_j\cup\{x_j\}},\ell_{y_j})$$
then we have \textbf{Factorization
Theorem II}
\begin{thm}\label{thm3.5} There exists a (noncanonical) isomorphism
 $$H^0(\sU_{X_1\cup X_2},\Theta_{\sU_{X_1\cup X_2}})\cong\bigoplus_{\mu}
H^0(\sU^{\mu}_{X_1},\Theta_{\sU^{\mu}_{X_1}})\otimes
H^0(\sU^{\mu}_{X_2},\Theta_{\sU^{\mu}_{X_2}})$$ where
$\mu=(\mu_1,\cdots,\mu_r)$ runs through $0\le\mu_r\le\cdots\le
\mu_1\le k-1.$
\end{thm}

\section{Invariance of spaces of generalized theta functions}

For a smooth projective curve $C$ of genus $g\ge 0$ and a finite set $I_1\subset C$ of points, to compute the
dimension of $H^0(\sU_C,\Theta_{\sU_C})$, we take a family $\{(X_t, I_t)\}_{t\in T}$ of curves with parabolic data such that
$$(X_1,I_1)=(C,I_1)$$ is the curve $C$ with given parabolic data and $(X_0,I_0)=(X, I)$ is an curve $X$ with one node and
parabolic data. If dimension of the spaces $H^0(\sU_{X_t},\Theta_{\sU_{X_t}})$ is invariant, we can reduce, by using \textbf{Factorization Theorem I},
the computation of dimension for
a genus $g$ curve to the computation of dimension for a genus $g-1$ curve. Then, by the same procedure and using
\textbf{Factorization Theorem II}, we can decrease the number of parabolic points.

In order to prove the invariance, we proved in \cite{S1} that $$H^1(\sU_X,\Theta_{\sU_X})=0$$ when $X$ is an irreducible curve of $g\ge 3$
with at most one node (which implies the invariance for $g\ge 3$). We recall in this section the proof of vanishing theorem for smooth curves and
remark that our arguments in \cite{S1} in fact imply the invariance for any smooth curves $X_t:=\widetilde{X}$.

Let $\widetilde{X}$ be a smooth projective curve of genus $\widetilde{g}$. Fix a line bundle $\sO(1)$ on $\widetilde{X}$ of
${\rm deg}(\sO(1))=c$, let $\widetilde{\chi}=d+r(1-\widetilde{g})$, $\widetilde{P}$ denote the polynomial
$\widetilde{P}(m)=crm+\widetilde{\chi}$, $\sO_{\widetilde{X}}(-N)=\sO(1)^{-N}$ and $V=\Bbb C^{\widetilde{P}(N)}$.
Let $\widetilde{\bold Q}$ be the Quot scheme of quotients $$V\otimes\sO_{\widetilde{X}}(-N)\to F\to 0$$ (of rank
$r$ and degree $d$) on $\widetilde{X}$. Thus there is on $\widetilde{X}\times\widetilde{\bold Q}$ a universal quotient
$V\otimes\sO_{\widetilde{X}\times\widetilde{\bold Q}}(-N)\to \sF\to 0$.
Let $\sF_x$ be the sheaf given by restricting $\sF$ to $\{x\}\times\widetilde{\bold Q}$,
$Flag_{\vec n(x)}(\sF_x)\to \widetilde{\bold Q}$ be the relative flag scheme of type $\vec n(x)$ and
$$\widetilde{\sR}=\underset{x\in I}{\times_{\widetilde{\bold Q}}}Flag_{\vec n(x)}(\sF_x)\to \widetilde{\bold Q}.$$
Let $\widetilde{\sR}_F$ denote open set of locally free quotients and
$$V\otimes\sO_{\widetilde{X}\times\widetilde{\sR}}(-N)\to
\widetilde{\sF}\to 0$$
denote pullback of the universal quotient
$V\otimes\sO_{\widetilde{X}\times\widetilde{\bold Q}}(-N)\to \sF\to
0$. The reductive group ${\rm SL}(V)$ acts on $\widetilde{\sR}$.

For large enough $m$, we have a ${\rm SL}(V)$-equivariant embedding
$$\widetilde{\sR}\hookrightarrow \mathbf{G}=Grass_{\widetilde{P}(m)}(V\otimes W_m)\times\bold {Flag},$$
where $W_m={\rm H}^0(\sO_{\widetilde{X}}(m))$, and $\bold{Flag}$ is
defined to be
$$\bold{Flag}=\prod_{x\in I}\{Grass_{r_1(x)}(V\otimes W_m)
\times\cdots\times Grass_{r_{l_x}(x)}(V\otimes W_m)\}.$$
For any given data $\omega=\{k,\vec n(x),\vec a(x)\}_{x\in I}$, $\tilde{\ell}$ is defined by
 \ga{4.1}{\sum_{x\in
I}\sum^{l_x}_{i=1}d_i(x)r_i(x)+r\tilde{\ell}=k(d+r(1-\tilde
g)=k\widetilde{\chi},} $\omega$ determines a polarisation (for fixed
$\sO(1)$) on $\mathbf{G}$:
$$\frac{\wt\ell+kcN}{c(m-N)}\times\prod_{x\in
I}\{d_1(x),\cdots, d_{l_x}(x)\}.$$

The set $\widetilde{\sR}_{\omega}^{ss}\subset\widetilde{\sR}_F$ of
GIT semistable (resp. stable) points for the ${\rm SL}(V)$ action
under this polarisation is precisely the set of semistable (resp.
stable) parabolic bundles on $\widetilde{X}$ of the type determined
by the given data. Its good quotient $\sU_{\widetilde{X},\,\omega}$
is our moduli space and
$$\Theta_{\widetilde{\sR}^{ss}_{\omega}}=
(det\,R\pi_{\widetilde{\sR}^{ss}_{\omega}}\widetilde{\sF})^{-k}\otimes\bigotimes_{x\in
I}
\{(det\,\widetilde{\sF}_x)^{\alpha_x}\otimes\bigotimes^{l_x}_{i=1}
(det\,\sQ_{x,i})^{d_i(x)}\}\otimes
(det\,\widetilde{\sF}_y)^{\tilde\ell_y}$$ where
$\tilde\ell_y+\sum_{x\in I}\alpha_x=\tilde\ell$, descends to an
ample line bundle $\Theta_{\sU_{\widetilde{X},\,\omega}}$ on
$\sU_{\widetilde{X},\,\omega}$. To prove
$H^1(\sU_{\widetilde{X},\,\omega},\Theta_{\sU_{\widetilde{X},\,\omega}})=0$,
we need essentially the following codimension estimates:

\begin{prop}[Proposition 5.1 of \cite{S1}]\label{prop4.1} Let $|{\rm I}|$ be the number
of parabolic points. Then
\begin{itemize}
 \item[(1)]  $\,\,{\rm codim}(\widetilde{\sR}^{ss}\setminus\widetilde{\sR}^s)\ge (r-1)(\tilde g-1)+\frac{1}{k}|{\rm
 I}|$,
\item[(2)]  $\,\,{\rm codim} (\widetilde{\sR}_F\setminus\widetilde{\sR}^{ss})>(r-1)(\tilde g-1)+\frac{1}{k}|{\rm
I}|$.
\end{itemize}
\end{prop}

\begin{prop}[Proposition 2.2 of \cite{S1}]\label{prop4.2} Let
$\omega_{\widetilde{X}}=\sO_{\widetilde{X}}(\sum q)$ and
$\omega_{\widetilde{\sR}_F}$ be the canonical sheaf of
$\widetilde{X}$ and $\widetilde{\sR}_F$ respectively. Then
$$\aligned\omega^{-1}_{\wt\sR_F}=&(det\,R\pi_{\wt\sR_F}\widetilde{\sF})^{-2r}\otimes
\bigotimes_{x\in I}\left\{(det\,\widetilde{\sF}_x)^{n_{l_x+1}-r}
\otimes\bigotimes^{l_x}_{i=1}(det\,\sQ_{x,i})
^{n_i(x)+n_{i+1}(x)}\right\}\\
&\otimes\bigotimes_q(det\,\widetilde{\sF}_q)^{1-r}\otimes(det\,\wt\sF_y)^{2\wt\chi+(r-1)(2\tilde g-2)}\otimes{\rm Det}^*(\Theta_y^{-2})
\endaligned$$
where ${\rm Det}: \wt\sR_F\to J^d_{\wt X}$ is the determinant morphism and $\Theta_y$ is the theta line bundle on $J^d_{\wt X}$.
\end{prop}

The following result due to F. Knop is essential in our arguments, whose global form was formulated
in \cite{NR}.

\begin{lem}[Lemma 4.17 of \cite{NR}]\label{lem4.3} Let $X$ be a normal, Cohen-Macaulay varirty on which a reductive group
$G$ acts, such that a good quotient $\pi:X\to Y$ exists. Suppose that the action is generically free and
$dim\,G=dim\,X-dim\,Y$. Suppose further that \begin{itemize}
\item[(1)] the subset where the action is not free has codimension $\ge 2$,
\item[(2)] for every prime divisor $D$ in $X$, $\pi(D)$ has codimension $\le 1$, where $D$ need not be invariant.
\end{itemize} Then $\omega_Y=(\pi_*\omega_X)^G$ where $\omega_X$, $\omega_Y$ are the respective dualizing sheaves.
\end{lem}

\begin{thm}[Theorem 5.1 of \cite{S1}]\label{thm4.4} Assume $(r-1)(\tilde g-1)+\frac{1}{k}|{\rm
 I}|\ge 2$. Then, for any data $\omega$ such that $\tilde{\ell}\in\mathbb{Z}$, we
 have $$H^1(\sU_{\widetilde{X},\,\omega},\Theta_{\sU_{\widetilde{X},\,\omega}})=0.$$
\end{thm}

\begin{proof} Note that, on good quotient
$\sU_{\widetilde{X},\,\omega}$, we always have for any $i\ge 0$
$$H^i(\sU_{\widetilde{X},\,\omega},\Theta_{\sU_{\widetilde{X},\,\omega}})
=H^i(\widetilde{\sR}_{\omega}^{ss},\Theta_{\widetilde{\sR}^{ss}_{\omega}})^{inv}.$$
By the assumption and Proposition \ref{prop4.1}, we have ${\rm
codim} (\widetilde{\sR}_F\setminus\widetilde{\sR}_{\omega}^{ss})>2.$
Thus
$H^1(\widetilde{\sR}_{\omega}^{ss},\Theta_{\widetilde{\sR}^{ss}_{\omega}})^{inv}
=H^1(\widetilde{\sR}_F,\Theta_{\widetilde{\sR}_F})^{inv}$, where
$$\Theta_{\widetilde{\sR}_F}=
(det\,R\pi_{\widetilde{\sR}_{F}}\widetilde{\sF})^{-k}\otimes\bigotimes_{x\in
I}
\{(det\,\widetilde{\sF}_x)^{\alpha_x}\otimes\bigotimes^{l_x}_{i=1}
(det\,\sQ_{x,i})^{d_i(x)}\}\otimes
(det\,\widetilde{\sF}_y)^{\tilde\ell_y}$$ with
$\tilde\ell_y+\sum_{x\in I}\alpha_x=\tilde\ell$. Let $J=J^d_{\wt X}$
be the Jacobian of line bundles of degree $d$ on $\wt X$, $\sL$ the
universal line bundle on $\wt X\times J$ and
$$\Theta_y={\rm det}(R\pi_J\sL)^{-1}\otimes\sL_y^{d+1-\tilde{g}}.$$
The line bundle ${\rm det}(\wt \sF)$ on $\wt X\times\wt\sR_F$
induces (for any data $\bar\omega$)
$${\rm Det}:\wt\sR_F\to
J\,\,,\quad {\rm Det}: \sU_{\wt X\,,\bar\omega}\to J$$
such that ${\rm det}\,R\pi_{\wt\sR_F}det\widetilde{\sF}={\rm Det}^*({\rm det}(R\pi_J\sL))$. Then we can write
$$\Theta_{\wt\sR_F}\otimes\omega_{\wt\sR_F}^{-1}=\hat{\Theta}_{\bar\omega}\otimes{\rm
Det}^*(\Theta_y)^{-2}$$
$$\aligned\hat\Theta_{\bar\omega}=&(det\,R\pi_{\wt\sR_F}\widetilde{\sF})^{-\bar k}\otimes
\bigotimes_{x\in I}\left\{(det\,\widetilde{\sF}_x)^{\bar\alpha_x}
\otimes\bigotimes^{l_x}_{i=1}(det\,\sQ_{x,i})
^{\bar d_i(x)}\right\}\\
&\otimes
(det\wt\sF_y)^{\bar\ell_y}\otimes\bigotimes_q(det\,\widetilde{\sF}_q)^{1-r}\otimes
(det\wt\sF_y)^{(r-1)(2\tilde g-2)}
\endaligned$$
where $\bar k=k+2r$, $\,\bar\alpha_x=\alpha_x+n_{l_x+1}(x)-r$, $\,\bar\ell_y=2\wt\chi+\tilde\ell_y\,$ and
$$\bar{d}_i(x)=d_i(x)+n_i(x)+n_{i+1}(x).$$
Let $\bar\omega=\{\bar k, \vec n(x),\vec {\bar a}(x)\}_{x\in I}$ with $\vec {\bar a}(x)=(\bar a_1(x), \bar a_2(x),\cdots,\bar a_{l_x+1}(x))$ such
that $\bar d_i(x)=\bar a_{i+1}(x)-\bar a_i(x)$ ($i=1,\,2,\,\cdots,\,l_x$). Let
$$\psi_{\bar\omega}:\wt\sR^{ss}_{\bar\omega}\to \wt\sR^{ss}_{\bar\omega}//{\rm SL}(V):=\sU_{\wt X}(r,d,\bar\omega)=\sU_{\wt X,\,\bar\omega},$$
there is an ample line bundle $\Theta_{\bar\omega}$ on $\sU_{\wt X,\,\bar\omega}$ such that $\hat\Theta_{\bar\omega}=\psi_{\bar\omega}^*\Theta_{\bar\omega}$ since
$$\bar\ell:=\frac{\bar k\wt\chi-\sum_{x\in I}\sum^{l_x}_{i=1}\bar d_i(x)r_i(x)}{r}=\tilde\ell+2\wt\chi-r|I|+\sum_{x\in I}n_{l_x+1}(x)$$
is an integer.
Then we have
$\Theta_{\wt\sR^{ss}_{\bar\omega}}=\psi_{\bar\omega}^*(\Theta_{\bar\omega}\otimes{\rm
Det}^*(\Theta_y)^{-2})\otimes\omega_{\wt\sR^{ss}_{\bar\omega}}$
and
$$({\psi_{\bar\omega}}_*\Theta_{\wt\sR^{ss}_{\bar\omega}})^{inv}=(\Theta_{\bar\omega}\otimes{\rm
Det}^*(\Theta_y)^{-2})\otimes({\psi_{\bar\omega}}_*\omega_{\wt\sR^{ss}_{\bar\omega}})^{inv}.$$
Since ${\rm
codim}(\widetilde{\sR}^{ss}_{\bar\omega}\setminus\widetilde{\sR}^s_{\bar\omega})\ge
2$, conditions in Lemma \ref{lem4.3} are satisfied and
$$({\psi_{\bar\omega}}_*\omega_{\widetilde{\sR}^{ss}_{\bar\omega}})^{inv}=\omega_{\sU_{\wt
X,\,\bar\omega}}.$$ Then, since $\Theta_{\bar\omega}\otimes{\rm
Det}^*(\Theta_y)^{-2}$ is ample by Lemma 5.3 of \cite{S1}, we have
$$H^1(\sU_{\widetilde{X},\,\omega},\Theta_{\sU_{\widetilde{X},\,\omega}})=H^1(\sU_{\wt
X,\,\bar\omega},\Theta_{\bar\omega}\otimes{\rm
Det}^*(\Theta_y)^{-2}\otimes\omega_{\sU_{\wt X,\,\bar\omega}})=0.$$
\end{proof}

The idea of the proof is to express
$H^1(\sU_{\widetilde{X},\,\omega},\Theta_{\sU_{\widetilde{X},\,\omega}})$
by $$H^1(M,\sL\otimes\omega_M)$$ such that $\sL$ is an ample line
bundle, where $M$ is another GIT quotient. In this process, we need
essentially the equality
$$H^1(\widetilde{\sR}_{\omega}^{ss},\Theta_{\widetilde{\sR}_F})^{inv}
=H^1(\widetilde{\sR}_F,\Theta_{\widetilde{\sR}_F})^{inv}$$ which
perhaps holds unconditional. In fact, we have the following
\begin{conj}\label{conj4.5} For any data $\omega$ satisfying \eqref{4.1} and any
$i\ge 0$ \ga{4.2}
{H^i(\widetilde{\sR}_{\omega}^{ss},\Theta_{\widetilde{\sR}_F})^{inv}
=H^i(\widetilde{\sR}_F,\Theta_{\widetilde{\sR}_F})^{inv},} where
$\Theta_{\widetilde{\sR}_F}$ is the polarization determined by
$\omega$.
\end{conj}
Then the proof of Theorem \ref{thm4.4} implies the following
\begin{cor}\label{cor4.6} Assume the Conjecture \ref{conj4.5} is true. Then, for any data $\omega$ , we
 have, for any $i>0$,
 $$H^i(\sU_{\widetilde{X},\,\omega},\Theta_{\sU_{\widetilde{X},\,\omega}})=0.$$
\end{cor}

\begin{proof}
For any data $\omega=\{k,\vec n(x),\vec a(x)\}_{x\in I}$, we choose
$$\omega(I')=\{k,\vec
n(x),\vec a(x)\}_{x\in I\cup I'}$$ such that
$(r-1)(\wt g-1)+\frac{|I\cup I'|}{k+2r}\ge i+2$. Note that the
projection
$$p_I:
\widetilde{\sR}(I')=\underset{x\in I\cup I'}{\times_{\wt{\bold
Q}_F}}Flag_{\vec n(x)}(\sF_x)\rightarrow \wt{\sR}_F=\underset{x\in
I}{\times_{\wt{\bold Q}_F}}Flag_{\vec n(x)}(\sF_x)$$ is a Flag
bundle and ${\rm SL}(V)$-invariant. By Conjecture \ref{conj4.5}, we
have
$$\aligned &{\rm
H}^i(\sU_{\wt X,\,\omega},\Theta_{\sU_{\wt X,\,\omega}})={\rm
H}^i({\wt\sR}^{ss}_{\omega}, \Theta_{\wt\sR_F})^{inv}\\&={\rm
H}^i(\wt\sR_F, \Theta_{\wt\sR_F})^{inv}={\rm
H}^i(\widetilde{\sR}(I'),
p_I^*(\Theta_{\wt\sR_F}))^{inv}.\endaligned$$ Write
$p_I^*(\Theta_{\wt\sR_F})\otimes\omega_{\widetilde{\sR}(I')}^{-1}:=\hat\Theta_{\bar\omega}\otimes{\rm
Det}^*(\Theta_y)^{-2}$, then we have
$$\aligned\hat\Theta_{\bar\omega}=&(det\,R\pi_{\wt\sR_F}\widetilde{\sF})^{-\bar k}\otimes
\bigotimes_{x\in I\cup I'}\left\{(det\,\widetilde{\sF}_x)^{\bar\alpha_x}
\otimes\bigotimes^{l_x}_{i=1}(det\,\sQ_{x,i})
^{\bar d_i(x)}\right\}\\
&\otimes
(det\wt\sF_y)^{\bar\ell_y}\otimes\bigotimes_q(det\,\widetilde{\sF}_q)^{1-r}\otimes
(det\wt\sF_y)^{(r-1)(2\tilde g-2)}
\endaligned$$
where $\bar k=k+2r$, $\,\bar\alpha_x=\alpha_x+n_{l_x+1}(x)-r$, $\,\bar\ell_y=2\wt\chi+\tilde\ell_y\,$ and
$$\bar{d}_i(x)=d_i(x)+n_i(x)+n_{i+1}(x)$$
(we define $\alpha_x=0$, $d_i(x)=0$ when $x\in I'$). Let $\bar\omega=\{\bar k, \vec n(x),\vec {\bar a}(x)\}_{x\in I\cup I'}$
with $\vec {\bar a}(x)=(\bar a_1(x), \bar a_2(x),\cdots,\bar a_{l_x+1}(x))$ such
that $$\bar d_i(x)=\bar a_{i+1}(x)-\bar a_i(x),\quad (i=1,\,2,\,\cdots,\,l_x).$$
Let
$\wt{\sR}(I')^{ss}_{\bar\omega}\subset \wt{\sR}(I')$ be the open set
of GIT semi-stable points (respect to the polarization defined by
$\bar\omega$), then
$$\aligned &{\rm
H}^i(\sU_{\wt X,\,\omega},\Theta_{\sU_{\wt X,\,\omega}})={\rm
H}^i({\wt\sR}^{ss}_{\omega}, \Theta_{\wt\sR_F})^{inv}={\rm
H}^i(\wt\sR_F, \Theta_{\wt\sR_F})^{inv}\\&={\rm
H}^i(\widetilde{\sR}(I'),
p_I^*(\Theta_{\wt\sR_F}))^{inv}={\rm
H}^i(\widetilde{\sR}(I')^{ss},
p_I^*(\Theta_{\wt\sR_F}))^{inv}\endaligned$$
the last equality holds since, by (2) of Proposition \ref{prop4.1}, we have
$${\rm
codim}(\wt{\sR}(I')\setminus
\wt{\sR}(I)^{ss}_{\bar\omega})> (r-1)(\wt g-1)+\frac{|I\cup
J|}{k+2r}\ge i+2.$$
Let $\psi: \widetilde{\sR}(I')^{ss}_{\bar\omega}\to
\sU_{\wt X,\,\bar\omega}$ be the good quotient. Then
$\hat\Theta_{\bar\omega}$ descends to an ample line bundle
$\Theta_{\bar\omega}$ on $\sU_{\wt X,\,\bar\omega}$ and
$(\psi_*\omega_{\wt\sR(I')^{ss}_{\bar\omega}})^{inv}=\omega_{\sU_{\wt
X,\,\bar\omega}}$ since
$${\rm
codim}(\wt{\sR}(I')^{ss}_{\bar\omega}\setminus
\wt{\sR}(I)^s_{\bar\omega})\ge (r-1)(\wt g-1)+\frac{|I\cup
J|}{k+2r}\ge i+2$$
by (1) of Proposition \ref{prop4.1}. Thus we have \ga{4.3}{{\rm H}^i(\sU_{\wt
X,\,\omega},\Theta_{\sU_{\wt X,\,\omega}})={\rm H}^i(\sU_{\wt
X,\,\bar\omega},\Theta_{\bar\omega}\otimes{\rm
Det}^*(\Theta_y)^{-2}\otimes\omega_{\sU_{\wt X,\,\bar\omega}})} for
any $i\ge 0$. In particular, ${\rm H}^i(\sU_{\wt
X,\,\omega},\Theta_{\sU_{\wt X,\,\omega}})=0$ for $i>0$.
\end{proof}

For $i=0$, Conjecture \ref{conj4.5} is true according to a general
fact
\begin{lem}[Lemma 4.15 of \cite{NR}]\label{lem4.7} Let $V$ be a projective scheme
on which a reductive group $G$ acts, $\wt\sL$ an ample line bundle
linearizing the $G$-action, and $V^{ss}\subset V$ the open set of
semi-stable points. Then, for any open $G$-invariant (irreducible)
normal subscheme $V^{ss}\subset W\subset V$,
$${\rm H}^0(V^{ss},\wt\sL)^{inv}={\rm H}^0(W,\wt L)^{inv}.$$
\end{lem}

\begin{cor}\label{cor4.8}
For any data $\omega=\{k,\vec n(x),\vec a(x)\}_{x\in I})$ such that $\ell\in\mathbb{Z}$, the dimension of
$${\rm
H}^0(\sU_{\wt{X},\,\omega},\Theta_{\sU_{\wt{X},\,\omega}})$$ is
independent of the choices of curve $\wt X$ and the points $x\in \wt
X$.
\end{cor}

\begin{proof} By the above Lemma \ref{lem4.7} and \eqref{4.3}, we
have $${\rm H}^0(\sU_{\wt X,\,\omega},\Theta_{\sU_{\wt
X,\,\omega}})={\rm H}^0(\sU_{\wt
X,\,\bar\omega},\Theta_{\bar\omega}\otimes{\rm
Det}^*(\Theta_y)^{-2}\otimes\omega_{\sU_{\wt X,\,\bar\omega}}).$$
The dimension of ${\rm H}^0(\sU_{\wt
X,\,\bar\omega},\Theta_{\bar\omega}\otimes{\rm
Det}^*(\Theta_y)^{-2}\otimes\omega_{\sU_{\wt X,\,\bar\omega}})$ is
independent of the choices of curve $\wt X$ and the points $x\in \wt
X$ since $${\rm H}^i(\sU_{\wt
X,\,\bar\omega},\Theta_{\bar\omega}\otimes{\rm
Det}^*(\Theta_y)^{-2}\otimes\omega_{\sU_{\wt X,\,\bar\omega}})=0$$
for all $i>0$.

\end{proof}

\section {Vanishing theorem for irreducible nodal curves}

When curves degenerate to a nodal curve $X$, the invariance of
spaces of generalized theta functions for smooth curves has proved
in last section (See Corollary \ref{cor4.8}). To complete the
program, we need the vanishing theorem ${\rm
H}^1(\sU_X,\Theta_{\sU_X})=0$. Its proof was reduced to prove a
vanishing theorem on the normalization $\sP$ of $\sU_X$.

Let $X$ be a connected nodal curve of genus $g$, with only one node
$x_0\in X$, let $\pi:\wt X\to X$ be the normalization of $X$ and
$\pi^{-1}(x_0)=\{x_1\,,\,x_2\}$. The normalization $\phi:
\sP\to\sU_X$ of $\sU_X$ is given by moduli space of semi-stable GPS
$(E,Q)$ on $\widetilde{X}$ with additional parabolic structures at
the points of $I$ (we identify $I$ with $\pi^{-1}(I)$) given by the
data $$\omega=\{k, \vec n(x),\,\vec a(x)\}_{x\in I}$$ satisfying
$$\sum_{x\in I}\sum^{l_x}_{i=1}d_i(x)r_i(x)+r\wt\ell=k\wt\chi$$ where
$d_i(x)=a_{i+1}(x)-a_i(x)$, $\wt\chi=\chi+r$, $\wt\ell=k+\ell$.
Recall that
$$\widetilde{\sR}'=Grass_r(\sF_{x_1}\oplus\sF_{x_2})\times_{\widetilde{\bold
Q}}\widetilde{\sR}$$ with the ${\rm SL}(V)$-equivariant embedding
$$\wt{\sR}'\hookrightarrow \mathbf{G}'=Grass_{\wt P(m)}(V\otimes W_m)\times\bold {Flag}\times Grass_r(V\otimes W_m),$$
where $W_m=H^0(\wt\sW(m))$, and $\bold{Flag}$ is defined to be
$$\bold{Flag}=\prod_{x\in I}\{Grass_{r_1(x)}(V\otimes W_m)
\times\cdots\times Grass_{r_{l_x}(x)}(V\otimes W_m)\}.$$ On
$\mathbf{G}'$, take the polarisation (determined by $\omega$)
\ga{5.1}{k\times\frac{(\ell+kcN)}{c(m-N)}\times\prod_{x\in I}
\{d_1(x),\cdots,d_{l_x}(x)\}.}  Then, when $X$ is irreducible,
$\sP:=\sP_{\omega}$ is the GIT (good) quotient
$$\psi:\widetilde{\sR}_{\omega}^{\prime ss}\to\sP_{\omega}:=\widetilde{\sR}_{\omega}^{\prime ss}//{\rm SL}(V).$$
There is a open subscheme $\sH\subset\widetilde{\sR}'$ such that
$\widetilde{\sR}_{\omega}^{\prime ss}\subset\sH$ for any data
$\omega$ (See Notation \ref{nota2.21}), one of the main results
proved in \cite{S1} and \cite{S2} is that $\sH$ is reduced, normal
and Cohen-Macaulay with only rational singularities (so is $\sP$).
Thus the Kodaira-type vanishing theorem and Hartogs-type extension
theorem for cohomology are applicable.

Let $\rho:\wt{\sR}'\to \widetilde{\sR}$ be the projection,
$V\otimes\sO_{\widetilde{X}\times\sH}(-N)\to \sE\to 0$,
$$\{\,\,\sE_{\{x\}\times \sH}=\sQ_{\{x\}\times \sH,\,l_x+1}\twoheadrightarrow\sQ_{\{x\}\times \sH,\,l_x}
\twoheadrightarrow \cdots\twoheadrightarrow \sQ_{\{x\}\times\,\sH,1}
\twoheadrightarrow0\,\,\}_{x\in I}$$ denote pullbacks of universal
quotients $V\otimes\sO_{\widetilde{X}\times\widetilde{\sR}}(-N)\to
\widetilde{\sF}\to 0$,
$$\{\,\,\wt\sF_{\{x\}\times \wt R}=\wt\sQ_{\{x\}\times \wt R,\,l_x+1}\twoheadrightarrow\wt\sQ_{\{x\}\times \wt R,\,l_x}
\twoheadrightarrow\cdots\twoheadrightarrow \wt\sQ_{\{x\}\times \wt
R,\,1}\twoheadrightarrow0\,\,\}_{x\in I}.$$ Then the restriction of
polarisation \eqref{5.1} to $\sH$ is
$$\hat\Theta_{\sH}':={\det}(\sQ)^k\otimes({\rm
det}R\pi_{\sH}\sE(m))^{\frac{\ell+kcN}{c(m-N)}}\otimes\bigotimes_{x\in
I}\left\{\bigotimes^{l_x}_{i=1}{\rm
det}(\sQ_{\{x\}\times\sH,\,i})^{d_i(x)}\right\}$$ where
$\sE_{x_1}\oplus\sE_{x_2}\to\sQ\to 0$ is the universal quotient on
$\sH$. If we choose $\sO(1)=\sO_{\wt X}(cy)$, note that
$\sO_{\sH}={\rm det}R\pi_{\sH}\sE(N)$, we have $$({\rm
det}R\pi_{\sH}\sE)^{-1}=({\rm det}\sE_y)^{cN},\quad {\rm
det}R\pi_{\sH}\sE(m)=({\rm det}\sE_y)^{c(m-N)},$$
$$\hat\Theta_{\sH}'={\det}(\sQ)^k\otimes({\rm
det}R\pi_{\sH}\sE)^{-k}\otimes\bigotimes_{x\in
I}\left\{\bigotimes^{l_x}_{i=1}{\rm
det}(\sQ_{\{x\}\times\sH,\,i})^{d_i(x)}\right\}\otimes({\rm
det}\sE_y)^{\ell}.$$
We will write $\hat\Theta_{\sH}'=\eta_y^k\otimes\rho^*\hat\Theta_{\wt\sR}$,
where $\eta_y={\det}(\sQ)\otimes{\det}(\sE_y)^{-1}$ and
$$ \hat\Theta_{\wt\sR}=({\rm
det}R\pi_{\wt\sR}\wt\sF)^{-k}\otimes\bigotimes_{x\in
I}\left\{\bigotimes^{l_x}_{i=1}{\rm
det}(\wt\sQ_{\{x\}\times\wt\sR,\,i})^{d_i(x)}\right\}\otimes({\rm
det}\wt\sF_y)^{\wt\ell}.$$ The universal quotient
$\sE_{x_1}\oplus\sE_{x_2}\to\sQ\to 0$ induces an exact sequence
\ga{5.2}{0\to\sF_{\sH}\to(\pi\times id_{\sH})_*\sE\to\,_{x_0}\sQ\to
0} on $X\times\sH$, where $\wt X\times\sH\xrightarrow{\pi\times
id_{\sH}} X\times\sH$. The sheaf $\sF_{\wt{\sR}_{\omega}^{\prime
ss}}$ defines $$\hat\phi:\wt{\sR}_{\omega}^{\prime ss}\to
\sU_X:=\sU_{X,\,\omega},$$ which induces a morphism
$\phi:\sP=\wt{\sR}_{\omega}^{\prime ss}//{\rm SL}(V)\to \sU_X$ such
that $$ \xymatrix{ \wt{\sR}_{\omega}^{\prime ss} \ar[dr]_{\hat\phi}
\ar[r]^{\psi}
                & \sP \ar[d]^{\phi}  \\
                & \sU_X             }$$
is commutative and
${\hat\phi}^*\Theta_{\sU_X}=\hat\Theta'_{\wt{\sR}_{\omega}^{\prime
ss}}$. Thus $\hat\Theta'_{\wt{\sR}_{\omega}^{\prime ss}}$ descends
to an ample line bundle $\Theta_{\sP}=\phi^*\Theta_{\sU_X}$. In fact, there are more general ample line bundles $\Theta_{\sP,\,\omega}$ on $\sP$, which are the descendants of
$$\aligned
\hat{\Theta}_{\omega}'&=(det\,R\pi_{\wt{\sR}'}\sE)^{-k}\otimes\bigotimes_{x\in
I} \lbrace(det\,\sE_x)^{\alpha_x}\otimes\bigotimes^{l_x}_{i=1}
(det\,\sQ_{x,i})^{d_i(x)}\rbrace\otimes(det\,\sE_y)^{\tilde\ell_y}
\otimes\eta_y^k\\&=\rho^*\Theta_{\wt\sR,\,\omega}\otimes(det\,\sQ\otimes
det\,\sE_y^{-1})^k\endaligned$$  such that
$\Theta_{\sP,\,\omega}=\phi^*\Theta_{\sU_X,\,\omega}$ where
$\wt\ell_y+\sum_{x\in I}\alpha_x=\wt\ell$, and
$\Theta_{\sU_X,\,\omega}=\Theta_{\sU_X,L}$ is determined (cf. Theorem \ref{thm3.1}) by the data $\omega=\{k, \vec n(x),\,\vec a(x)\}_{x\in I}$ and
$$L=\ell_yy+\sum_{x\in I}\alpha_xx.$$
By Lemma 5.5 of \cite{S1}, we have injection $\phi^*: {\rm
H}^1(\sU_X,\Theta_{\sU_X,\,\omega})\hookrightarrow {\rm
H}^1(\sP,\Theta_{\sP,\,\omega}).$
Thus it is enough to show ${\rm H}^1(\sP,\Theta_{\sP,\,\omega})=0$. Let $\sK$ be the kernel of
$$V\otimes\sO_{\widetilde{X}\times\widetilde{\sR}'}(-N)\to \sE\to 0,$$
and consider $0\to\sK\to V\otimes\sO_{\widetilde{X}\times\sH}(-N)\to \sE\to 0.$
The line
bundle ${\rm det}(\sK)^{-1}\otimes\sO_{\widetilde{X}\times\sH}(-{\rm
dim}(V)N)$ on $\widetilde{X}\times\sH$ defines
${\rm Det}_{\sH}: \sH\to J^d_{\wt X}$
which induces the determinant morphism (cf. Lemma 5.7 of \cite{S1})
\ga{5.3} {{\rm Det}: \sP\to J^d_{\wt X}.}

\begin{prop}[Proposition 3.4 of \cite{S1}]\label{prop5.1} Let $\omega_{\wt X}=\sO(\sum_qq)$ and
$$\Theta_{J^d_{\wt X}}=(detR\pi_{J^d_{\wt X}}\sL)^{-2}\otimes\sL_{x_1}^r\otimes\sL_{x_2}^r\otimes\sL_y^{2\wt\chi-2r}\otimes\bigotimes_q\sL_q^{r-1}$$
where $\sL$ is the universal line bundle on $\wt X\times
J^d_{\wt X}$. Then we have
$$\aligned&\omega^{-1}_{\sH}=(det\,R\pi_{\sH}\sE)^{-2r}\otimes\\&
\bigotimes_{x\in I}\left\{(det\,\sE_x)^{n_{l_x+1}-r}
\otimes\bigotimes^{l_x}_{i=1}(det\,\sQ_{x,i})
^{n_i(x)+n_{i+1}(x)}\right\}\otimes(det\,\sQ)^{2r}\\&
\otimes(det\,\sE_y)^{2\wt\chi-2r}\otimes{\rm Det}_{\sH}^*(\Theta_{J^d_{\wt X}}^{-1}).
\endaligned$$
\end{prop}

We will prove $R^1{\rm Det}_*(\Theta_{\sP,\,\omega})=0$ and
$H^1(J^d_{\wt X}, {\rm Det}_*\Theta_{\sP,\,\omega})=0,$ which imply ${\rm H}^1(\sP,\Theta_{\sP,\,\omega})=0$. To recall the proof of
$H^1(J^d_{\wt X}, {\rm Det}_*\Theta_{\sP,\,\omega})=0$.
Let $\widetilde{\sR}'_F\subset\wt{\sR}'$, $\widetilde{\sR}_F\subset\wt{\sR}$ denote open set of locally
free quotients, for $\mu=(\mu_1,\cdots,\mu_r)$  with $0\le\mu_r\le\cdots\le\mu_1\le
k,$ let $$\{d_i=\mu_{r_i}-\mu_{r_i+1}\}_{1\le i\le l}$$ be the
subset of nonzero integers in
$\{\mu_i-\mu_{i+1}\}_{i=1,\cdots,r-1}.$ We define
$$r_i(x_1)=r_i,\quad r_i(x_2)=r-r_{l-i+1}, \quad l_{x_1}=l_{x_2}=l$$
$$\vec n(x_j)=(r_1(x_j),r_2(x_j)-r_1(x_j),
\cdots,r_{l_{x_j}}(x_j)-r_{l_{x_j}-1}(x_j)),$$
$${\wt\sR}_F^{\mu}=\underset{x\in
I\cup\{x_1,\,x_2\}}{\times_{\wt{\mathbf{Q}}_F}}Flag_{\vec
n(x)}(\sF_x)\xrightarrow{p^{\mu}}\wt\sR_F=\underset{x\in
I}{\times_{\wt{\mathbf{Q}}_F}}Flag_{\vec n(x)}(\sF_x).$$
Then, by Remark 4.2 of \cite{S1}, we have decomposition (on
$\wt{\sR}_F$)
\ga{5.4} {\rho_*(\hat{\Theta}_{\omega}')=\bigoplus_{\mu}p_*^{\mu}(\hat\Theta_{\mu})}
$\mu=(\mu_1,\cdots,\mu_r)$ runs through integers
$0\le\mu_1\le\cdots\mu_r\le k$ and
$$\hat\Theta_{\mu}=(det\,R\pi_{\widetilde{\sR}^{\mu}_{F}}\widetilde{\sF})^{-k}\otimes
\bigotimes_{x\in I\cup\{x_1,\,x_2\}}
\{(det\,\widetilde{\sF}_x)^{\alpha_x}\otimes\bigotimes^{l_x}_{i=1}
(det\,\sQ_{x,i})^{d_i(x)}\}\otimes
(det\,\widetilde{\sF}_y)^{\ell_y}$$ where $\,
r_i(x_1)=r_i$, $d_i(x_1)=d_i$, $l_{x_1}=l$, $\alpha_{x_1}=\mu_r$, $
r_i(x_2)=r-r_{l-i+1}$, $d_i(x_2)=d_{l-i+1}$, $l_{x_2}=l$,
$\alpha_{x_2}=k-\mu_1$ and for $j=1,2$, we set
$$\vec a(x_j)=\left(\mu_r,\mu_r+d_1(x_j),\cdots,\mu_r+
\sum^{l_{x_j}-1}_{i=1}d_i(x_j),\mu_r+\sum^{l_{x_j}}_{i=1}d_i(x_j)\right).$$
It is easy to check that $$\sum_{x\in
I\cup\{x_1,\,x_2\}}\sum_{i=1}^{l_x}d_i(x)r_i(x)+r\sum_{x\in
I\cup\{x_1,\,x_2\}}\alpha_x+r\ell_y=k\wt\chi.$$

\

For the data $\omega^{\mu}=\{k, \vec n(x),\vec a_i(x)\}_{x\in I\cup\{x_1,\,x_2\}}$, we choose $$\omega^{\mu}(I')=\{k, \vec n(x),\vec a_i(x)\}_{x\in I\cup\{x_1,\,x_2\}\cup I'}$$ such that
$(r-1)(\wt g-1)+\frac{2+|I\cup I'|}{k+2r}\ge 2$. Note that the
projection
$$p_I:
\widetilde{\sR}^{\mu}(I')=\wt{\sR}^{\mu}_F\times_{\wt{\bold
Q}_F}\left(\underset{x\in I'}{\times_{\wt{\bold
Q}_F}}Flag_{\vec n(x)}(\wt\sF_x)\right)\rightarrow \wt{\sR}^{\mu}_F$$ is a ${\rm SL}(V)$-invariant Flag
bundle, consider the commutative diagram
\ga{5.5} {\xymatrix{ \widetilde{\sR}^{\mu}(I') \ar[dr]_{\hat{{\rm Det}}^{I'}_{\mu}}
\ar[r]^{p_I}
                & \wt{\sR}^{\mu}_F \ar[d]^{\hat{{\rm Det}}_{\mu}}  \\
                & J^d_{\wt X}             }}
and write $p_I^*(\hat\Theta_{\mu})\otimes\omega^{-1}_{\wt{\sR}^{\mu}(I')}=\hat\Theta_{\bar\omega_{\mu}}\otimes(\hat{{\rm Det}}^{I'}_{\mu})^*(\Theta_y)^{-2}$.
Then
$$\aligned\hat\Theta_{\bar\omega_{\mu}}=&(det\,R\pi\widetilde{\sF})^{-\bar k}\otimes
\bigotimes_{x\in I\cup\{x_1,\,x_2\}\cup I'}
\{(det\,\widetilde{\sF}_x)^{\bar\alpha_x}\otimes\bigotimes^{l_x}_{i=1}
(det\,\sQ_{x,i})^{\bar d_i(x)}\}\\&\otimes
(det\,\widetilde{\sF}_y)^{\bar\ell_y+(r-1)(2\wt g-2)}\otimes\bigotimes_q(det\wt\sF_q)^{1-r}\endaligned$$
where $\bar k=k+2r$, $\bar\alpha_x=\alpha_x+n_{l_x+1}(x)-r$, $\bar \ell_y=2\wt\chi+\wt\ell_y$ and
$$\bar d_i(x)=d_i(x)+n_i(x)+n_{i+1}(x),$$
$\bar\omega_{\mu}=\{\bar k,\vec n(x),\vec {\bar a}(x)\}_{I\cup\{x_1,\,x_2\}\cup I'}$ with $\vec {\bar a}(x)=(\bar a_1(x),\bar a_2(x),\cdots,\bar a_{l_x+1}(x))$
(note: $\bar a_{l_x+1}(x)-\bar a_1(x)=\sum_{i=1}^{l_x}\bar d_i(x)=a_{l_x+1}(x)-a_1(x)+2r-n_1(x)-n_{l_x+1}(x)\le k+2r-n_1(x)-n_{l_x+1}(x)<\bar k$).

\

Let
$\wt{\sR}^{\mu}(I')^{ss}_{\bar\omega_{\mu}}\subset \wt{\sR}^{\mu}(I')$ be the open set
of GIT semi-stable points (respect to the polarization defined by
$\bar\omega_{\mu}$), then $${\rm
codim}(\wt{\sR}^{\mu}(I')^{ss}_{\bar\omega_{\mu}}\setminus
\wt{\sR}^{\mu}(I')^s_{\bar\omega_{\mu}})\ge (r-1)(\wt g-1)+\frac{2+|I\cup
I'|}{k+2r}\ge 2.$$ Let $\psi: \widetilde{\sR}^{\mu}(I')^{ss}_{\bar\omega_{\mu}}\to
\sU_{\wt X,\,\bar\omega_{\mu}}$ be the good quotient. Then
$\hat\Theta_{\bar\omega_{\mu}}$ descends to an ample line bundle
$\Theta_{\bar\omega_{\mu}}$ on $\sU_{\wt X,\,\bar\omega_{\mu}}$ and
$(\psi_*\omega_{\wt\sR^{\mu}(I)^{ss}_{\bar\omega_{\mu}}})^{inv}=\omega_{\sU_{\wt
X,\,\bar\omega_{\mu}}}$.

\begin{lem} Let ${\rm Det}^{I'}_{\mu}: \sU_{\wt X,\,\bar\omega_{\mu}}\to J^d_{\wt X}$ be the morphism induced by
$$\hat{{\rm Det}}^{I'}_{\mu}: \widetilde{\sR}^{\mu}(I')^{ss}_{\bar\omega_{\mu}}\to J^d_{\wt X}$$
and ${\rm Det}: \sP\to J^d_{\wt X}$ be the determinant morphism. Then
\ga{5.6}{{\rm Det}_*(\Theta_{\sP,\,\omega})=\bigoplus_{\mu}({\rm Det}^{I'}_{\mu})_*
(\Theta_{\bar\omega_{\mu}}\otimes({\rm Det}^{I'}_{\mu})^*(\Theta_y)^{-2}\otimes\omega_{\sU_{\wt X,\,\bar\omega_{\mu}}})}
where $\mu=(\mu_1,\cdots,\mu_r)$ runs through integers $0\le\mu_1\le\cdots\mu_r\le k$. In particular, we have
$$H^i(J^d_{\wt X}, {\rm Det}_*\Theta_{\sP,\,\omega})=0\quad \forall\,\,i>0.$$
\end{lem}

\begin{proof} Note ${\rm Det}_*(\Theta_{\sP,\,\omega})=\{({\rm Det}_{{\wt{\sR}}^{\prime\,ss}})_*{\hat\Theta}'_{\omega}\}^{inv}=\{({\rm Det}_{{\wt{\sR}_F}^{\prime}})_*{\hat\Theta}'_{\omega}\}^{inv}$ and
$({\rm Det}_{{\wt{\sR}_F}^{\prime}})_*{\hat\Theta}'_{\omega}=({\rm Det}_{\wt{\sR}_F})_*\rho_*{\hat\Theta}'_{\omega}$, by the decomposition \eqref{5.4}, we have
$$({\rm Det}_{{\wt{\sR}_F}^{\prime}})_*{\hat\Theta}'_{\omega}=\bigoplus_{\mu}(\hat{{\rm Det}}_{\mu})_*\hat\Theta_{\mu}$$
where $\hat{{\rm Det}}_{\mu}: \wt{\sR}^{\mu}_F\to J^d_{\wt X}$ satisfies the commutative diagram
$$ \xymatrix{ \wt{\sR}^{\mu}_F \ar[dr]_{\hat{{\rm Det}}_{\mu}}
\ar[r]^{p^{\mu}}
                & \wt{\sR}_F \ar[d]^{{\rm Det}_{\wt{\sR}_F}}  \\
                & J^d_{\wt X}             }$$
By diagram \eqref{5.5} and $p_I^*(\hat\Theta_{\mu})=\hat\Theta_{\bar\omega_{\mu}}\otimes(\hat{{\rm Det}}^{I'}_{\mu})^*(\Theta_y)^{-2}\otimes\omega_{\wt{\sR}^{\mu}(I')}$, we have
\ga{5.7}{(\hat{{\rm Det}}_{\mu})_*\hat\Theta_{\mu}=(\hat{{\rm Det}}^{I'}_{\mu})_*(\hat\Theta_{\bar\omega_{\mu}}\otimes(\hat{{\rm Det}}^{I'}_{\mu})^*(\Theta_y)^{-2}\otimes\omega_{\wt{\sR}^{\mu}(I')}).}
Recall $\psi: \widetilde{\sR}^{\mu}(I')^{ss}_{\bar\omega_{\mu}}\to\sU_{\wt X,\,\bar\omega_{\mu}}$, $\hat\Theta_{\bar\omega_{\mu}}=\psi^*\Theta_{\bar\omega_{\mu}}$,
$(\psi_*\omega_{\wt\sR^{\mu}(I')^{ss}_{\bar\omega_{\mu}}})^{inv}=\omega_{\sU_{\wt
X,\,\bar\omega_{\mu}}},$ then we have the decomposition \eqref{5.6}. The vanishing result follows the decomposition clearly since
$\Theta_{\bar\omega_{\mu}}\otimes({\rm Det}^{I'}_{\mu})^*(\Theta_y)^{-2}$ is ample.
\end{proof}

To prove $R^1{\rm Det}_*(\Theta_{\sP,\,\omega})=0$, the idea is same with Section 4. Let
$$\wt{\sR}(I')=\underset{x\in I\cup I'}{\times_{\wt{\bold Q}}}Flag_{\vec n(x)}(\sF_x)\xrightarrow{p_I} \wt{\sR}
=\underset{x\in I}{\times_{\wt{\bold Q}}}Flag_{\vec n(x)}(\sF_x),$$
$$\wt{\sR}'(I')=Grass_r(\sF_{x_1}\oplus\sF_{x_2})\times_{\widetilde{\bold
Q}}\widetilde{\sR}(I')\xrightarrow{p_I} \wt{\sR}'=Grass_r(\sF_{x_1}\oplus\sF_{x_2})\times_{\widetilde{\bold
Q}}\widetilde{\sR}$$
be the projection, $\sH(I')\subset\wt{\sR}'(I')$, $\sH\subset \wt{\sR}'$ be the open set defined in Notation \ref{nota2.21}.
By Proposition \ref{prop5.1},  we have
\ga{5.8} {p_I^*(\hat\Theta'_{\omega})\otimes\omega^{-1}_{\sH(I')}=
\hat{\Theta}_{\bar\omega}'\otimes{\rm Det}_{\sH(I')}^*(\Theta_{J^d_{\wt X}}^{-1})}
with $\bar\omega=(d, r,
\bar k, \bar\ell_y, \{\bar\alpha_x,\,\bar d_i(x)\}_{x\in I\cup J,1\le i\le
l_x})$ and
$$\aligned \hat{\Theta}_{\bar\omega}'=&(det\,R\pi_{\sH(I')}\sE)^{-\bar k}\otimes\bigotimes_{x\in I\cup I'}
\lbrace(det\,\sE_x)^{\bar\alpha_x}\otimes\bigotimes^{l_x}_{i=1}
(det\,\sQ_{x,i})^{\bar
d_i(x)}\rbrace\\&\otimes(det\,\sE_y)^{\bar\ell_y}
\otimes(det\,\sQ)^{\bar k}\otimes(det\,\sE_y)^{-\bar k}\endaligned$$
where $\bar k=k+2r$, $\bar\alpha_x=\alpha_x+n_{l_x+1}(x)-r$,
$\bar\ell_y=\wt\ell_y+2\wt\chi$, and
$$\bar d_i(x)=d_i(x)+n_i(x)+n_{i+1}(x).$$

Let $\wt{\sR}'(I')^{ss}_{\bar\omega}\subset \sH(I')$
be the open set of GIT semi-stable points (respect to $\bar\omega$),
$\psi: \wt{\sR}'(I')^{ss}_{\bar\omega}\to \sP_{\bar\omega}:=\wt{\sR}'(I)^{ss}_{\bar\omega}//{\rm SL}(V)$
be the quotient map. There is an ample line bundle $\Theta_{\sP,\,\bar\omega}$ on $\sP_{\bar\omega}$ such that
$\hat{\Theta}_{\bar\omega}'=\psi^*(\Theta_{\sP,\,\bar\omega})$, and $\omega_{\sP_{\bar\omega}}=(\psi_*\omega_{\wt{\sR}'(I')^{ss}_{\bar\omega}})^{inv}$ if
\ga{5.9}{(r-1)(\wt g-1)+\frac{|I|+|I'|}{k+2r}\ge 2}
where we need essentially the estimate of codimension from \cite{S1}.

\begin{prop}[Proposition 5.2 of \cite{S1}]\label{prop5.2} Let $\sD_1^f=\hat\sD_1\cup\hat\sD_1^t$
and $\sD_2^f=\hat\sD_2\cup\hat\sD_2^t$, where $\hat\sD_i\subset \wt{\sR}'$ is the Zariski closure of
$\hat\sD_{F,\,1}\subset\wt{\sR}'_F$ consisting of $(E,Q)\in\wt{\sR}'_F$ that $E_{x_i}\to Q$ is not an isomorphism, and
${\hat\sD}_1^t\subset\wt{\sR}'$ (rep. ${\hat\sD}_2^t\subset\wt{\sR}'$) consists of $(E,Q)\in\wt{\sR}'$ such that $E$ is not locally
free at $x_2$ (resp. at $x_1$). Then \begin{itemize}
\item [(1)] ${\rm codim}(\sH\setminus\wt\sR_{\omega}^{\prime ss})>(r-1)\tilde g+\frac{|I|}{k}.$
\item [(2)] the complement in $\wt\sR_{\omega}^{\prime ss}\setminus\{\sD_1^f\cup\sD^f_2\}$ of the set
$\wt\sR_{\omega}^{\prime s}$ of stable points has codimension $\ge
(r-1)\tilde g+\frac{|I|}{k}$.\end{itemize}
\end{prop}

\begin{lem}\label{lem5.3} When $(r-1)\wt g+\frac{|I|}{k}\ge 2$ and $I'\subset \wt X\setminus I$ satisfying \eqref{5.9},
\ga{5.10}{H^1(\sP_{\omega},\Theta_{\sP,\,\omega})=
H^1(\sP_{\bar\omega},\Theta_{\sP,\,\bar\omega}\otimes{\rm
Det}_J^*(\Theta_{J^d_{\wt X}}^{-1})\otimes\omega_{\sP_{\bar\omega}})}
where ${\rm Det}_J: \sP_{\bar\omega}\to J^d_{\wt X}$ is induced by ${\rm Det}_{\sH(I')}:\sH(I')\to J^d_{\wt X}$.
\end{lem}

\begin{proof} By using
Proposition \ref{prop4.1} (1) and Proposition \ref{prop5.2} (2), we have
$$(\psi_*\omega_{\wt{\sR}'(I')^{ss}_{\bar\omega}})^{inv}=\omega_{\sP_{\bar\omega}}$$ (cf. Lemma 5.6
of \cite{S1}). By Proposition \ref{prop5.2} (1), we have
$${\rm codim}(\sH\setminus\wt\sR_{\omega}^{\prime ss})\ge 3,\quad {\rm codim}(\sH(I')\setminus{\wt\sR}'(I')_{\bar\omega}^{ss})\ge 3$$
for any data $\omega$. Thus, by theory of local cohomology, we have
$$\aligned H^1(\sP_{\omega},\Theta_{\sP,\,\omega})&=H^1(\wt{\sR}^{\prime
ss}_{\omega},\hat\Theta'_{\omega})^{inv}=H^1(\sH,\hat\Theta'_{\omega})^{inv}=H^1(\sH(I'),p^*_I(\hat\Theta'_{\omega}))^{inv}\\
&=H^1(\sH(I'),\hat\Theta'_{\bar\omega}\otimes{\rm
Det}_{\sH(I')}^*(\Theta_{J^d_{\wt X}}^{-1})\otimes\omega_{\sH(I)})^{inv}\\
&=H^1(\wt{\sR}'(I')^{ss}_{\bar\omega},\hat\Theta'_{\bar\omega}\otimes{\rm
Det}_{\wt{\sR}'(I')^{ss}_{\bar\omega}}^*(\Theta_{J^d_{\wt X}}^{-1})\otimes\omega_{\wt{\sR}'(I')^{ss}_{\bar\omega}})^{inv}
\\&=H^1(\wt{\sR}'(I')^{ss}_{\bar\omega},\psi^*(\Theta_{\sP,\,\bar\omega}\otimes{\rm
Det}_J^*(\Theta_{J^d_{\wt X}}^{-1}))\otimes\omega_{\wt{\sR}'(I')^{ss}_{\bar\omega}})^{inv}
\\&=H^1(\sP_{\bar\omega},\Theta_{\sP,\,\bar\omega}\otimes{\rm
Det}_J^*(\Theta_{J^d_{\wt X}}^{-1})\otimes(\psi_*\omega_{\wt{\sR}'(I')^{ss}_{\bar\omega}})^{inv})
\\&=H^1(\sP_{\bar\omega},\Theta_{\sP,\,\bar\omega}\otimes{\rm
Det}_J^*(\Theta_{J^d_{\wt X}}^{-1})\otimes\omega_{\sP_{\bar\omega}}).
\endaligned$$

\end{proof}

When $X$ is irreducible, $\Theta_{\sP,\,\bar\omega}\otimes{\rm
Det}_J^*(\Theta_{J^d_{\wt X}}^{-1})$ may not be an ample line bundle on $\sP_{\bar\omega}$.
But, for any $L\in J^d_{\wt X}$, on the fiber $\sP^L_{\omega}={\rm Det}^{-1}(L)$ of
$${\rm Det}:\sP_{\omega}\to J^d_{\wt X}$$
and the fiber $\sP^L_{\bar\omega}={\rm Det}_J^{-1}(L)$ of ${\rm Det}_J:\sP_{\bar\omega}\to J^d_{\wt X}$
we have  $$H^1(\sP^L_{\omega},\Theta^L_{\sP,\,\omega})=
H^1(\sP^L_{\bar\omega},\Theta^L_{\sP,\,\bar\omega}\otimes\omega_{\sP^L_{\bar\omega}})=0$$
when $(r-1)(g-1)+\frac{|I|}{k}\ge 2$, which means ${\rm R}^1{\rm
Det}_*(\Theta_{\sP,\,\omega})=0.$

\begin{thm}[Theorem 5.3 of \cite{S1}]\label{thm5.4} If $X$ is an irreducible curve of genus
$g$ with one node and $(r-1)(g-1)+\frac{|I|}{k}\ge 2$, then
$$H^1(\sU_{X},\Theta_{\sU_X,\,\omega})\cong H^1(\sP_{\omega},\Theta_{\sP,\,\omega})=0.$$
\end{thm}

\begin{rmk}\label{rmk5.5} The condition $(r-1)(g-1)+\frac{|I|}{k}\ge 2$ is used only for the proof of
$H^1(\wt{\sR}^{\prime ss}_{\omega},\hat\Theta'_{\omega})^{inv}=H^1(\sH,\hat\Theta'_{\omega})^{inv}$
in Lemma \ref{lem5.3}, which may hold unconditional. In fact, we conjecture that for any $i\ge 0$ and $\omega$,
$$H^i(\wt{\sR}^{\prime ss}_{\omega},\hat\Theta'_{\omega})^{inv}=H^i(\sH,\hat\Theta'_{\omega})^{inv}.$$
If the conjecture is true, $H^i(\sP^L_{\omega},\Theta^L_{\sP,\,\omega})=0$ holds unconditional for $i>0$, which
implies that $H^i(\sP_{\omega},\Theta_{\sP,\,\omega})=0$ for $i>0$.
\end{rmk}

\section {generalized parabolic sheaves on reducible nodal curves}

A natural idea to prove a vanishing theorem
$H^1(\sU_{X},\Theta_{\sU_X,\,\omega})=0$ for $X=X_1\cup X_2$ is to extend above method to reducible curves.
In this section, we give estimates of various codimension and compute canonical line bundle of moduli space of generalized parabolic sheaves on a reducible curve.
However, the estimate is not good enough to prove a vanishing theorem via the method in last section.

Let $\chi_1$ and $\chi_2$ be integers such that
$\chi_1+\chi_2-r=\chi$, and fix, for $i=1,2$, the polynomials
$P_i(m)=c_irm+\chi_i$ and $\sW_i=\sO_{X_i}(-N)$ where
$\sO_{X_i}(1)=\sO (1)|_{X_i}$ has degree $c_i$. Write $V_i=\Bbb
C^{P_i(N)}$ and consider the Quot schemes $Quot(V_i\otimes\sW_i,
P_i)$, let $\wt{\textbf{Q}}_i$ be the closure of the open set
$$\mathbf{Q}_i=\left\{\aligned&\text{$V_i\otimes\sW_i\to E_i\to 0$, with
$H^1(E_i(N))=0$ and}\\&\text{$V_i\to H^0(E_i(N))$ induces an
isomorphism}\endaligned\right\},$$ we have the universal quotient
$V_i\otimes\sW_i\to \sF^i\to 0$ on $X_i\times \wt{\textbf{Q}}_i$ and
the relative flag scheme
$$\sR_i=\underset{x\in I_i}{\times_{\wt{\textbf{Q}}_i}}
Flag_{\vec n(x)}(\sF^i_x)\to \wt{\textbf{Q}}_i.$$ Let
$\sF=\sF^1\oplus\sF^2$ denote direct sum of pullbacks of $\sF^1$,
$\sF^2$ on $$\wt X\times
(\wt{\textbf{Q}}_1\times\wt{\textbf{Q}}_2)=(X_1\times\wt{\textbf{Q}}_1)\sqcup(X_2\times\wt{\textbf{Q}}_2).$$
Let $\sE$ be the pullback of $\sF$ to $\wt
X\times(\sR_1\times\sR_2)$, and
$$\rho:\widetilde{\sR}':=Grass_r(\sE_{x_1}\oplus\sE_{x_2})\to\wt\sR:=\sR_1\times\sR_2\to
\wt{\textbf{Q}}:=\wt{\textbf{Q}}_1\times\wt{\textbf{Q}}_2.$$
When $m$ is large enough, we have a $G$-equivariant embedding
$$\wt{\sR}'\hookrightarrow \mathbf{G}'=Grass_{\wt P(m)}(\wt V\otimes W_m)\times\bold {Flag}\times Grass_r(\wt V\otimes W_m).$$
For $\omega=(r,\chi_1,\chi_2,\{\vec n(x),\,\vec a(x)\}_{x\in I},\mathcal{O}(1),k)$, give $\mathbf{G}'$
polarization
\ga{6.1}{\frac{\ell+kcN}{c(m-N)}\times\prod_{x\in
I} \{d_1(x),\cdots,d_{l_x}(x)\}\times k.}
where $I=I_1\cup I_2$, $d_i(x)=a_{i+1}(x)-a_i(x)$, $r_i(x)=n_1(x)+\cdots+n_i(x)$,
$$\ell=\frac{k\chi-\sum_{x\in I}\sum^{l_x}_{i=1}d_i(x)r_i(x)}{r}.$$
Let $\sH\subset\wt{\sR}'$ be the open set defined in Notation \ref{nota2.21}, $\wt{\sR}_{\omega}^{\prime ss}\subset\sH$
be the open set of GIT semi-stable points (respect to the polarization). Let
$$\psi:\wt{\sR}_{\omega}^{\prime ss}\to \sP_{\omega}:=\wt{\sR}_{\omega}^{\prime ss}//G.$$
If $\sO(1)|_{X_j}=\sO_{X_j}(c_jy_j)$, the restriction of polarization \eqref{6.1} to $\sH$ is
$${\hat\Theta}'_{\sH}=\rho^*(\hat\Theta_{\sR_1}\boxtimes\hat\Theta_{\sR_2})\otimes {\rm det}(\sQ)^k$$
where (for $j=1,\,2$, $\pi_{\sR_j}:X_j\times\sR_j\to\sR_j$ is projection) we have
$$\hat\Theta_{\sR_j}=({\rm det}R\pi_{\sR_j}\sE^j)^{-k}\otimes\bigotimes_{x\in I_j}
\left\{\bigotimes^{l_x}_{i=1}{\rm det}(\sQ_{\{x\}\times\sR_j,\,i})^{d_i(x)}\right\}\otimes({\rm det}\sE^j_{y_j})^{\frac{c_j\ell}{c_1+c_2}}$$
where we assume that $\ell$ and $\ell_j:=\frac{c_j\ell}{c_1+c_2}$ are integers. The sequence
$$0\to\sF\to (\pi\times id)_*\sE\to\,_{x_0}\sQ\to 0$$
on $X\times\wt{\sR}_{\omega}^{\prime ss}$ defines a morphism $\hat\phi: \wt{\sR}_{\omega}^{\prime ss}\to \sU_X$ such that
$$\aligned{\hat\phi}^*(\Theta_{\sU_X})=&{\rm det}R\pi_{\wt{\sR}_{\omega}^{\prime ss}}(\sF)^{-k}\otimes\bigotimes_{x\in I}
\left\{\bigotimes^{l_x}_{i=1}{\rm det}(\sQ_{\{x\}\times\wt{\sR}_{\omega}^{\prime ss},\,i})^{d_i(x)}\right\}\\&\otimes({\rm det}\sF_{y_1})^{\ell_1}\otimes({\rm det}\sF_{y_2})^{\ell_2}={\hat\Theta}'_{\wt{\sR}_{\omega}^{\prime ss}}.\endaligned$$
Clearly, $\hat\phi$ induces a morphism $\phi: \sP_{\omega}\to \sU_X$ such that $\hat\phi=\phi\cdot\psi$. Thus
${\hat\Theta}'_{\wt{\sR}_{\omega}^{\prime ss}}$ descends to an ample line bundle $\Theta_{\sP_{\omega}}=\phi^*(\Theta_{\sU_X})$ on $\sP_{\omega}$.
Similarly, $\phi^*: H^1(\sU_X, \Theta_{\sU_X})\hookrightarrow H^1(\sP_{\omega}, \Theta_{\sP_{\omega}})$ is injective. To prove
$$H^1(\sP_{\omega}, \Theta_{\sP_{\omega}})=0,$$
we need as before to compute canonical bundle $\omega_{\sP_{\omega}}$ and to estimate the codimension of non-semistable points. However, the situation
is slightly different with the case when $\wt X$ is connected. We firstly figure out some necessary conditions when $(E,Q)\in \wt{\sR}_{\omega}^{\prime ss}$.

For $(E, E_{x_1}\oplus E_{x_2}\xrightarrow{q}Q\to 0)\in\sH$, $F=(F_1,F_2)\subset E=(E_1,E_2)$, let
$$D_m(F):=r(F)\frac{par\chi_m(E)-r}{r}-(par\chi_m(F)-t)$$
$$D(F):=\left(r_1\frac{par\chi(E_1)}{r}-par\chi(F_1)\right)+\left(r_2\frac{par\chi(E_2)}{r}-par\chi(F_2)\right)$$
where $t={\rm dim}(Q^F)$, $Q^F=q(F_{x_1}\oplus F_{x_2})\subset Q$, $r_i={\rm rk}(F_i)$. Then
\ga{6.2}{\aligned D_m(F)&=D(F)+\frac{(r_1-r_2)}{r}\left(D_m(E_1)-{\rm dim}(Q^{E_1})\right)+
t-r_2\\&=D(F)+\frac{(r_2-r_1)}{r}\left(D_m(E_2)-{\rm dim}(Q^{E_2})\right)+
t-r_1.\endaligned}

\begin{lem}\label{lem6.1} For $(E, Q)\in \wt{\sR}_{\omega}^{\prime ss}$, let $E_j=E_j'\oplus\,_{x_j}\mathbb{C}^{s_j}$ and
$$n_j^{\omega}=\frac{1}{k}\left(r\ell_j+\sum_{x\in I_j}\sum^{l_x}_{i=1}d_i(x)r_i(x)\right)\quad (j=1,\,2).$$
Then, for the fixed $\chi_j:=\chi(E_j)$ ($j=1,\,2$), we have \begin{itemize}
\item [(1)] $n_j^{\omega}\le \chi_j\le n_j^{\omega}+r$ $\,\,(j=1,\,2)$,
\item [(2)] $s_1\le n^{\omega}_2+r-\chi_2$, $\,\,s_2\le n^{\omega}_1+r-\chi_1$,
\item [(3)] let $(E,Q)\in\sH\setminus\{\sD_1^f\cup\sD^f_2\}$ with $n_j^{\omega}\le \chi(E_j)\le n_j^{\omega}+r$, then
$$E_1\in \sR_1^{ss},\,\,E_2\in\sR_2^{ss}\,\,\,\Rightarrow\,\,\,(E,Q)\in  \wt{\sR}_{\omega}^{\prime\, ss}.$$
Moreover, when $n^{\omega}_1<\chi_1<n^{\omega}_1+r$, we have
$(E,Q)\in \wt{\sR}_{\omega}^{\prime s}$ if one of $E_1$, $E_2$ is a stable parabolic bundle,
\item [(4)] let $(E,Q)\in\sH\setminus\{\sD_1^f\cup\sD^f_2\}$, if $\chi_1=n^{\omega}_1+r$ or $\chi_1=n^{\omega}_1$, then
$$(E,Q)\in \wt{\sR}_{\omega}^{\prime\, ss}\,\,\,\Rightarrow\,\,\,E_1\in\sR_1^{ss},\,\,E_2\in\sR_2^{ss}.$$
\end{itemize}
\end{lem}

\begin{proof} Note that $\chi_1+\chi_2=\chi+r$ and $n_1^{\omega}+n_2^{\omega}=\chi$,
(1) and (2) are clear by the following formulas ($j=1,\,2$)
$$\chi(E_j)=n_j^{\omega}+{\rm dim}(Q^{E_j})-D_m(E_j)$$
$$\chi(E_1)+s_2=n_1^{\omega}+{\rm dim}(Q^{E^s_1})-D_m(E^s_1)$$
$$\chi(E_2)+s_1=n_2^{\omega}+{\rm dim}(Q^{E^s_2})-D_m(E^s_2)$$
where $E^s_1=(E_1, \,_{x_2}\mathbb{C}^{s_2})$, $E^s_2=(\,_{x_1}\mathbb{C}^{s_1},\,E_2)$. The formula \eqref{6.2} becomes
\ga{6.3}{\aligned D_m(F)&=D(F)+\frac{r_2-r_1}{r}(\chi_1-n_1^{\omega})+
{\rm dim}(Q^F)-r_2\\&=D(F)+\frac{r_1-r_2}{r}(\chi_2-n_2^{\omega})+
{\rm dim}(Q^F)-r_1.\endaligned}

To prove (3), by \eqref{6.3} and ${\rm dim}(Q^F)-r_j\ge 0$ ($j=1,\,2$), we have $D_m(F)\ge 0$ whenever $D(F)\ge 0$. Thus
$$E_1\in \sR_1^{ss},\,\,E_2\in\sR_2^{ss}\,\,\,\Rightarrow\,\,\,(E,Q)\in  \wt{\sR}_{\omega}^{\prime\, ss}.$$
When $n^{\omega}_1<\chi_1<n^{\omega}_1+r$ (which implies $n^{\omega}_2<\chi_2<n^{\omega}_2+r$), we have
$D_m(F)>D(F)\ge 0$ if $r_1\neq r_2$. Thus $(E,Q)\in \wt{\sR}_{\omega}^{\prime s}$ if one of $E_1$, $E_2$ is a stable parabolic bundle.

To prove (4), if $\chi_1=n^{\omega}_1+r$ or $\chi_1=n^{\omega}_1$, the formula \eqref{6.3} becomes
\ga{6.4}{D_m(F)=D(F)+{\rm dim}(Q^F)-r_1.}
For $F_1\subset E_1$ of rank $r_1$, take $F=(F_1,0)\subset E$ in \eqref{6.4}, we have
$$D_m(F)=D(F)=r_1\frac{par\chi(E_1)}{r}-par\chi(F_1)$$
which implies that $E_1\in\sR_1^{ss}$ if $(E,Q)\in \wt{\sR}_{\omega}^{\prime\, ss}$. For $F_2\subset E_2$ of rank $r_2$, take
$F=(E_1,F_2)\subset E$ in \eqref{6.4}, we have
$$D_m(F)=D(F)=r_2\frac{par\chi(E_2)}{r}-par\chi(F_2)$$
which implies that $E_2\in\sR_2^{ss}$ if $(E,Q)\in \wt{\sR}_{\omega}^{\prime\, ss}$.
\end{proof}

\begin{nota}\label{nota6.2} For $\omega=(r,\chi_1,\chi_2,\{\vec n(x),\,\vec a(x)\}_{x\in I},\mathcal{O}(1),k)$, let
$$\sH^{\omega}=\left\{\aligned&\text{$(E,Q)\in\sH$, with
$n_j^{\omega}\le \chi(E_j)=\chi_j\le n_j^{\omega}+r$ ($j=1,\,2$), and}\\&\text{${\rm dim}({\rm Tor}(E_1))\le n^{\omega}_2+r-\chi_2$, $\,\,{\rm dim}({\rm Tor}(E_2))\le n^{\omega}_1+r-\chi_1$}\endaligned\right\}.$$
\end{nota}

\begin{prop}\label{prop6.3} Let $\sD_1^f=\hat\sD_1\cup\hat\sD_1^t$
and $\sD_2^f=\hat\sD_2\cup\hat\sD_2^t$. Then \begin{itemize}
\item [(1)] ${\rm codim}(\sH^{\omega}\setminus\wt\sR_{\omega}^{\prime ss})>\underset{1\le i\le 2}{{\rm min}}\left\{(r-1)(g_i-\frac{r+3}{4})+\frac{|I_i|}{k}\right\}.$
\item [(2)] ${\rm codim}(\wt\sR_{\omega}^{\prime ss}\setminus\{\sD_1^f\cup\sD^f_2\}\setminus\wt\sR_{\omega}^{\prime s})>\underset{1\le i\le 2}{{\rm min}}\left\{(r-1)(g_i-1)+\frac{|I_i|}{k}\right\}$ when $n_1^{\omega}<\chi_1<n_1^{\omega}+r$.
\item [(3)] ${\rm codim}(\wt\sR_{\omega}^{\prime ss}\setminus\{\sD_1^f\cup\sD^f_2\}\setminus\wt\sR_{\omega}^{\prime -s})\ge\underset{1\le i\le 2}{{\rm min}}\left\{(r-1)(g_i-1)+\frac{|I_i|}{k}\right\}$
when $\chi_1=n_1^{\omega}$ or $n_1^{\omega}+r$, where
$$\wt\sR_{\omega}^{\prime -s}:=\left\{\aligned
&(E,Q)\in\wt\sR_{\omega}^{\prime ss} \,\text{satisfies $\,par\mu(F)<par\mu(E)$ for any }\\
&\text{nontrivial $F\subset E$ of rank $(r_1,r_2)\neq (0,r)$ or $(r,0)$}\endaligned\right\}.$$
\end{itemize}
\end{prop}

\begin{proof} To prove (1),
let $(E,Q)\in \sH^{\omega}\setminus\wt{R}_{\omega}^{\prime ss}$ with $E=(E_1,E_2)$, then there exists a $F=(F_1,F_2)\subset E$ such that $E/F$ is torsion free and
\ga{6.5}{par\chi_m(F)-dim(Q^F)>r(F)\frac{par\chi_m(E)-r}{r}.}
Let $t=dim(Q^F)$, $r_i=rk(F_i)$, $m_i(x)=dim\frac{F_x\cap F_{i-1}(E)_x}{F_x\cap F_i(E)_x}$, $\chi_i=\chi(E_i)$
$$m(F)=\frac{r(F)-r_1}{k}\sum_{x\in I_1}a_{l_x+1}(x)+\frac{r(F)-r_2}{k}\sum_{x\in I_2}a_{l_x+1}(x)$$
where $r(F)=\frac{c_1r_1+c_2r_2}{c_1+c_2}$. Then we can rewrite \eqref{6.5} as
\ga{6.6}{\aligned r\chi(F)-r(F)\chi>&rt-rm(F)+\frac{r(F)}{k}\sum_{x\in I}\sum^{l_x+1}_{i=1}a_i(x)n_i(x)\\
&-\frac{r}{k}\sum_{x\in I}\sum^{l_x+1}_{i=1}a_i(x)m_i(x)\endaligned}
$$0\to F\to E\to E/F:=\wt F=(\wt F_1,\wt F_2)\to 0$$
Write $E=E'\oplus\,_{x_1}\mathbb{C}^{s_1}\oplus\, _{x_2}\mathbb{C}^{s_2}$, $F=F'\oplus\,_{x_1}\mathbb{C}^{s_1}\oplus\, _{x_2}\mathbb{C}^{s_2}$ and
$F_1=F_1'\oplus\,_{x_1}\mathbb{C}^{s_1}$, $F_2=F_2'\oplus\, _{x_2}\mathbb{C}^{s_2}$ where $E'$, $F'$ (thus $F_1'$, $F_2'$) are torsion free sheaves
satisfying the exact sequences
$$0\to F'_1\to E_1'\to\wt F_1\to 0, \quad 0\to F'_2\to E_2'\to\wt F_2\to 0.$$
Let $d_i={\rm deg}(F'_i)$, $r_i={\rm rk}(F'_i)$, ${\rm deg}(\wt F_i)=\chi_i-r(1-g_i)-d_i-s_i$ and
$$P_i(m)=c_ir_im+d_i+r_i(1-g_i), \quad \wt P_i(m)=c_irm+\chi_i-s_i-P_i(m).$$
For $\sW_i=\sO_{X_i}(-N)$, $V_i=\mathbb{C}^{P_i(N)}$ (resp. $\wt V_i=\mathbb{C}^{\wt P_i(N)}$), let $$Q_i\subset Quot(V_i\otimes\sW_i,P_i)$$
(resp. $\wt Q_i\subset Quot(\wt V_i\otimes\sW_i,\wt P_i)$) be the open set of locally free quotients $F_i'$ (resp. $\wt F_i$) with
vanishing $H^1(F'_i(N))$ (resp. $H^1(\wt F_i(N))$) and $F'_i(N)$ (resp. $\wt F_i(N)$) generated by global sections.
Let $\sF'_i$ (resp. $\wt F_i$) be the universal quotient on $X_i\times Q_i$ (resp. on $X_i\times \wt Q_i$), let $\sV_i=Q_i\times \wt Q_i$ and
$\sG_i={\wt F_i}^{\surd}\otimes\sF_i'$ on $X_i\times\sV_i$. Then we have $$\sV_i=\bigcup_{h_i\ge 0}\sV_i^{h_i}$$ such that $R^1f_{i*}(\sG_i)$ is locally free of
rank $h_i$ on $\sV_i^{h_i}$ where $f_i:X_i\times\sV_i\to \sV_i$ is the projection. Let $P_{h_i}=\mathbb{P}(R^1f_{i*}(\sG_i)^{\surd})\to \sV_i^{h_i}$
be the projective bundle on $\sV_i$ and $0\to\sF_i'\otimes\sO_{P_{h_i}}(-1)\to\sE'_i(h_i)\to\wt\sF_i\to 0$ be the universal extension on $X_i\times P_{h_i}$ (we set $P_{h_i}=\sV_i$ and
$\sE'_i(h_i)=\sF_i'\oplus\wt\sF_i$ if $h_i=0$). For $v_i'=(d_i,r_i,\{m_1(x),\cdots,m_{l_x+1}(x)\}_{x\in I_i},h_i)$, we can define a
variety $X(v_i')\to P_{h_i}$.
It parametrises a family of parabolic bundles $E_i'$, which occur as extensions
$0\to F'_i\to E_i'\to \wt F_i\to 0$ (the extension being split if $h_i=0$), with
parabolic structures at $x\in I_i$ of type
$\vec n(x)=n_1(x),\cdots,n_{l_x+1}(x))$, whose induced parabolic structures
on $F'_i$ are of type $(m_1(x),\cdots,m_{l_x+1}(x))$ (we will forget $m_j(x)$ if it
is zero). Let $0\to\sF'_i(-1)\to\sE'(v_i')\to\wt\sF_i\to 0$ be the pull back of
universal extension to $X_i\times X(v_i')$, $\sE(v_i')=
\sE'(v_i')\oplus\,_{x_i}\sO^{s_i}$ and let $F(v_i')$ be the frame bundle of the direct image of $\sE(v_i')(N)$ (under the projection $X_i\times X(v_i')\to X(v_i')$).
Write $\sE(v'):=\sE(v_1')\oplus \sE(v_2')$, we consider
$$G_{v'}:=Grass_r(\sE(v')_{x_1}\oplus\sE(v')_{x_2})\to F(v_1')\times F(v_2')$$ and define a
subvariety of $G_{v'}$ by
$$X(v):=\left\{\aligned
&(E_{x_1}\oplus E_{x_2}\xrightarrow{q}Q\to0)\in G_{v'}),\,\text{
$ker(q)\cap(\Bbb C^{s_1}\oplus\Bbb C^{s_2})=0$,}\\
&\text{$dim(ker(q)
\cap(F'_{x_1}\oplus\Bbb C^{s_1}\oplus F'_{x_2}\oplus\Bbb C^{s_2}))=r_1+r_2
+s-t $}\endaligned\right\}.$$
Then $X(v)$ parametrises a family of GPS $(E=E'\oplus\,_{x_1}\Bbb C^{s_1}\,
\oplus\,_{x_2}\Bbb C^{s_2},\,Q)$, where $E'=(E'_1, E'_2)$ occurs as extensions
$0\to F'_i\to E_i'\to \wt F_i\to 0$ (it is split if $h_i=0$)
with parabolic structures at $x\in I$ of type $\vec n(x)$,
whose induced parabolic structures on $F'_i$ are of type
$(m_1(x),\cdots,m_{l_x+1}(x))$ (we will forget $m_i(x)$ if it is zero),
such that $_{x_1}\Bbb C^{s_1}\,\oplus\,_{x_2}
\Bbb C^{s_2}\to Q$ is injective and
$\text{rank$(F'_{x_1}\oplus\Bbb C^{s_1}\oplus F'_{x_2}\oplus\Bbb C^{s_2}\to Q)=t$}.$
There is a morphism $X(v)\to\sH^{\omega}\setminus\wt\sR_{\omega}^{ss}$ whose image contains $(E,Q)$.
Therefore we have a (countable) number of quasi-projective
varieties $X(v)$ and morphisms $\varphi_v:X(v)\to\sH^{\omega}\setminus\wt\sR_{\omega}^{ss}$ such that
the union of the images covers $\sH^{\omega}\setminus\wt\sR_{\omega}^{ss}$.

One computes ${\rm dim}\,F(v_i')={\rm dim}\,X(v_i')+(c_irN+\chi_i)^2$,
$${\rm dim}\,X(v_i')=\left\{
\begin{array}{llll}\sum_{x\in I_i}{\rm dim}\,X_{v_i(x)}+h_i-1+{\rm dim}\,Q_i+{\rm dim}\,\wt Q_i,  &\mbox{if $h_i\neq 0$}\\
\sum_{x\in I_i}{\rm dim}\,X_{v_i(x)}+{\rm dim}\,Q_i+{\rm dim}\,\wt Q_i &\mbox{if
$h_i=0$}\end{array}\right.$$
${\rm dim}\,Q_i+{\rm dim}\,\wt Q_i=(g_i-1)(r_i^2+(r-r_i)^2)+P_i(N)^2+\wt P_i(N)^2$ and the dimension of $\sH$, $\,X(v)$
are (let $s=s_1+s_2$):
$$r^2(g-2)+r^2+\sum_{i=1}^2(c_irN+\chi_i)^2+\sum_{x\in I}dim\,
Flag_{\vec n(x)}(\sF_x),$$
$$r(r+s)-(r-t)(r_1+r_2+s-t)+\sum_{i=1}^2(c_irN+\chi_i)^2+\sum_{i=1}^2{\rm dim}\,X(v_i').$$

To estimate the minimum $e$ of fiber dimension of $\varphi_v$,  note that
$${\rm dim}\,Aut(E)\ge {\rm dim}\,Aut(E_1')+{\rm dim}\,Aut(E'_2)+rs+s_1^2+s_2^2$$
and $0\to F'_i\to E'_i\to\wt F_i\to 0$, we have
$${\rm dim}\,Aut(E_i')\ge\left\{
\begin{array}{llll} 1+h^0({\wt F_i}^\surd\otimes F'_i),  &\mbox{if $h_i\neq 0$}\\
2+h^0({\wt F_i}^\surd\otimes F'_i) &\mbox{if
$h_i=0$}\end{array}\right.$$
Define $e(h_i)=1$ when $h_i\neq 0$ and $e(h_i)=2$ when $h_i=0$, then
$$\aligned e\ge& rs+s_1^2+s_2^2+h^0({\wt F_1}^\surd\otimes F'_1)+h^0({\wt F_2}^\surd\otimes F'_2)+e(h_1)\\
&+e(h_2)-4+P_1(N)^2+\wt P_1(N)^2+P_2(N)^2+\wt P_2(N)^2.\endaligned$$
Then the codimension of $\sH^{\omega}\setminus\wt\sR_{\omega}^{ss}$ is bounded below by
$$\aligned & \sum^2_{i=1}r_i(r-r_i)(g_i-1)+\sum^2_{i=1}(r_i+s_i-t)s_i+(r-t)(r_1+r_2-t)+\\
& r\chi(F)-(r_1\chi_1+r_2\chi_2)+\sum_{x\in I_1} \sum^{l_x+1}_{j=1}(r_1-\sum^j_{i=1}m_i(x))(n_j(x)-m_j(x))\\&+\sum_{x\in I_2} \sum^{l_x+1}_{j=1}(r_2-\sum^j_{i=1}m_i(x))(n_j(x)-m_j(x)).\endaligned$$
If $r_1\ge r_2$, use $\chi_1+s_2\le n^{\omega}_1+r$ and $\chi_2=\chi+r-\chi_1$ to get
\ga{6.7}{\aligned r\chi(F)-&(r_1\chi_1+r_2\chi_2)\ge r\chi(F)-r(F)\chi+rm(F)-\\&r_1r+(r_1-r_2)s_2+\frac{r_1-r(F)}{k}\sum_{x\in I_1}\sum^{l_x+1}_{i=1}a_i(x)n_i(x)
\\&+\frac{r_2-r(F)}{k}\sum_{x\in I_2}\sum^{l_x+1}_{i=1}a_i(x)n_i(x).\endaligned }
Similarly, if $r_2\ge r_1$, we have
\ga{6.8}{\aligned r\chi(F)-&(r_1\chi_1+r_2\chi_2)\ge r\chi(F)-r(F)\chi+rm(F)-\\&r_2r+(r_2-r_1)s_1+\frac{r_1-r(F)}{k}\sum_{x\in I_1}\sum^{l_x+1}_{i=1}a_i(x)n_i(x)
\\&+\frac{r_2-r(F)}{k}\sum_{x\in I_2}\sum^{l_x+1}_{i=1}a_i(x)n_i(x).\endaligned }

By using of the inequalities \eqref{6.6}, \eqref{6.7} and \eqref{6.8}, we have
$$\aligned codim(\sH^{\omega}\setminus\wt\sR_{\omega}^{\prime ss})>&\sum^2_{i=1}r_i(r-r_i)(g_i-1)+({\rm max}\{r_1,r_2\}-t)s\\&+s_1^2+s_2^2+r\cdot{\rm min}\{r_1,r_2\}-t(r_1+r_2-t)\\
&+\sum_{x\in I_1}
\left\{\aligned &\sum^{l_x+1}_{j=1}(r_1-\sum^j_{i=1}m_i(x))(n_j(x)-m_j(x))\\
&+\sum^{l_x+1}_{j=1}(r_1n_j(x)-rm_j(x))\frac{a_j(x)}{k}\endaligned\right\}\\&+\sum_{x\in I_2}
\left\{\aligned &\sum^{l_x+1}_{j=1}(r_2-\sum^j_{i=1}m_i(x))(n_j(x)-m_j(x))\\
&+\sum^{l_x+1}_{j=1}(r_1n_j(x)-rm_j(x))\frac{a_j(x)}{k}\endaligned\right\},
\endaligned$$
where $s=s_1+s_2$. Let $f(r_1,r_2,s_1,s_2,t)$ denote
$$\aligned&({\rm max}\{r_1,r_2\}-t)s+s_1^2+s_2^2+r\cdot{\rm min}\{r_1,r_2\}-t(r_1+r_2-t)=\\&(t-\frac{r_1+r_2+s}{2})^2+\frac{2(s_1^2+s_2^2)+(s_1-s_2)^2}{4}+
\frac{{\rm max}\{r_1,r_2\}-{\rm min}\{r_1,r_2\}}{2}s\\&+{\rm min}\{r_1,r_2\}(r-{\rm max}\{r_1,r_2\})-\frac{(r_1-r_2)^2}{4}.
\endaligned$$
When $r_1=r_2$, it is clear that $f(r_1,r_2,s_1,s_2,t)\ge r_1(r-r_1)$ and we have
$$codim(\sH^{\omega}\setminus\wt\sR_{\omega}^{\prime ss})>r_1(r-r_1)(g-1)+\frac{|I|}{k}.$$
In general, we have only $f(r_1,r_2,s_1,s_2,t)\ge -\frac{(r-1)^2}{4}$ and
$${\rm codim}(\sH^{\omega}\setminus\wt\sR_{\omega}^{\prime ss})>\underset{1\le i\le 2}{{\rm min}}\left\{(r-1)(g_i-\frac{r+3}{4})+\frac{|I_i|}{k}\right\}.$$

To prove (2), note $s_1=s_2=0$, ${\rm max}\{r_1,r_2\}\le t$ for $(E,Q)\in \wt\sR_{\omega}^{\prime ss}\setminus\{\sD_1^f\cup\sD^f_2\}$, we have
$f(r_1,r_2,s_1,s_2,t)=r\cdot{\rm min}\{r_1,r_2\}+t(t-r_1-r_2)\ge 0.$
Then, when $n_1^{\omega}<\chi_1<n_1^{\omega}+r$, which implies $(r_1,r_2)\neq (r,0),\,(0,r)$,
$${\rm codim}(\wt\sR_{\omega}^{\prime ss}\setminus\{\sD_1^f\cup\sD^f_2\}\setminus\wt\sR_{\omega}^{\prime s})>\underset{1\le i\le 2}{{\rm min}}\left\{(r-1)(g_i-1)+\frac{|I_i|}{k}\right\}.$$

The assertion (3) follows the same arguments of (2) and the definition of $\wt\sR_{\omega}^{\prime -s}$. In fact, $\wt\sR_{\omega}^{\prime -s}=\rho^{-1}(\sR^s_1\times\sR^s_2)$ by Lemma \ref{lem6.1} (4), where
$$\rho:\wt\sR_{\omega}^{\prime ss}\setminus\{\sD_1^f\cup\sD^f_2\}\to \sR_1^{ss}\times\sR^{ss}_2.$$
\end{proof}

The schemes $\sH$ and $\sP$ are Gorenstein, so they have canonical sheaves. To compute the canonical sheaves $\omega_{\sH}$ and $\omega_{\sP}$, let
\ga{6.9}{0\to\sK^j\to V_j\otimes\sO_{X_j\times\sR_j}(-N)\to\sE^j\to 0\quad(j=1,\,2)}
be the universal quotient on $X_j\times\sR_j$ ($\sK^j$ are in fact locally free), and
$$\aligned\omega^{-1}_{\sR_j}=&(det\,R\pi_{\sR_j}\sE^j)^{-2r}\otimes
\bigotimes_{x\in I_j}\left\{(det\,\sE^j_x)^{n_{l_x+1}(x)-r}
\otimes\bigotimes^{l_x}_{i=1}(det\,\sQ_{x,i})
^{n_i(x)+n_{i+1}(x)}\right\}\\
&\otimes\bigotimes_{q\in X_j}(det\,\sE^j_q)^{1-r}\otimes(det\,R\pi_{\sR_j}det\sE^j)^2
\endaligned$$
where $\omega_{X_j}=\sO_{X_j}(\sum_{q\in X_j}q)$. Let $\hat{\rm Det}_j:\sR_j\to J^{d_j}_{X_j}$, where $d_j=\chi_j+r(g_j-1)$, be defined by
${\rm det}\sE^j:=({\rm det}\sK^j)^{-1}\otimes\sO_{X_j\times\sR_j}(-P_j(N)N),$
let $\sL_j$ be a universal line bundle on $X_j\times J^{d_j}_{X_j}$ and
\ga{6.10}{\Theta_{J^{d_j}_{X_j}}=({\rm det}R\pi_{J^{d_j}_{X_j}}\sL_j)^{-2}\otimes(\sL_j)_{x_j}^{r}\otimes\bigotimes_{q\in X_j}(\sL_j)_q^{r-1}\otimes(\sL_j)_{y_j}^{2\chi_j-r}}
(which are independent of the choices of $\sL_j$). Let
$$\hat{\rm Det}_{\wt\sR}:=(\hat{\rm Det}_1,\hat{\rm Det}_2):\wt\sR=\sR_1\times\sR_2\to J^d_{\wt X}:=J^{d_1}_{X_1}\times J^{d_2}_{X_2},$$
which induces $\hat{\rm Det}_{\sH}:\sH\to J^d_{\wt X}$ and ${\rm Det}:\sP_{\omega}\to J^d_{\wt X}$ such that
$$\xymatrix{
  \sH \ar[d]_{\rho} \ar[dr]^{\hat{\rm Det}_{\sH}}   \\
  \wt\sR \ar[r]_{\hat{\rm Det}_{\wt\sR}}  & J^d_{\wt X}}\qquad \qquad\xymatrix{
  {\wt\sR}^{\prime ss}_{\omega} \ar[d]_{\psi} \ar[dr]^{\hat{\rm Det}_{{\wt\sR}^{\prime ss}_{\omega}}}        \\
  \sP_{\omega} \ar[r]_{{\rm Det}}  & J^d_{\wt X}      }$$
are commutative. Let $\Theta_{J^d_{\wt X}}=p_1^*\Theta_{J^{d_1}_{X_1}}\otimes p_2^*\Theta_{J^{d_2}_{X_2}}$ (where $p_j: J^d_{\wt X}:=J^{d_1}_{X_1}\times J^{d_2}_{X_2}\to J^{d_j}_{X_j}$ are projections). Then similar arguments of \cite{S1} give

\begin{prop}\label{prop6.4} Let $\rho:\sH\to \wt{\sR}:=\sR_1\times\sR_2$ and
$\sE^1_{x_1}\oplus\sE^2_{x_2}\to \sQ\to 0$ be the universal quotient on $\sH$. Then
$$\aligned\omega_{\sH}^{-1}&=\rho^*(\omega^{-1}_{\sR_1}\otimes\omega^{-1}_{\sR_2})\otimes(det\sQ)^{2r}\otimes(det\sK^1_{x_1})^{r}\otimes(det\sK^2_{x_2})^r=\\
&(detR\pi_{\sH}\sE)^{-2r}\otimes\bigotimes_{x\in I}\left\{(det\,\sE_x)^{-r_{l_x}(x)}
\otimes\bigotimes^{l_x}_{i=1}(det\,\sQ_{x,i})
^{n_i(x)+n_{i+1}(x)}\right\}\\& \otimes(det\sQ)^{2r}\otimes\bigotimes^2_{j=1}(det\sE_{y_j})^{2\chi_j-r}\otimes{\hat{\rm Det}}^*_{\sH}(\Theta_{J^d_{\wt X}}^{-1})={\hat\Theta}'_{\omega^c}\otimes {\hat{\rm Det}}^*_{\sH}(\Theta_{J^d_{\wt X}}^{-1})\endaligned $$
where $$\aligned {\hat\Theta}'_{\omega^c}=&(det\,R\pi_{\sH}\sE)^{-2r}\otimes
\bigotimes_{x\in I}\left\{\bigotimes^{l_x}_{i=1}(det\,\sQ_{x,i})
^{n_i(x)+n_{i+1}(x)}\right\}\otimes det(\sQ)^{2r}\\&\otimes({\rm det}\sE_{y_1})^{2\chi_1-r}\otimes({\rm det}\sE_{y_2})^{2\chi_2-r}\otimes\bigotimes_{x\in I} (det\,\sE_x)^{-r_{l_x}(x)}.
\endaligned$$
\end{prop}

Let $J_i\subset X_i\setminus (I_i\cup\{x_i\})$ be a subset, $J=J_1\cup J_2$ and
$$\sR(J)_i=\underset{x\in I_i\cup J_i}{\times_{\wt{\textbf{Q}}_i}}
Flag_{\vec n(x)}(\sF^i_x)\to \wt{\textbf{Q}}_i, $$
$\wt\sR(J)=\sR(J)_1\times\sR(J)_2\xrightarrow{p_J}\wt\sR=\sR_1\times\sR_2$ be the projection. Consider
$$\CD
  {\wt\sR}(J)' @>p_J>> {\wt\sR}' \\
  @V \rho VV @V \rho VV  \\
  \wt\sR(J) @>p_J>> \wt\sR
\endCD
$$ and $\sH(J):=p_J^{-1}(\sH)\xrightarrow{p_J}\sH$. Then, by Proposition \ref{prop6.4}, we have
\ga{6.11}{\omega^{-1}_{\sH(J)}={\hat\Theta}'_{\omega^c(J)}\otimes {\hat{\rm Det}}^*_{\sH(J)}(\Theta_{J^d_{\wt X}}^{-1}),}
where
$$\aligned {\hat\Theta}'_{\omega^c(J)}=&(det\,R\pi_{\sH(J)}\sE)^{-2r}\otimes
\bigotimes_{x\in I\cup J}\left\{\bigotimes^{l_x}_{i=1}(det\,\sQ_{x,i})
^{n_i(x)+n_{i+1}(x)}\right\}\otimes det(\sQ)^{2r}\\&\otimes({\rm det}\sE_{y_1})^{2\chi_1-r}\otimes({\rm det}\sE_{y_2})^{2\chi_2-r}\otimes\bigotimes_{x\in I\cup J} (det\,\sE_x)^{-r_{l_x}(x)}.
\endaligned$$
Let $\omega^c(J)=(r,\chi_1,\chi_2,\{\{n_i(x)\}_{1\le i\le l_x+1},\,\{d^c_i(x)\}_{1\le i\le l_x}\}_{x\in I\cup J},\mathcal{O}(1),k^c)$ where $k^c=2r$, $d^c_i(x)=n_i(x)+n_{i+1}(x)$, let $\ell_j^c=2\chi_j-r-\sum_{x\in I_j\cup J_j}r_{l_x}(x)$ and $\ell^c=\ell^c_1+\ell^c_2=2\chi-\sum_{x\in I\cup J}r_{l_x}(x).$
Then $$\sum_{x\in I\cup J}\sum^{l_x}_{i=1} d^c_i(x)r_i(x) +r\ell^c= k^c\chi.$$
The type $\{\vec n(x)\}_{x\in J}$ of flags at $x\in J$ will be chosen to satisfy
\ga{6.12}{\ell^c_1=\frac{c_1}{c_1+c_2}\ell^c}
which is equivalent to the following condition
\ga{6.13}{\aligned &c_1\sum_{x\in J_2}r_{l_x}(x)-c_2\sum_{x\in J_1}r_{l_x}(x)=\\&c_1\left(2\chi_2-r-\sum_{x\in I_2}r_{l_x}(x)\right)-c_2\left(2\chi_1-r-\sum_{x\in I_1}r_{l_x}(x)\right).\endaligned}

The choices of $\{\vec n(x)\}_{x\in J}$ satisfying \eqref{6.12} for arbitrary large $|J_1|$ and $|J_2|$ are possible since
the equation \eqref{6.13} has arbitrary large integer solutions. In this case, the line bundle ${\hat\Theta}'_{\omega^c(J)}$ is (algebraically) equivalent
to the restriction (on $\sH(J)$) of the following polarization
$$\frac{\ell^c+k^ccN}{c(m-N)}\times\prod_{x\in
I\cap J} \{d^c_1(x),\cdots,d^c_{l_x}(x)\}\times k^c.$$
On the other hand, it is easy to compute that $n_j^{\omega^c(J)}=\chi_j-\frac{r}{2}$, thus
$$n_j^{\omega^c(J)}<\chi_j<n_j^{\omega^c(J)}+r\qquad (j=1,\,\,2).$$
Moreover, for any polarization \eqref{6.1} (determined by $\omega$), let ${\hat\Theta}'_{\sH}$ be its restriction to $\sH$. Then we can write
$$p^*_J({\hat\Theta}'_{\sH})=\omega_{\sH(J)}\otimes{\hat\Theta}'_{\bar\omega}\otimes{\hat{\rm Det}}^*_{\sH(J)}(\Theta_{J^d_{\wt X}}^{-1}),$$
where $\bar\omega=(r,\chi_1,\chi_2,\{\{n_i(x)\}_{1\le i\le l_x+1},\,\{\bar d_i(x)\}_{1\le i\le l_x}\}_{x\in I\cup J},\mathcal{O}(1),\bar k)$,
$$\aligned {\hat\Theta}'_{\bar\omega}=&(det\,R\pi_{\sH(J)}\sE)^{-\bar k}\otimes
\bigotimes_{x\in I\cup J}\left\{\bigotimes^{l_x}_{i=1}(det\,\sQ_{x,i})
^{\bar d_i(x)}\right\}\otimes det(\sQ)^{\bar k}\\&\otimes({\rm det}\sE_{y_1})^{\ell_1+2\chi_1-r}\otimes({\rm det}\sE_{y_2})^{\ell_2+2\chi_2-r}\otimes\bigotimes_{x\in I\cup J} (det\,\sE_x)^{-r_{l_x}(x)}
,\endaligned$$ $\bar k=k+2r$, $\bar d_i(x)=d_i(x)+n_i(x)+n_{i+1}(x)$ ($d_i(x)=0$ for $x\in J$). Let
$$\bar\ell_j=\ell_j+2\chi_j-r-\sum_{x\in I_j\cup J_j}r_{l_x}(x)=\ell_j+\ell_j^c,$$
$$\bar\ell:=\bar\ell_1+\bar\ell_2=\ell+2\chi-\sum_{x\in I\cup J}r_{l_x}(x)=\ell+\ell^c.$$ Then it is easy to see that $\bar\ell_j=\frac{c_j}{c_1+c_2}\bar\ell$
(by \eqref{6.12}),
$$\sum_{x\in I\cup J}\sum^{l_x}_{i=1}\bar d_i(x)r_i(x) +r\bar\ell=\bar k\chi,$$
and ${\hat\Theta}'_{\bar\omega}$ is (algebraically) equivalent to the restriction of polarization determined by $\bar\omega$.
The condition \eqref{6.12} implies the following identities
\ga{6.14}{2r(\chi_j-n_j^{\bar\omega})=r^2+k(n_j^{\bar\omega}-n_j^{\omega})\quad (j=1,2).}

\begin{lem}\label{lem6.5} For any $(E,Q)\in \sH(J)$, we have $n_j^{\bar\omega}\le\chi_j\le n_j^{\bar\omega}+r$
(which is the necessary condition that $\wt{\sR}(J)_{\bar\omega}^{\prime ss}\neq \emptyset$).
\end{lem}

\begin{proof} If $n_1^{\bar\omega}\ge n_1^{\omega}$, by \eqref{6.14}, we have
$n_1^{\bar\omega}<\chi_1\le n_1^{\omega}+r\le n_1^{\bar\omega}+r$, which implies
$n_2^{\bar\omega}\le\chi_2 <n_2^{\bar\omega}+r$. If $n_1^{\bar\omega}<n_1^{\omega}$, by $n_1^{\bar\omega}+n_2^{\bar\omega}=\chi=n_1^{\omega}+n_2^{\omega}$, we have
$n_2^{\bar\omega}>n_2^{\omega}$ which implies $n_2^{\bar\omega}<\chi_2 \le n_2^{\omega}+r <n_2^{\bar\omega}+r$ by \eqref{6.14} again (thus $n_1^{\bar\omega}<\chi_1< n_1^{\bar\omega}+r$).
\end{proof}

To prove $H^1(\sP_{\omega},\Theta_{\sP_{\omega}})=0$ via the same method of Section 5, even if we assume that $\underset{1\le i\le 2}{{\rm min}}\left\{(r-1)(g_i-\frac{r+3}{4})+\frac{|I_i|}{k}\right\}\ge 3$, we only have
$$\aligned H^1(\sP_{\omega},\Theta_{\sP_{\omega}}):&=H^1(\wt{\sR}_{\omega}^{\prime ss},{\hat\Theta}'_{\wt{\sR}_{\omega}^{\prime ss}})^{inv.}=H^1(\sH^{\omega},{\hat\Theta}'_{\sH})^{inv.}\\&=H^1(p_J^{-1}(\sH^{\omega}),p^*_J({\hat\Theta}'_{\sH}))^{inv.}\\
&=H^1(p_J^{-1}(\sH^{\omega}),\omega_{\sH(J)}\otimes{\hat\Theta}'_{\bar\omega}\otimes{\hat{\rm Det}}^*_{\sH(J)}(\Theta_{J^d_{\wt X}}^{-1}))^{inv.}.\endaligned$$
If $p_J^{-1}(\sH^{\omega})=\sH(J)^{\bar\omega}$, we would have (choosing $|J_1|$, $|J_2|$ large enough)
$$\aligned &H^1(p_J^{-1}(\sH^{\omega}),\omega_{\sH(J)}\otimes{\hat\Theta}'_{\bar\omega}\otimes{\hat{\rm Det}}^*_{\sH(J)}(\Theta_{J^d_{\wt X}}^{-1}))^{inv.}\\&=H^1(\sP_{\bar\omega}, \omega_{\sP_{\bar\omega}}\otimes \Theta_{\sP_{\bar\omega}}\otimes {\rm Det}^*_{\sP_{\bar\omega}}(\Theta_{J^d_{\wt X}}^{-1}))\endaligned$$
which vanishes by Kodaira-type theorem and the following lemma.
\begin{lem}\label{lem6.4} When $X=X_1\cup X_2$ with node $x_0$, the line bundle
$$\Theta_{\sP_{\bar\omega}}\otimes {\rm Det}^*_{\sP_{\bar\omega}}(\Theta_{J^d_{\wt X}}^{-1})$$
on $\sP_{\bar\omega}$ is ample if $\bar
k>2r$.
\end{lem}

\begin{proof} When $X=X_1\cup X_2$, the moduli space
$\sP_{\bar\omega}$ is a disjoint union of
$$\{\sP_{d_1,d_2}\}_{d_1+d_2=d}.$$ It is enough to consider
$\sP_{\bar\omega}=\sP_{d_1,d_2}$, thus we the flat morphism
$${\rm Det}:\sP_{\bar\omega}\to J^d_{\wt X}=J^{d_1}_{X_1}\times
J^{d_2}_{X_2}=J^d_X$$ and $J^0_{\wt X}=J^0_{X_1}\times
J^0_{X_2}=J^0_X$ acts on $\sP_{\bar\omega}$ by
$$((E,Q),\sN)\mapsto (E\otimes\pi^*\sN, Q\otimes \sN_{x_0}).$$
Let $\sP_{\bar\omega}^L={\rm Det}_{\sP_{\bar\omega}}^{-1}(L)$ (which is unirational),
consider the morphism
$$f:\sP_{\bar\omega}^L\times J^0_X\to \sP_{\bar\omega}.$$
Then it is enough to check the ampleness of
$$f^*(\Theta_{\sP_{\bar\omega}}\otimes{\rm
Det}_{\sP_{\bar\omega}}^*(\Theta_{J^d_{\wt X}}^{-1}))|_{\{(E,\,Q)\}\times J^0_X}\,,\quad
f^*(\Theta_{\sP_{\bar\omega}}\otimes{\rm
Det}_{\sP_{\bar\omega}}^*(\Theta_{J^d_{\wt X}}^{-1}))|_{\sP_{\bar\omega}^L\times \{\sN\}}.$$ It is
clearly that $f^*(\Theta_{\sP_{\bar\omega}}\otimes{\rm
Det}_{\sP_{\bar\omega}}^*(\Theta_{J^d_{\wt X}}^{-1}))|_{\sP_{\bar\omega}^L\times \{\sN\}}$ is
ample, and
$$f^*(\Theta_{\sP_{\bar\omega}}\otimes{\rm
Det}_{\sP_{\bar\omega}}^*(\Theta_{J^d_{\wt X}}^{-1}))|_{\{(E,\,Q)\}\times J^0_X}=M_1\otimes M_2$$
where $M_1=f_1^*(\Theta_{\sP_{\bar\omega}})$,
$M_2=f_2^*(\Theta_{J^d_{\wt X}}^{-1})$, $f_1: J^0_X\to \sP_{\bar\omega}$, $ f_2:
J^0_X\to J^d_X$,
$$f_1(\sN)=(E\otimes\pi^*\sN,Q\otimes\sN_{x_0}),\quad
f_2(L_0)=L_0^r\otimes L.$$ Then $M_1$ (resp. $M_2$) is algebraically
equivalent to $\Theta_y^{r\bar k}$ (resp. $\Theta_y^{-2r^2}$) (see
Lemma 5.3 of \cite{S1} for details). Thus $M_1\otimes M_2$ is
algebraically equivalent to $\Theta_y^{r\bar k-2r^2}$, which is ample
when $\bar k>2r$.
\end{proof}

\begin{rmks} (1) The equality $p_J^{-1}(\sH^{\omega})=\sH(J)^{\bar\omega}$ is equivalent to the statement that for any $(E,Q)\in \sH(J)$ with torsion $\tau_i$ at $x_i$ we have
\ga{6.15} {\tau_i\le n_j^{\omega}+r-\chi_j\,\,(j\neq i)\,\,\Leftrightarrow\, \,\tau_i\le  n_j^{\bar\omega}+r-\chi_j\,\,(j\neq i)}
which may not be true unfortunately. (2) The proof of Proposition \ref{prop6.3} in fact implies the following estimate
\ga{6.16}{{\rm codim}(\sH\setminus\wt\sR_{\omega}^{\prime -ss})>\underset{1\le i\le 2}{{\rm min}}\left\{(r-1)(g_i-\frac{r+3}{4})+\frac{|I_i|}{k}\right\}}
where the open set $\wt\sR_{\omega}^{\prime -ss}\subset \sH$ satisfying $\wt\sR_{\omega}^{\prime -ss}\supset\wt\sR_{\omega}^{\prime ss}$ is defined to be
$$\wt\sR_{\omega}^{\prime -ss}:=\left\{\aligned
&(E,Q)\in\sH \,\text{satisfies $\,par\mu(F)\le par\mu(E)$ for any }\\
&\text{nontrivial $F\subset E$ of rank $(r_1,r_2)\neq (0,r)$ or $(r,0)$}\endaligned\right\}.$$
\end{rmks}

We end up by some comments about quantization conjecture of Guillemin-Sternberg.
Let $M$ be a projective variety with an action of a reductive group $G$ and an ample $L$ linearizing the action of $G$. If $ M^{ss}_{L}\subset M$ is the open set of GIT semistable points, then the so called quantization conjecture of Guillemin-Sternberg predict that
\ga{6.17} {H^i(M,L)^{inv.}=H^i(M^{ss}_L,L)^{inv.}}
which was proved when $M$ is projective and has at most rational singularities (see \cite{Te}, \cite{TZ} and \cite{Zh}). There is an example in \cite{Te} showing the failure of \eqref{6.17} when $M$ has worse singularities. However, for the applications of studying moduli spaces in algebraic geometry, $M$ is in general a locally closed subvariety of Quotient schemes or Hilbert schemes (for example, $M=\wt{\sR}_F,\,\,\sH$ in this article, which are quasi-projective and have at most rational singularities). Thus the following question seems natural and important for application.

\begin{question}\label{question6.8} Let $M$ be a normal, projective variety with action by a reductive group $G$. If $M_0\subset M$ is an $G$-invariant open set such that $M^{ss}_{L}\subset M_0$ for any ample linearization $L$. Does the equality $$H^i(M_0,L)^{inv.}=H^i(M^{ss}_{L}, L)^{inv.}$$ holds for any $i\ge 0$ ?
\end{question}

If the question has an affirmative answer, conjecture in Remark \ref{rmk5.5} and Conjecture \ref{conj4.5} will hold, which imply $H^1(\sU_X,\Theta_{\sU_X,\,\omega})=0$ for any irreducible $X$ with one node and any data $\omega$ (see Remark \ref{rmk5.5}). However, the affirmative answer of
Question \ref{question6.8} seems not imply $H^1(\sU_X,\Theta_{\sU_X,\,\omega})=0$ for reducible $X=X_1\cup X_2$.

Let $q_L: M^{ss}_L\to \sM_L:=M^{ss}_L//G$ be the GIT quotient and assume that $L$ descends to a line bundle $\sL$ (i.e. $L$ is the pullback of $\sL$). One of the general strategy of proving $H^i(\sM_L, \sL)=0$ is to use equalities
$$H^i(\sM_L, \sL)=H^i(M^{ss}_L,L)^{inv.}=H^i(M_0,L)^{inv.}$$
where the first equality holds by definition and the second holds by the affirmative answer of
Question \ref{question6.8}. Then one can write (on $M_0$)
$$L=\omega_{M_0}\otimes L',\quad L'=\omega_{M_0}^{-1}\otimes L$$
where $\omega_{M_0}$ is the canonical bundle of $M_0$. Let $q_{L'}: M^{ss}_{L'}\to \sM_{L'}$ be the GIT quotient and
$L'$ descend to $\sL'$. Assume that
\ga{6.18} {H^i(M_0,L)^{inv.}=H^i(M^{ss}_{L'}, L)^{inv.}, \quad \omega_{M_0}=q_{L'}^*(\omega_{\sM_{L'}}).}
Then $H^i(\sM_L,\sL)=H^i(\sM_{L'}, \omega_{\sM_{L'}}\otimes\sL')=0$ ($\forall\,\,i>0$). Assumption \eqref{6.18} does
not hold in general, which need a good estimate of codimension of $M_0\setminus M^{ss}_{L'}$ and $M^{ss}_{L'}\setminus M^{s}_{L'}$.
It is the reason that this strategy does not work for reducible $X=X_1\cup X_2$ since we do not have a good estimate of codimension of
$\sH\setminus \wt\sR_{\omega}^{\prime ss}$ (we have only an estimate of ${\rm codim}(\sH^{\omega}\setminus\wt\sR_{\omega}^{\prime ss})$).
However, we will prove vanishing theorems in a forthcoming article \cite{SZ} for all of these moduli spaces by a method of modulo $p$ reduction, which essentially
needs the estimates of codimension and computation of canonical bundles.

\bibliographystyle{plain}

\renewcommand\refname{References}

\end{document}